\documentclass[11pt,twoside]{article}
\usepackage[utf8]{inputenc}
\usepackage{mathrsfs}
\usepackage{indentfirst}
\usepackage{bbold}
\usepackage{bm}
\usepackage{caption}
\usepackage{lineno}
\usepackage{authblk}
\usepackage{multicol}
\usepackage{amsmath}
\usepackage{amssymb}
\usepackage[svgnames]{xcolor}
\usepackage{array}
\usepackage{a4wide}
\usepackage{amsthm}
\DeclareMathOperator{\arcsinh}{arcsinh}
\usepackage{subfigure}
\usepackage{pst-node}
\usepackage[colorlinks=true,linkcolor=blue,citecolor=blue]{hyperref}
\usepackage{mathtools}
\usepackage{tikz-cd} 
\usetikzlibrary{arrows,shapes,positioning}
\usepackage{graphicx}
\usepackage{tikz}
\usepackage{wrapfig}
\usepackage{tensor}
\usepackage{enumitem}
\usepackage{outlines}
\usepackage{marginnote}

\reversemarginpar
\usetikzlibrary{arrows}
\usetikzlibrary{decorations.markings}
\usepackage[nottoc]{tocbibind}
\usepackage{cite}
\usepackage{fancyhdr}
\usepackage{graphicx}
\usetikzlibrary{calc,arrows}
\usepackage{mathtools}

\newtheorem{theorem}{Theorem}[section]
\newtheorem{prop}[theorem]{Proposition}

\newtheorem{lem}[theorem]{Lemma}

\theoremstyle{definition}
\newtheorem{Def}[theorem]{Definition}
\newtheorem{obs}[theorem]{Observation}
\newtheorem{rem}[theorem]{Remark}
\newtheorem{Exam}[theorem]{Example}

 \tikzstyle directed=[postaction={decorate,decoration={markings,
     mark=at position .9 with {\arrow[very thick, blue] {stealth} }}}]

 \tikzstyle reverse directed=[postaction={decorate,decoration={markings,
     mark=at position .8 with {\arrowreversed[very thick, blue]{stealth}}}}]

\numberwithin{equation}{section}

\def\Mt{\widetilde{M}}

\def\Gt{\widetilde{G}}
\def\Rt{\widetilde{R}}
\def\xt{\tilde{x}}
\def\yt{\tilde{y}}

\def\SDiff{{\mathop\mathrm{SDiff}\nolimits}}
\def\ori{{\mathop\mathfrak{or}\nolimits}}

\def\blambda{{\boldsymbol\lambda}}

\def\defn#1{{\itshape\bfseries#1}}
\def\d{\mathrm{d}}

 \author[a]{N.\ Balabanova}
\author[b]{J.A.\ Montaldi}

\affil[a]{School of Mathematics, University of Birmingham, Edgbaston,
Birmingham, B15 2TT,
UK; corresponding author. Email address: n.balabanova@bham.ac.uk}
\affil[b]{Dept of Mathematics,  University of Manchester, Oxford Rd, Manchester, M13 9PL, UK. Email address: j.montaldi@manchester.ac.uk}

\title{A Hamiltonian approach for point vortices on non-orientable surfaces}

\begin{document}

\maketitle
\begin{abstract}
    We investigate the motion of point vortices on 
    the M\"obius band and Klein bottle. Since these are non-orientable surfaces, the standard Hamiltonian approach does not apply.  We therefore begin by establishing a modified Hamiltonian approach which works for arbitrary non-orientable surfaces, through describing the phase space, the Hamiltonian and the local equations of motion.  We use a combination of twisted functions and oriented double covers to adapt some of the known notions of vortex dynamics to non-orientable surfaces.  For both of the surfaces of interest, we write Hamiltonian-type equations of vortex motion explicitly and follow that by the description of relative equilibria and an investigation of the motion of one and two vortices. 
  \end{abstract}
\newpage
   \tableofcontents

\normalsize

\normalsize
\section{Introduction}
Classical approaches to fluid dynamics do not prohibit considering fluid motion on non-orientable manifolds: the Euler equations 
 
$$
\rho \frac{Dv}{Dt} =  -\nabla p
 $$
with $\rho$ the density, $p$ the pressure and $v$ the divergence free flow of the fluid  only demand that the manifold in question be Riemannian  \cite{marchioro2012mathematical}; secondly, Arnold's interpretation \cite{arnold1999topological} of the fluid flow in terms of volume-preserving diffeomorphisms is a construction defined in the non-orientable case as well.   On the other hand, the standard definition of vorticity does depend on a choice of orientation, so the equation for the vorticity is only well-defined if the manifold is oriented. 

Much has been written about point vortices on a general \emph{oriented} surface (see \cite{aref1988point, aref1979motion,kidambi1998motion, Pekarsky-Marsden-1988,Montaldi-Gaxiola-Hyperbolic} and many others), and the setup involves using the strength $\Gamma\in\mathbb{R}$ of each point vortex, the symplectic form on the surface and the Hamiltonian which involves the hydrodynamic Green's function. The strength of a point vortex is positive if the local direction of flow agrees with the orientation of the surface, and is negative otherwise.  The symplectic form (or area form) defines the orientation of the surface.  

This work describes how to adapt this Hamiltonian approach to point vortices on non-orientable surfaces. On such surfaces, the sign of the vortex strength cannot be defined consistently, and neither is there a symplectic form.

To circumvent this, we define the vortex strength as a  \emph{twisted} or \emph{pseudo-scalar}, which means that to each local orientation one associates a scalar, and opposite orientations give rise to opposite scalar values. The equations of motion, in turn, will be described locally on oriented charts; however, as we will see below, their form will allow for globally defined trajectories of motion.

A system of $N$ point vortices, of strengths $\Gamma_1,\dots,\Gamma_N$,  on an oriented surface without boundary $M$ has phase space $M^N\setminus\Delta$ where $M^N = M\times M\times\dots\times M$ and $\Delta$ is the big diagonal\footnote{We use the symbol $\Delta$ for both the big diagonal and the Laplacian operator; in context, the two should be easily distinguishable.} (i.e.\ the set of $N$-tuples with two coinciding elements, representing collisions). The symplectic form on this phase space is
\begin{equation}
\label{eq:omega usual}\sigma = \bigoplus_{j=1}^N\Gamma_j\sigma_j\end{equation}
where $\sigma_j$ is the symplectic form on the $j$th copy of $M$, and the Hamiltonian is equal to the sum 
\begin{equation}
\label{eq:vortex Hamiltonian}
\mathcal{H}(x_1,\dots,x_N) = -\sum_{i\neq j} \Gamma_i\Gamma_j\,G(x_i,x_j) -
\frac12\sum_j\Gamma_j^2R(x_j),
\end{equation}
where $G(x,y)$ is the hydrodynamic Green's function (the fundamental solution of the Laplacian) and $R(x)$ is the Robin function (see for example \cite{dritschel2015motion,boatto-koiller-surfaces} and references therein) which is defined by
\begin{equation}
\label{eq: Robin function definition}
R(x) = \lim_{y\to x} \left(G(x,y) -\frac1{2\pi}\log d(x,y)\right).
\end{equation}
where $d(x,y)$ is the geodesic distance between $x$ and $y$.

The simplest example of such a system is the one on the plane,  consisting of the standard planar symplectic form and the logarithmic Hamiltonian $\mathcal{H} = -\frac{1}{2\pi}\sum_{\alpha<\beta}\Gamma_{\alpha}\Gamma_{\beta}  \log(|z_{\alpha} - z_{\beta}|)$ (see \cite{aref2007point,newton2001, kirchhof1883vorlesungen},etc.).

Imposing specific symmetries on a planar system allows for its interpretation as one on a two-dimensional manifold. 
    
    For example, dividing the plane into evenly spaced vertical strips of width $2\pi r$, each with the exact same set of point vortices gives us a system on a cylinder (see \cite{montaldi2003vortex,aref1996motion} for detailed construction), with the Hamiltonian given by $  \mathcal{H} = -\frac{1}{2\pi}\sum_{k<l} \Gamma_k \Gamma_l \log\left|\sin \frac{z_k-z_l}{2r}\right| $. 
    
    Periodic vortex configurations on 
     lattices yield systems on  flat tori, the dynamics of which have  been studied, for example, in \cite{o1989hamiltonian, stremler1999motion}.  In
 \cite{laurent2001point, dritschel2015motion} and \cite{sakajo2016point}  systems on curved tori and spheres are investigated; \cite{ragazzo2017motion} derives an algorithm determining the motion of a single point vortex on a surface of constant curvature and of genus greater than one. In a more general setup, \cite{ragazzo2017hydrodynamic} establishes existence and uniqueness of the Green's function under hydrodynamic conditions at the boundaries  and ends of the surface for an oriented Riemannian surface of finite topological type. The Green's function, together with the associated Robin function, is then used to describe the motion of systems of point vortices.

   In this paper we develop a `Hamiltonian' approach to point vortex motion on non-orientable surfaces and investigate in detail the motion on the (infinite) M\"obius band and on the Klein bottle. 

   While examples of fluid motion restricted to non-orientable surfaces is hard to come by, the other generalisation that this work provides can prove to be extremely useful. In principle, this work treats non-orientable manifolds as the results of factorisation; thus, a finite number of vortices on our model surfaces give rise to infinite periodic vortex lattices in the enveloping space.  
   Complexly arranged infinite (or very large scale) vortex structures have been used by physicists to model a wide variety of phenomena. In a (since proven false) attempt at mathematical description of Faraday's magneto-optic rotation, Lord Kelvin suggested treating atoms within a transparent dielectric material as loci of rotation, thus forming a vortex lattice. Perhaps the most fascinating example of vortex lattices in modern physics lies in the description of Bose–Einstein condensate systems -- see, for example, \cite{newton2009vortex}.  Laser stirring introduces angular momentum to a system of boson particles in an external potential that have been cooled to near-zero temperatures, enabling us to treat them as nearly parallel vortex filaments. 
   The ideas described in the scope of this paper allow to reduce  such filaments with a large number of vortices to a manageable number, using symmetries. It also provides a mathematical framework for a succinct description and prediction of behaviour of periodic systems with a large number of vortices.

\subsection*{Hamiltonian approach to point vortex flows on non-orientable manifolds. Symmetry  reduction}

 We establish that for flows on non-orientable manifolds vorticity must be interpreted as a two-form; for point vortex flow it means that its scalar counterpart, point vortex strength $\Gamma$, must be a so-called \textit{twisted scalar}. The same interpretation of vorticity and stream function as twisted functions was adopted in \cite{vanneste2021vortex}.

 Using the approach to construction of the phase space from \cite{Marsden-Weinstein-Vortices}, we demonstrate that the phase space for the motion of $N$ point vortices on a non-orientable manifold $M$ is the $N$-fold Cartesian product $\Mt^N = \bigotimes_{i=1}^N\Mt$ , where $\Mt$ is the orientable double cover of $M$.
   
   We make an observation that certain symmetries of systems of point vortices allow us to descend from an orientable double cover to a non-orientable manifold (systems with such symmetries have been investigated; for instance, \cite{laurent2001point} considers, in the language of this work, systems of vortices on the projective plane): vortices in such configurations come in pairs of opposite strength that lie at antipodal points on the double cover. 

We find the explicit form of the Hamiltonian (\ref{eq: Hamiltonian on mtilde halved}) and prove that it is a regular (i.e., not twisted) function on $M$ and $M^N$ and is  constructed in a standard way from the Green's function of the Laplace operator and its Robin function (Proposition \ref{st: green function and robin function}). 

   \subsection*{Explicit formulae}
Sections \ref{sec:vortex motion Mobius} and \ref{sec: Klein bottle general} are dedicated to detailed investigation of vortex motion on the M\"obius band and on the Klein bottle respectively.
   
   The model of the M\"obius band that we employ is an infinite strip of width $\pi$ on the plane with oppositely oriented vertical sides (as in Figure \ref{fig:Mb}), with a cylinder $\mathbb{C}/2\pi\mathbb{Z}$ as a double cover. We call the boundary of the strip the \textit{imaginary boundary}.
   
   For the Klein bottle, we use a square with sides of length $\pi$ with the appropriate identification of the sides and a torus $\mathbb{R}^2/((2\pi\mathbb{Z})\times(\pi\mathbb{Z}))$
   as the double cover. We discuss in detail the relations between periodising the Hamiltonian for the vortex motion on the covering space $\mathbb{R}^2$ and the antisymplectic involution of the double cover. 
   
Since both of our  models are elements of a plane periodisation, they admit local orientation naturally induced from $\mathbf{R
}^2$. This allows us to assume that point vortices move inside the strip or square as usual, but upon reaching the imaginary boundary they `jump' to the other side of it (see Figure \ref{fig:Mb}), changing the sign of their  strength and the sign of their $y$ coordinate if the sides had opposite orientation. 

Periodising the formulae on the double covers, we are able to write the Hamiltonian and the Hamilonian type equations of the system in both cases.  On the M\"obius band, we introduce a \textit{M\"obius flip} (formalization of the `vortex jump' described above and the orientation changing isometric involution for the double cover) of the covering system on a cylinder and prove that the Hamiltonian and the equations of motion are invariant under the M\"obius flip,  applied to any number of vortices. This guarantees that not only the equations of motion are well-defined, but that it also does not matter where we assume the imaginary boundary to be located: the motion will be the same regardless of where we draw it.

In Section \ref{sec: Ham Klein}, we demonstrate for the case of the Klein bottle  that out of the two natural methods of summing up the Hamiltonian, each one is compatible with only one periodisation; changing either of the ingredients leads to a not well-defined function. Namely, the Hamiltonian loses its invariance under  the orientation changing involution of the double cover. We establish the form of the Hamiltonian compatible with our natural periodisation and demonstrate that it conforms to all the conditions, i.e. is a well-defined (regular) function on the Klein bottle.

   \subsection*{Motion of point vortices}
   
   Having derived the equations of motion on the M\"obius band  early in Section \ref{sec:vortex motion Mobius}, we observe that the symmetry group of the system is $S^1$, acting by horizontal translations; this action has a globally defined  momentum map that is a regular (as opposed to twisted) function. In Section \ref{sec:equilibria} we show the existence of a general class of `equatorial' equilibria, and  Section \ref{sec: n ring equilibria} is dedicated to particular instances of relative equilibria that are aligned and staggered rings of vortices on the M\"obius band. 
   
   We describe in detail the motion of a single point vortex in Section \ref{sec: one vortex motion}.  Due to its simplicity, this case allows for a very clear outlook on how the trajectories of the global motion are glued together. Indeed, even in this simple case it becomes clear that the dynamics is not defined on $M$, but on $\Mt$. 
   
      For comparison with the planar case, we investigate the motion of two point vortices with opposite strengths, with centres placed with a vertical or horizontal symmetry. The former configuration will be a relative equilibrium and an example of a so-called vortex street \cite{stremler2014point}; the trajectories of the system with horizontal symmetry will be closed loops. 
   
      The next arrangement we investigate is an asymptotic one: point vortex strengths are assumed to be arbitrary, but the moduli of $y$-coordinates of both vortices are taken to be infinitely large. The two   will then move in horizontal lines, in the same or opposite directions, depending on the strengths of the vortices. However, due to different velocities this motion will not be a relative equilibrium. 
   
   In Section \ref{sec:two vortices}, we describe the motion of a pair of arbitrary point vortices, starting with a description of the fixed equilibria of two vortices. 
   
   In order to perform the calculations for the system of two point vortices, we observe that this system on a M\"obius band lifts to a system of four point vortices on a cylinder, and vice versa: the motion of a system of four point vortices with certain imposed symmetries on a cylinder correspond to that of  a pair of vortices on a M\"obius band. 
   
   Using the first integral of motion and the dependence of the Hamiltonian solely on $x_1-x_2$, we reduce the system and establish the minimum number of critical points. Numerical calculations point towards existence of two  kinds of  arrangements of the level sets of the reduced Hamiltonian; we describe the motion of the system corresponding to all principal types of the level sets within each of them.

   Section \ref{sec: Klein bottle general} is dedicated to analogous investigations for the point vortices on the Klein bottle; some of the arguments, especially in the later parts, can be repeated almost verbatim from Section \ref{sec:vortex motion Mobius}. 

Using formulation of vortex motion through complex numbers, we write the equations of motion and investigate the behaviour of one point vortex in Section \ref{sec: motion one vort klein}, proving that as in the previous case, it either moves horizontally (when the centre lies on the lines $y = 0, \pm\frac{\pi}{4}$ or $y=\pm\frac{\pi}{2}$) or remains stationary.

We establish in Section \ref{sec: symmetries Klein} that the symmetry group for systems of vortices on the bottle is $S^1$ as well, acting by horizontal translations-- due to periodisation, the vertical translational symmetry is lost. Leveraging on the results of Section \ref{sec:vortex motion Mobius}, we describe certain configurations that are fixed or relative equilibria, such as aligned or staggered vortex rings. 

The  constant of motion  corresponding to the action of $S^1$ is $C = \sum_k\Gamma_ky_k$; on the Klein bottle, it  is a local invariant only; however, we circumvent this for the motion of two vortices by covering the  $\pi$-by-$\pi$ square model of the Klein bottle with a cylinder  radius $2\pi$; on it, we follow the motion of a certain pair of vortices, whether or not they leave our copy of the bottle. 

This enables us to treat $C$ as a global  invariant and transfer to the reduced Hamiltonian  while losing no information about the motion of the system.

In Section \ref{sec: two point vortices Klein} we write the explicit form of equations of motion for two point vortices, show that the reduced Hamiltonian has critical points only on the line $x_1-x_2 = 0,\pm\pi$ 
and describe all the possible singularities occurring from collisions between the vortices themselves or their covering copies. Armed with that, we  restore the trajectories of original motion from the reduced system.

\begin{rem}
Most work on vorticity, whether via the Euler-Arnold equations or directly, either considers continuous vorticity or discrete point vortices. However, recent work of Shimizu \cite{Shimizu-currents} successfully combines the two by treating vorticity as a de Rham current. 
\end{rem}

  \section{Preliminaries}
%\subsection{Preliminaries}

%\textcolor{green}{1. It is recommended that a new ”Preliminary” section be added to ensure accessibility for
%readers less familiar with differential topology.
%This section should explicitly explain what fails when moving from orientable to nonorientable manifolds and how these issues affect the formulation of fluid dynamics. It is
%further suggested to introduce key Riemannian geometry concepts such as divergence, gradient, the Hodge Laplacian, and the Hodge decomposition within the context of non-orientable
%manifolds, emphasizing any notable differences or complications. While Section 2.1 offers a
%concise overview, it may not sufficiently clarify the fundamental challenges for readers accustomed to orientable manifolds. A more detailed exposition of these issues would significantly
%1
%enhance both the clarity and the impact of the manuscript.
%}

%\textcolor{orange}{2. Some readers of JFM are unfamiliar with the mathematical formulation presented in
%this paper. I suggest a brief explanation of mathematical concepts and ideas so that
%they can understand the formulation intuitively.}

%\fbox{Refer to \cite{GR-G-K-2024}}

\subsection{Twisted scalars and vector fields}
\label{sec: twisted definition}
As we have mentioned above, several hydrodynamic notions are directly tied to local orientation; we need their counterparts that are well-defined in the non-orientable case in order to describe our system in a Hamiltonian-like manner. 

Let $M$ be a manifold, possibly with boundary. Then at (or in a neighbourhood of) each point $p\in M$ there are two possible orientations. If we let $(p,o)$ be a point together with an orientation, then the collection of such pairs forms the orientation bundle over $M$, 
$$\pi:\Mt\to M$$
which we denote $\ori(M)$.  The fibre consists of two points, the two orientations of the manifold at that point. The total space $\Mt$ of $\ori(M)$ is connected if and only if $M$ is non-orientable. Furthermore, it has a tautological orientation:  the projection is a local diffeomorphism, so on each sheet one associates the orientation corresponding to that sheet. 

The manifold $\Mt$ is thereby an oriented double cover of $M$, and for the three principal 2-dimensional cases this is recognizable globally as the sphere as the double cover of the projective plane, the cylinder as the double cover of the M\"obius band and the torus as the double cover of the Klein bottle. 

Let $\tau:\Mt\to \Mt$ denote the  automorphism of $\mathfrak{or}(M)$ that swaps the two points in each fibre; over each point $p$ it acts by reversing the orientation at $p$ (for example, for the projective plane this is the antipodal map on the sphere).  Then $M$ is clearly the quotient of this total space by the action of the group generated by $\tau$.    

Any real-valued function, differential form or vector field on $M$ lifts to one on $\Mt$ which is in each case, invariant under $\tau$.  Conversely, any such object (scalar field, vector field etc) on $\Mt$ invariant under $\tau$ descends to a similar object on $M$. We refer to such objects as \defn{regular}.

On the other hand, if application of $\tau$ to say a scalar function $f$ on $\Mt$ changes its sign ($f\circ\tau=-f$), then it is called a \defn{twisted} or \defn{pseudo}-function on $M$; the same applies to forms and vector fields.  In other words, the object in question depends on the choice of orientation, and changing orientation changes the sign of the object.  This is the case for example of the curl of a vector field in $\mathbb{R}^3$. (For in-depth discussion of twisted objects we refer the reader to \cite{abraham2012manifolds,bott2013differential, bossavitdifferential}.)

\begin{rem}\label{rmk:line bundle}
Twisted functions are often defined (see \cite{ncatlab,penrose2006road}, for example) as sections of a twisted line bundle on $M$. The relationship with the approach above is as follows.  Let $\widetilde{L}$ be the trivial line bundle on $\Mt$.  Then $\pi_*\widetilde{L}$ is a rank-2 vector bundle on $M$ whose fibre over $p\in M$ is $\widetilde{L}_{p'}\oplus \widetilde{L}_{p''}$, where $\pi^{-1}(p)=\{p',p''\}$. Now $\tau$ induces an action on $\pi_*\widetilde{L}$ by swapping the two summands.  We can then decompose $\pi_*\widetilde{L}$ as the sum of two line bundles over $M$
$$\pi_*\widetilde{L} = L_+\oplus L_-$$
where sections of $L_+$ are those sections invariant under $\tau$ and sections of $L_-$ are those whose sign is changed by $\tau$.  The sections of $L_+$ correspond to genuine (regular) functions on $M$, while sections of $L_-$ correspond to twisted functions on $M$. 
A similar construction can be made for vector fields or forms, starting from $T\Mt$ or $T^*\Mt$ in place of $\widetilde{L}$. The sections of $(T\Mt)_-$ or $(T^{\ast}\Mt)_-$ are respectively twisted vector fields and twisted one forms.
\end{rem}

\subsection{Operators on twisted and regular forms}
From the description above, we can surmise that one of the possible ways to deal with twisted forms is simply lifting them to the orientable double cover and treating them as a linear subspace within the space of all forms. However, it is interesting to see how the notion of parity (twistedness or non-twistedness) of $k$-forms and $k$-vectors interacts with standard operations on them.  

\begin{Def}
    For a given Riemannian manifold $M$ we denote the set of regular $k$-forms on it by $\Omega_k$ and the set of twisted $k$-forms by $\widetilde{\Omega}_k$. 
\end{Def}

For brevity, we do not go into detail on the concepts and notions that can be found in the literature, the primary source being \cite{bott2013differential}. However, we present the formal computations along with the intuitive reasoning for the concepts that we have not seen formalised before in the language of differential geometry. 
 
 \textcolor{black}{Recall that the \textit{flat} musical isomorphism $\flat $ defines a map $\flat:TM\to T^{\ast}M$ for a Riemannian manifold $(M,g)$,  using the non-degenerate bilinear form $g$ that is defined at every point on it. Explicitly, for a vector $v\in T_p M$, $v^{\flat} = g_p(v,\cdot)\in T^{\ast}_pM$. This operator can be generalised to mapping $k-$vectors to $k$-forms by applying it ``individually" to each component: $\left(v_1\wedge\ldots\wedge v_k\right)^{\flat} = v_1^{\flat}\wedge\ldots\wedge v_k^{\flat}$.}  

 The Riemannian structure on $M$ is inherited from $\Mt$ and has no connection with orientability; therefore, it is clear that the operation $\flat$ is not only defined on $M$ but also preserves parity of  forms and vectors. 

 For an orientable manifold the operation $\diamond$ establishes duality between $k$-vectors and $n-k$ forms through the orientation.
 \textcolor{black}{Let $\omega$ be the volume form. Then for a $k$-vector $v_1\wedge\ldots\wedge v_k$  one defines  $\diamond:v_1\wedge\ldots\wedge v_k\mapsto \omega(v_1,\ldots,v_k,\cdot, \ldots, \cdot)$. The latter is clearly an $n-k$ form. }
 
 In the non-orientable case, volume form is twisted, and therefore the correspondence \textcolor{black}{constructed analogously to the one above -- by the twisted generalisation $\widetilde{\diamond}$ of the operator} $\diamond $ will be between twisted $n-k$-vectors and regular $k$-forms and vice versa. 

 Combining the two operations, as one does in the orientable case, we get the \textit{\textbf{twisted Hodge star}}
 \[
 \widetilde{\ast}:=\flat\ \widetilde{\diamond}.\]

 \textcolor{black}{As we stated above, the operator  $\widetilde{\diamond}$ is a twisted version of the standard $\diamond$, while $\flat$ is independent of orientation so is not twisted.
 This operator $\widetilde{\ast}$ therefore  establishes a linear mapping between twisted $k$-forms and regular $n-k$ forms-- this can be seen from the properties of the two components in the composition. }

 A regular Hodge star is widely used in coordinate-free definitions of tensor operations; to reformulate then intrinsically for the  twisted case, it remains to see how twistedness interacts with differentiation and whether the regular Hodge star can be replaced everywhere by its twisted counterpart.
\begin{prop}
\label{st: d preserves twistedness}$\mathrm{d}\left(\widetilde{\Omega}_k\right)\subset\widetilde{\Omega}_{k+1}$, i.e. a differential of a twisted form will be a twisted form. 
\end{prop}

\begin{proof}
By definition, for a $(k-1)$-form $\omega$ and vector fields $V_0,\ldots V_k$ we have
\[
\mathrm{d}\omega\left(V_0,\ldots V_k\right) = \sum\limits_i(-1)^iV_i\left(\omega\left(V_0,\ldots ,\hat{V}_i,\ldots V_k\right)\right) + \sum\limits_{i<j}\omega\left([V_i,V_j],V_0,\ldots ,\hat{V}_i,\ldots \hat{V}_j,\ldots V_k\right)
\]
Now take $V_0, \ldots,V_k$ to be vector fields and $\omega$ a twisted $(k-1)$-form. Then 
\begin{equation*}
    \begin{split}
       \tau\circ \mathrm{d}\omega(V_0,\ldots,V_k)  &= \sum\limits_i \mathrm{d}\tau(V_i)\left(\tau\circ \omega(V_0,\ldots,\hat{V}_i,\ldots, V_k)\right)\\ &+ \sum\limits_{i<j}(\tau^{\ast}\omega)\left(\mathrm{d}\tau[V_i,V_j],\mathrm{d}\tau V_0,\ldots ,\mathrm{d}\tau \hat{V}_i,\ldots \mathrm{d}\tau\hat{V}_j,\ldots \mathrm{d}\tau V_k\right)\\&=\sum\limits_i -\mathrm{d}\tau(V_i)\left( \omega(V_0,\ldots,\hat{V}_i,\ldots, V_k)\right)\\ &+ \sum\limits_{i<j}-\omega\left([V_i,V_j],\mathrm{d}\tau V_0,\ldots , \hat{V}_i,\ldots \hat{V}_j,\ldots 
       V_k\right)\\&
       =-\mathrm{d}\omega(V_0,\ldots,V_k).
    \end{split}
\end{equation*}
\end{proof}
Since a differential of a regular form is clearly itself a regular form, differentiation is a parity-preserving operation. 

\begin{Def}
\label{def: deltatilde}
    On a Riemannian manifold $M$, define the (co-differential) operator ${\delta}$ on twisted and regular $k$-forms as
    \begin{equation}
        \label{eq: deltatilde}
       {\delta} = (-1)^{k}\ \widetilde{\ast}\ \mathrm{d}\ \widetilde{\ast}
    \end{equation}
\end{Def}
It easily follows from Proposition \ref{st: d preserves twistedness} that ${\delta}$ sends twisted $k$-forms to twisted $k-1$ forms and regular $k$-forms to regular $k-1$-forms. 

\begin{rem}
    In the light of the last two statements, we keep the notation as follows: parity-preserving operations, such as $\delta$ and d will not have a tilde on top, whereas parity-changing operations, such as $\widetilde{\ast}$, will.  The relations between twisted and regular forms and operators $\delta$, $\widetilde{\ast}$ and d
 are as in Figure \ref{fig: operators and forms}.
\end{rem}

 These considerations allow us to generalise all operations on regular forms and extend them to twisted: the only thing one has to bear in mind is for the operator $\widetilde{\ast}$ to be applied an even number of times.  We point out a few obvious generalisations.
 \begin{enumerate}
     \item By Proposition \ref{st: d preserves twistedness}, the gradient $\mathrm{d} f$ of a twisted function $f$ will be a twisted  1-form. 
     \item Let  $V$ be a vector field (twisted or regular). Then define:
        \[
       \nabla(V):= \widetilde{\ast}\,\mathrm{d}\,\widetilde{\ast}(V^{\flat}).
        \]
       This operation clearly preserves parity, since $\widetilde{\ast}$ is applied twice. Computationally, if $V$ is a twisted vector field on an $n$-dimensional manifold, then 
        \begin{equation*}
            \begin{split}
                V^{\flat}\in\widetilde{\Omega}_1 &\implies \widetilde{\ast}\left(V^{\flat}\right)\in \Omega_{n-1}\implies \mathrm{d}\left(\widetilde{\ast}\left(V^{\flat}\right)\right)\in\Omega_n\implies {\nabla}(v)\in\widetilde{\Omega}_0. 
            \end{split}
        \end{equation*}
        \item Recall the coordinate free definition of the Laplacian  $\Delta$ for an $n$-dimensional Riemann manifold: $\Delta := \delta \  \mathrm{d}$. Since we have a notion of the generalised $\delta$, an analogous twisted structure can be defined. Take 
\begin{equation} 
\label{eq: twisted delta}
\Delta(f):= \delta \ \mathrm{d} f, 
\end{equation}
interpreting $\delta$ as in Definition \ref{def: deltatilde}. It is evident that 
\[
\Delta:\Omega_0\to\Omega_0, \ \ \ \Delta:\widetilde{\Omega}_0\to\widetilde{\Omega}_0.
\]
\item If $M$ is Riemannian and compact, the space of $k$-forms admits an orthogonal decomposition (see \cite{GR-G-K-2024, wells1980differential})
\[
\Omega_k = \mathrm{Im} \ \mathrm{d}\oplus\mathrm{Im}\ \delta\oplus\mathrm{Ker} \left(\d \delta+ \delta\d\right).
\]
Armed with the discussions above, we can state that the same decomposition holds for twisted forms as well: to see it, we need to lift everything to the double cover and recall that both operations involved are parity-preserving. Hence,
\[
\widetilde{\Omega}_k = \mathrm{Im} \ \mathrm{d}\oplus\mathrm{Im}\ \delta\oplus\mathrm{Ker} \left(\d \delta+ \delta\d\right).
\]
 \end{enumerate}

\begin{figure}
    \centering
    
\includegraphics{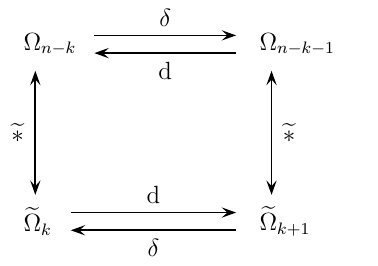}

    \caption{Commutative diagram for operators $\delta,$ d, twisted and regular forms on an $n$-dimensional manifold $M$.}
    \label{fig: operators and forms}
\end{figure}

Lastly, we derive the most important generalisation for this work:
\begin{Def}
Define the (twisted) Poisson equation as 
\begin{equation}
    \label{eq: twisted Poisson}
    \Delta f = \omega
\end{equation}
\end{Def}
On orientable manifolds, this equation can be interpreted as the standard Poisson equation -- due to parity-preserving nature of $\Delta$. But on non-orientable manifolds we can treat it as a differential equation in twisted forms, as well as regular ones -- below we will see how hydrodynamical notions can be given a twisted interpretation. 

Henceforth, when we mention the Poisson equation or the Laplace operator, we invariably mean the operator as in \eqref{eq: twisted delta} and the equation as in  \eqref{eq: twisted Poisson}.

%\textcolor{green}{2.3. Confirm that the Hamiltonian vector field of the stream function corresponds to a
%smooth, divergence-free vector field on the non-orientable surface.}

%It doesn't though, does it? The Hamiltonian vecor field does not exist on the non-orientable surface. The local "Hamiltonian" fields will be divergence free automatically, since they are identical locally to the one on the double cover. 

%The Hamiltonian twisted vector field??

  \section[The Hamiltonian approach]{Hamiltonian approach to point vortex flows on \\ non-orientable manifolds} 

\subsection{Vorticity and vortex strength}

\label{sec: vort and vort st}

Consider a divergence-free vector field $\mathbf{v}$ on a Riemannian surface $M$ (tangent to $\partial M$, if the boundary is non-empty), and let $\alpha=\mathbf{v}^\flat$ denote the corresponding 1-form defined using the Riemannian metric.  Then let $\blambda=\d\alpha$: this is the \defn{vorticity 2-form}.  It is clearly defined independently of any orientation. 

On the other hand, the vorticity-as-a-scalar does depend on the orientation and is defined on the double cover $\Mt$ as follows.  Since $\Mt$ is oriented and Riemannian, it has a natural symplectic form we denote by $\sigma$.  Given the vorticity 2-form $\blambda$ on $M$ we can lift this to $\Mt$, and find a smooth function $\omega$ on $\Mt$ satisfying
$\pi^*\blambda=\omega\sigma$.  This is the scalar vorticity on $\Mt$.

Now $\tau^*\pi^*\blambda = \pi^*\blambda$, while $\tau^*\sigma=-\sigma$, and it follows that $\tau^*\omega=-\omega$, which is to say that the scalar vorticity $\omega$ is a \emph{twisted} scalar on $M$, as we know from the discussion in the introduction. Indeed, the operation passing from $\blambda$ to $\omega$ is an example of the twisted Hodge star operator, which on non-oriented manifolds takes regular $p$-forms to twisted $(n-p)$-forms.

\textcolor{black}{The vorticity equation can be derived from the the Euler equations (see e.g., \cite{acheson, falkovich})  -- they are represented in a simplified form at the beginning of the paper. It takes the following form} on $M$ or $\Mt$:
\begin{equation}
    \label{eq: scalar vorticity}
\blambda_t+L_{\mathbf{v}}\blambda=0,\quad\text{or}\quad \omega_t+L_{\mathbf{v}}\omega=0.\end{equation}
The first is an equation of 2-forms on $M$ (involving the Lie derivative $L_{\mathbf{v}}$), while the second is of twisted scalars. 

In the classical 3-dimensional  hydrodynamical  approach for a small oriented  neighbourhood of the point vortex the \emph{point vortex strength} $\Gamma$ is the flux of the vorticity through an open surface $A$ bounded by a curve  $C$ that encircles the vortex -- or, by Stokes' theorem, the circulation of the vector field around $C$.  In our terminology, that will be an integral of the one-form $\mathbf{v}^{\flat}$: 

\begin{equation}
\Gamma =  \oint_C  \ \mathbf{v}^{\flat} =\int_{\bar{A}} \ \blambda
\end{equation}
where $\bar{A}$ is the region $A$ inside $C$ with the given orientation. Observe how the sign of $\Gamma$ depends on the (local) orientation.

An alternative interpretation of point vortex strength on an oriented manifold $\Mt$ is  the coefficient $\Gamma$ at the Dirac delta function of the vorticity-as-a-scalar of a point vortex.

Analogously, for point vortices on the non-orientable $M$, $\Gamma\delta(x)$ is a twisted function (or twisted distribution). As we will see below, the `twistedness' of $\Gamma$ will be crucial  for the equations of motion to be well-defined; therefore, we assume that $\delta(x)$ is a regular `function' (distribution), while $\Gamma$ is a twisted scalar (that is, a real number associated to the point vortex whose sign depends on the local orientation chosen).

%\begin{rem}
%Equation (\ref{eq: scalar vorticity}) on the set of exact 2-forms $\blambda$ on a manifold $\Mt$ can be obtained from differentiating the Euler equation on the space $\mathfrak{g}^{\ast}=\Omega^1/\mathrm{d}\Omega^0$, the dual space to the space of divergence-free vector fields. This equation is Hamiltonian, with $H([u]) = - \frac12 \left<A^{-1}[u],[u]\right>$, with $[u]$ an element of $\mathfrak{g}^{\ast}$ and $A$ the inertia operator (see \cite{Marsden-Weinstein-Vortices,arnold1999topological} for further details). 
%
%Vorticities $\omega_1,\ldots,\omega_k $, created by  point vortices, describe weak solutions of the vorticity equation. The correspondence mentioned above allows to rewrite the Hamiltonian in terms of $\omega_1,\ldots,\omega_k$; since their support is finite dimensional, the Hamiltonian is finite dimensional as well. 
% 
%    However, our goal was to find an approach that is closest to the classical Hamiltonian one, and determine which aspects of it could be written intrinsically for a non-orientable manifold.  
%\end{rem}

\subsection{Phase space for point vortices}
We follow the approach of Marsden and Weinstein  \cite{Marsden-Weinstein-Vortices}, and note that for the most part this also holds for non-orientable surfaces. 
Given an initial vorticity 2-form $\blambda$ on a 2-dimensional Riemannian manifold $M$, the phase space\footnote{Note that if $M$ is oriented then $\SDiff(M)$ is not connected, and it would be more efficient to restrict to $\SDiff(M)_0$, which consists of volume-preserving diffeomorphisms isotopic to the identity; this is possibly what the authors of \cite{Marsden-Weinstein-Vortices} had in mind.} for the dynamics is 
$$\mathcal{O}_\blambda = \left\{\phi^*\blambda \mid \phi\in\SDiff(M)\right\},$$
where $\SDiff(M)$ is the group of volume-preserving diffeomorphisms of $M$; this can be interpreted as the coadjoint orbit of $\SDiff(M)$ containing $\blambda$. The symplectic form is given by the Kostant-Kirillov-Souriau form,
$$\Omega_\blambda(L_u\blambda,L_v\blambda) = \int_M\blambda(u,v)\,\d A$$
where $u,v\in\mathcal{X}_0(M)$ (divergence-free vector fields), $L_u$ and $L_v$ are the Lie derivatives and $\d A$ is the area measure consistent with the metric. This does not require $M$ itself to be oriented or symplectic.  For a regular (i.e., smooth) vorticity form, the Hamiltonian is given by the total kinetic energy of the fluid.

The motion of a system of $N$ point vortices on an \emph{orientable} surface $M$ is determined by the location and strength of the point vortices; the strength here being the (scalar) vorticity concentrated at a point, in the form of delta functions.  
In this case, it was pointed out in \cite{Marsden-Weinstein-Vortices} that the phase space $\mathcal{O}_\blambda$ can be identified with $M\times M\times\dots \times M\setminus\Delta$ (the product of $N$ copies of $M$ with the collision set removed). The identification is 
\begin{equation}\label{eq:phase space - oriented case}
(x_1,\dots,x_N) \longmapsto \sum_j\Gamma_j\delta_{x_j}\sigma
\end{equation}
where $\sigma$ is the symplectic form (area form) on $M$ and $\delta_p$ is a delta function supported at $p$. The symplectic form on $\mathcal{O}_\blambda$ becomes
$$\Omega_\blambda(\mathbf{u},\mathbf{v}) = \sum_j\Gamma_j\,\sigma_{x_j}(u_j,v_j)$$
where $u_j,v_j\in T_{x_j}M$.

As described in \eqref{eq:vortex Hamiltonian}, the Hamiltonian is given by 
\begin{equation}\label{eq:Hamiltonian oriented case}
\mathcal{H}(x_1,\dots,x_N) = -\sum_{i<j} \Gamma_i\Gamma_j G(x_i,x_j) -  \tfrac12\sum_j\Gamma_j^2 R(x_j).
\end{equation}

Now consider the case where $M$ is non-orientable. Two or three problems arise when adapting the discussion above and we need to employ the twisted construction described above.

Most obviously, since on $M$ the $\Gamma_j$ are twisted scalars, it is not clear what the first term in the Hamiltonian \eqref{eq:Hamiltonian oriented case} means.  Moreover, the phase space $\mathcal{O}_\blambda$ can no longer be identified with a product of $N$ copies of $M$. This is because, a given $N$-tuple of points $(x_1,\dots,x_N)$ does not determine the interaction between the vortices. Indeed, when $M$ is non-orientable, there is a volume preserving diffeomorphism which takes say $x_1$ back to $x_1$ by an `orientation-changing' loop, which has the effect of reversing the contribution to the vorticity 2-form at the point $x_1$. 

Thus, in $\mathcal{O}_\blambda$ there are vorticities supported at the same points, but with any subset of the local vorticity 2-forms reversed.  That is, if
$\blambda=\sum_j\lambda_{x_j}$
where $\lambda_{x_j}$ is a delta-function 2-form (a 2-current), then $\mathcal{O}_\blambda$ also contains
$\blambda'=\sum_j(\pm\lambda_{x_j})$
for all choices of signs.  This leads to a parametrization of $\mathcal{O}_\blambda$ by the product of $N$ double covers of $M$.

 For the vorticity 2-form $\blambda$ on $M$, let $\{x_1,\dots,x_N\}$ be the support of $\blambda$. For each $x$ in this set let $\pi^{-1}(x)= \{\xt,\,\tau(\xt)\}$ for some $\xt\in\Mt$. 

 \textcolor{black}{\begin{rem}To distinguish between points and their lifts, we adhere to the following notation: the points without tildes belong to the non-orientable manifold $M$ and with tildes to its double cover $\Mt$.\end{rem}}
 
 The pull-back of $\blambda$ to $\Mt$ is a 2-form, and the tautological symplectic from $\sigma$ allows us to define the vortex strengths $\Gamma$ at $\xt$ and $-\Gamma$ at $\tau(\xt)$ (we do not assume $\Gamma>0$, the vorticity will be positive at the point where the orientation agrees with the direction of circulation of the vortex, and negative (of the same magnitude) at the other point).  The system can be parametrized by 
$(\xt_1,\dots,\xt_N)$ and the associated vorticities $(\Gamma_1,\dots,\Gamma_N)$. \textcolor{black}{Thus, we have demonstrated}

\begin{prop}\label{prop:phase space}
The phase space $\mathcal{O}_\blambda$ for a system of $N$ point vortices on a non-orientable surface $M$ can be parametrized by
$\Mt\times\dots\times\Mt\setminus\widetilde{\Delta}$, where $\widetilde{\Delta}$ is the `very big diagonal' consisting of points $(\xt_1,\dots\xt_N)\in\Mt^N$ for which $i\neq j\Rightarrow \pi(\xt_i)\neq\pi(\xt_j)$. 
\end{prop}

\subsection{The Hamiltonian}
\label{sec: the Hamiltonian}
In the same notation as above, given a point  $x$ on $M$, for the preferred lift we are making choices between $\xt$ and $\tau(\xt)$. The choice could be made by requiring all $\Gamma_j>0$ for example. However, that turns out not to be convenient for calculations, so we allow an arbitrary choice for each $x_j$.

\textcolor{black}{As we have established above, for one vortex $x_0$, the pullback of the vorticity form $\lambda_{x_0}$ on $M$ to $\Mt$ must have two singularities: at $\xt_0$ and $\tau(\xt_0)$. These two  singularities correspond to the two different choices of orientation.}

Thus, given any point $(\xt_1,\dots \xt_N)\in\Mt^N\setminus \widetilde{\Delta}$ the system on $\Mt$ consists of $2N$ vortices at points $\{\xt_1,\dots\xt_N,\tau(\xt_1),\dots,\tau(\xt_N)\}$. 
\textcolor{black}{Of the two lifts of $x_j$, we call $\xt_j$ the \emph{preferred} lift, and set $\Gamma_j$ to be the vorticity there. The vorticity at $\tau(x_j)$ is then $-\Gamma_j$ (this follows from $\Gamma$ being a twisted scalar)}.

Using \eqref{eq:Hamiltonian oriented case}, and this choice for the $\Gamma_j$, the Hamiltonian on this covering phase space is given by 
\begin{equation} \label{eq: Ham of the double cover - long}
\begin{split}
   \mathcal{H}_{\Mt}(x_1,\ldots,x_N) =&  -\sum_{i< j}\Gamma_i\Gamma_j\,\Gt(\xt_i,\xt_j)
   	-\sum_{i< j}\Gamma_i\Gamma_j\,\Gt(\tau(\xt_i),\tau(\xt_j)) 
   +\sum_{i\neq  j}\Gamma_i\Gamma_j\,\Gt(\xt_i,\tau(\xt_j)) \\
    &{}+ \sum_j\Gamma_j^2\Gt(\xt_j,\tau(\xt_j))  - \frac12\sum_j\Gamma_j^2\Rt(\xt_j)- \frac12\sum_j\Gamma_j^2\Rt(\tau(\xt_j))
    \end{split}
\end{equation}
where $\Gt(x,y)$ is the Green's function of the Laplacian  on $\Mt$ and $\Rt$ the corresponding  Robin function.
{\red Note that if we change one of the preferred lifts, say $\xt_i$ to $\tau(\xt_i)$, then we change the sign of the corresponding $\Gamma_i$ and the whole sum in the Hamiltonian is invariant.}

However, an adjustment needs to be made to (\ref{eq: Ham of the double cover - long}): for justification, we return to the body of fluid on  $\Mt$. Due to the absence of external forces, the total energy of the fluid  is given by the surface integral $\int_{\Mt}\frac12\rho|v|^2\mathrm{d}s$, where $\rho$ is the constant density of the fluid and $v$ its velocity. Transferring to the flow on $M$  means that we only consider half of the fluid; therefore, the total energy must be divided by 2. 

\textcolor{black}{Since $\xt $ and $\tau(\xt)$ are functions of $x$ (by using $\xt$ in the formula below)}, we conclude that, \textcolor{black}{bearing in mind the covering system on $\Mt$,} the total energy of the fluid on $M$ will be given by 
\begin{equation}
\label{eq: Hamiltonian on mtilde halved}
    \begin{split}
         \mathcal{H}(x_1,\ldots,x_N) =&  -\frac12\sum_{i< j}\Gamma_i\Gamma_j\,\Gt(\xt_i,\xt_j)
   	-\frac12\sum_{i< j}\Gamma_i\Gamma_j\,\Gt(\tau(\xt_i),\tau(\xt_j)) 
   +\frac12\sum_{i\neq  j}\Gamma_i\Gamma_j\,\Gt(\xt_i,\tau(\xt_j)) \\
    &{}+ \frac12\sum_j\Gamma_j^2\Gt(\xt_j,\tau(\xt_j))  - \frac14\sum_j\Gamma_j^2\Rt(\xt_j)- \frac14\sum_j\Gamma_j^2\Rt(\tau(\xt_j)),
    \end{split}
\end{equation}
\textcolor{black}{that is, exactly the half of the Hamiltonian for the covering system. }

Since $\tau$ is a globally defined isometry of $\Mt$ it follows that both $\Gt$ and $\Rt$ are invariant under $\tau$, which implies some of these terms coincide.

Suppose that we fix $\yt$, a preferred lift of some other  point $y\in M$ and $\xt$, a preferred lift of $x$. Observe how (\ref{eq: Hamiltonian on mtilde halved}) is unaffected by which of the two preimages we choose to be our preferred lift, so we are free to pick either one. This allows us to express Green's function on $M$ through the one on $\Mt$:
%With the preferred lift of $x$ being $x'$ define the Green's and Robin functions on $M$ to be
\begin{Def}
We define the Green's function on $M$, \textcolor{black}{twisted in each of its arguments,} as 
\begin{equation}
\label{eq: twisted Green}
G_M(x,y) = \Gt(\xt,\yt) -\Gt(\xt,\tau(\yt))
\end{equation}
and the (regular) Robin function on $M$ as
\begin{equation}
\label{eq: twisted Robin}
     R_M(x) = \Rt(\xt)-\Gt(\xt,\tau(\xt))
\end{equation}
%Notice that $G_M$ is twisted in each of its arguments while $R_M$ is a regular (non-twisted) function on $M$.
\end{Def}
\textcolor{black}{\begin{rem}
 The function $G(x,y)$ is a well-defined twisted function in each argument---since $\Gt$ is invariant with respect to isometries, a different choice of the preferred lift in either $\xt$ or $\yt$ will lead to a change in the sign.
     \end{rem}}
\begin{prop}
\label{st: green function and robin function}
 $G_M$ is indeed  a twisted Green's function of the Laplacian on $M$, and $R_M(x)$ is the Robin function of $G_M(x,y)$.
\end{prop}
\begin{proof}

Consider the Laplacian of $G_M$ as a function on $\Mt$:
\[
\Delta_{\xt}\left( \Gt(\xt,\yt) -\Gt(\xt,\tau(\yt))\right) = \delta_{\yt}(\xt) - \delta_{\tau(\yt)}(\xt)  
\]
For a fixed $y$ (and, consequently, $\yt$) $\Delta_{\xt}\left(\Gt(\xt,\yt) -\Gt(\xt,\tau(\yt))\right)$ descends to $M$ as a twisted function $\widetilde{\delta}_y(x)$. \textcolor{black}{Note that $\widetilde{\delta}_y(x)$ is a \textit{twisted function} defined on the manifold $M$ (whereas $\delta_{\yt}(\xt) - \delta_{\tau(\yt)}(\xt)$ is a function on $\Mt$).} 

Analogously to the regular Dirac delta function, it indicates when one of the preimages of $x$ on $\Mt$ coincides with one of the preimages of $y$ on $\Mt$.  

As one can easily see, $G_M$ depends on the choice of the preferred lift, since it is twisted in both of its variables. 

Now we want to check that $R_M$ coincides with the desingularization of $G_M$.

Saying  that $x\to y$ for two points $x,y\in M$ is equivalent to one of the two statements for $\Mt$: $\xt\to \yt$ or $\xt\to \tau(\yt)$. Observe, however, that  $R_M$, despite being a limit, is a function of one variable only and is therefore covered by a function on $\Mt$. 

 Thus, we are restricted to one lift only: either $\xt\to \yt$ or $\xt\to \tau(\yt)$, depending on $\tau$ and our choice of preferred lifts. Since $\xt$ is a `mute' variable, it does not matter which of the two options we choose, as can be seen from the calculations  below.

We have from (\ref{eq: Robin function definition}) and (\ref{eq: twisted Robin}):
\begin{equation}
\label{eq: Robin form HAM}
    \begin{split}
        \Rt(\yt)- \Gt(\yt,\tau(\yt)) 
        =& \lim_{\xt\to \yt}\left(\Gt(\yt,\xt) -\frac{1}{2\pi}\log d(\yt,\xt)  - \frac12\Gt(\yt, \tau(\xt))  - \frac12\Gt(\xt,\tau(\yt))\right)  \\ =&\lim_{\xt\to \yt}\biggl(\frac12\Gt(\yt,\xt) + \frac12\Gt(\tau(\yt),\tau(\xt))-\frac{1}{2\pi}\log d(\yt,\xt)  \\& \qquad\quad -\frac12\Gt(\yt, \tau(\xt))  - \frac12\Gt(\xt,\tau(\yt))\biggr)  \\ =&\lim_{\xt\to \yt}\left(G_M(\yt,\xt) -\frac{1}{2\pi}\log d(\yt,\xt) \right).
    \end{split}
\end{equation}
Since we choose one lift out of two, the 
explicit form of $\Rt(\tau(\yt)) - \Gt(\yt,\tau(\yt))$  is the limit with $\tau(\xt)\to\tau(\yt)$, which can be shown to coincide with (\ref{eq: Robin form HAM}). Hence, $R_M(x)$ is a regular function on $M$, independent of our choice of preferred lift.
\end{proof}

This enables us to rewrite  \eqref{eq: Ham of the double cover - long} as
\begin{equation} \label{eq: Ham of the double cover}
   \mathcal{H}(x_1,\ldots,x_N) =  -\sum_{i< j}\Gamma_i\Gamma_j\,G_M(x_i,x_j)
   	- \tfrac12 \sum_j\Gamma_j^2R_M(x_j)
\end{equation}

Since exchanging $\xt$ with $\tau(\xt)$ also involves changing the sign of the corresponding $\Gamma$, it is clear from \eqref{eq: Ham of the double cover - long} that this Hamiltonian function $\mathcal{H}$ is well-defined on $M\times\dots\times M\setminus\Delta$. 

Notice that even if the Robin function on $\Mt$ is constant, the one on $M$ may not be, because of the term $\Gt(\xt,\tau(\xt))$. For example, on the projective plane with the usual metric the Robin function is constant, while on the M\"obius band and Klein bottle it is not, as we see below.  

To finish formalizing the Hamiltonian approach, we establish the symplectic form on the phase space: 
\begin{lem}
\label{st: symplectic form double cover}
The symplectic form on $\Mt\times\ldots\times\Mt\setminus\widetilde{\Delta}$ from Proposition \ref{prop:phase space} is given by $\bigoplus\limits_{k=1}^N\Gamma_k\sigma_k$, where $\sigma_j$ is the symplectic form on $j$th copy of $\Mt$. 
\end{lem}
\begin{proof}
The vorticity of the fluid on the $j$th copy of $\Mt$ is given by $\Gamma_j\delta_{\xt_j}\sigma - \Gamma_j\delta_{\tau(\xt_j)}\sigma$. Due to the symmetries of our system, if the  velocity of the point $\xt_j$ is given by a vector $v$, the velocity of $\tau(\xt_j)$ will be $\mathrm{d}\tau(v)$. Hence, the calculation of the symplectic form as in \cite{Marsden-Weinstein-Vortices} gives us
\[
\Omega_{\lambda}(\mathbf{u},\mathbf{v}) = \sum_j\Gamma_j\sigma_{\xt_j}(u_j,v_j) - \Gamma_j\sigma_{\tau(\xt_j)}\left(\d\tau(u_j), \d\tau(v_j)\right) = 2\sum_j\Gamma_j\sigma_{\xt_j}(u_j,v_j),
\]
since $\tau^{\ast}\sigma = -\sigma$. We have divided the Hamiltonian by 2 and therefore have to do the same to the symplectic form, obtaining the statement of the lemma. Note that this symplectic form is  independent of our choice of preferred lift.
\end{proof}
Obviously, the manifold $\Mt\times\ldots\times\Mt\setminus\widetilde{\Delta}$ is a double cover of the surface $M\times\ldots\times M\backslash \Delta $. The latter obviously cannot have a Hamiltonian vector field defined on it; however, it still inherits a certain structure. 

\begin{lem}
    The Hamiltonian vector field on $\Mt\times\ldots\times\Mt\setminus\widetilde{\Delta}$, constructed using the Hamiltonian (\ref{eq: Ham of the double cover}) and the symplectic form $\sigma$ from Lemma \ref{st: symplectic form double cover} descends to a twisted, smooth, divergence-free ``Hamiltonian" vector field on $M\times\ldots\times M\backslash \Delta $.
\end{lem}
\begin{proof}
    The Hamiltonian vector field on $\Mt\times\ldots\times\Mt\setminus\widetilde{\Delta}$ is smooth and divergence-free, because it is a Hamilonian vector field; if it induces a (twisted) vector field on $M\times\ldots\times M\backslash \Delta $, then the latter will clearly inherit these properties. Hence, we only need to show that a twisted field on $M\times\ldots\times M\backslash \Delta $ that lifts to our Hamiltonian field, exists. 

We have already shown that $\mathcal{H}$ is a regular function on $M\times\ldots\times M\backslash \Delta $. For $\sigma$, we have
\[
\tau^{\ast}\sigma = \tau^{\ast}\left(\sum\limits_{k=1}^N\Gamma_k\sigma_k\right) = -\sum\limits_{k=1}^N\Gamma_k\sigma_k,
\]
since $\sigma_k$ are top-dimensional forms (as pointed out in the proof of Lemma \ref{st: symplectic form double cover}). Note that $\Gamma_k$ do not change their sign since we have lost that symmetry by fixing a preferred lift. 
\end{proof}
\begin{rem}

  By the symmetric placement of vortices on $\Mt,$ the stream function $\psi_{\widetilde{M}}$ on $\widetilde{M}$  is
\begin{equation}
\label{eq: stream function nonoritentable}
\psi_{\widetilde{M}}(x) = \sum_k\Gamma_k\left(\Gt(x,x_k) -\Gt(x,\tau(x_k)) \right),
\end{equation}
where $x_k$ and $\tau(x_k)$ are the centres of the point vortices and their copies respectively.

Observe that $\psi_{\Mt}$ is a twisted function in $x$: $\psi_{\Mt}\left(\tau(x)\right) = -\psi_{\Mt}(x)$. However, $\psi_{\Mt}$ is not sensitive to exchanging $x_k$ and $\tau(x_k)$, since that changes not only the sign of $G(x,x_k) -G(x,\tau(x_k))$, but the sign of $\Gamma_k$ as well.

Note that Robin functions are not involved -- in principle, it follows from twistedness of the vorticity form and parity preserving properties of the Laplacian; computationally, since the stream function depends on the difference between functions of $x_k$ and $\tau(x_k)$, Robin functions cancel out. 

This makes sense from the hydrodynamic point of view: the descent of $\psi_{\Mt}$ to $M$, the stream function $\psi_M$, must be a twisted function of $x\in M$ (\cite{vanneste2021vortex}), unaffected by our choice of preferred lift.
\end{rem}
\begin{prop}
    The stream function $\psi_{M}(x)$ (\ref{eq: stream function nonoritentable}) solves the (twisted) Poisson equation on the manifold $M$ for the twisted point vortex vorticity $\omega$.
\end{prop}
\begin{proof}
We need to demonstrate the statement for one vortex, since the streamfunction and Laplacian are additive on the set of vortices.

By definition, for a single vortex on $M$ (corresponding to two on $\Mt$) the  Laplacian of the stream function will be 
\begin{equation}
\label{eq: delta psi twisted}
\Delta\psi_{\Mt}  = \Gamma(\delta_{\yt}(\xt)- \delta_{\tau(\yt)}(\xt)),
\end{equation}
which we write as a function on $\Mt$ and which, as we already know, descends to  a twisted function on $M$. 
However,  bear in mind (\ref{eq: delta psi twisted}) was written on $\Mt$ interpreting $\Gamma$ as a scalar; an assumption that cannot be made for $M$. Therefore, when writing vorticity-as-a-scalar intrinsically on $M$, we  `delegate' the twistedness  to $\Gamma$, as we did at the end of Section \ref{sec: vort and vort st}. 

This is done in the following way: for a fixed $y\in M $, we lift the regular delta function $\delta_y x$ to $\Mt$, to obtain $\delta_{\yt} \xt + \delta_{\tau(\yt)}\xt$. However, by placing point vortices of strengths $\Gamma$ and $-\Gamma$ at $\yt $ and $\tau(\yt)$ respectively, we 'twist' it, making it into (\ref{eq: delta psi twisted}).

Thus, after applying the Laplacian operator to the streamfunction $\psi_{\Mt}$, we obtain a twisted product of a Dirac delta function and the magnitude of vorticity, which was the desired result.
\end{proof}

\subsection{The momentum map} \label{sec:momentum map}\def\gg{\mathfrak{g}}
Suppose a connected Lie group $G$ acts by isometries on $M$ and denote its Lie algebra by $\gg$. We describe the (possibly local) momentum map on the phase space associated to this action. 

The action lifts to an orientation-preserving  action of {\red a cover $\Gt$ of $G$} by isometries on $\Mt$ (see \cite{bredon1972introduction}); therefore, the action in question is symplectic. 

Consider an element $\xi\in\gg$ and let $\xi_{\Mt}$ be the associated vector field of the infinitesimal action on $\Mt$. Then, since the action of $G$ on $\Mt$ is the lift of one on $M$, $\mathrm{d}\tau(\xi(x)) = \xi(\tau(x))$, i.e. $\xi_{\Mt}$ is a regular vector field.

Let $\psi:\Mt\to\gg^*$ be the associated momentum map (which exists globally provided the cohomological obstruction vanishes---if it does not vanish then the map exists locally). The map $\psi$ is  defined up to a constant by the equation
\begin{equation}
    \label{eq: momentm map}
    \left<\mathrm{d}\psi_x(v),\xi\right> = \sigma_x\left(\xi_{\Mt}(x),v\right),
\end{equation}
where $v$ is any vector field on $\Mt$.

\textcolor{black}{Let $x\in\Mt$ and $v\in T_{x}\Mt$, then
\begin{eqnarray*}
    \left<\mathrm{d}\psi_{\tau(x)}(\mathrm{d}\tau(v)),\, \xi\right> &=&
    \sigma_{\tau(x)}(\xi_{\Mt}(\tau(x)),\, \mathrm{d}\tau_x(v)) \\
    &=& - \sigma_x(\xi_{\Mt}(x),\,v),
\end{eqnarray*} 
where we used the fact that $\tau^*\sigma=-\sigma$.  Now replacing $v$ by a regular vector field $V$, we see
$$\left<\mathrm{d}\psi_{\tau(x)}(V(\tau(x)),\, \xi\right> = 
    - \left<\mathrm{d}\psi_{x}(V(x)),\, \xi\right>,$$
    thereby showing that $\psi$ is twisted. }

On the phase space, the momentum map  $\Phi:\Mt^N\to\gg^*$ is the usual one given by
$$\Phi(z_1,\dots,z_N) = \sum_j\Gamma_j\,\psi(z_j).$$
Applying $\tau$ to any of the $z_j$ reverses both $\psi$ and $\Gamma_j$ and thereby leaves this expression unchanged. This momentum map therefore descends to a map (rather than a twisted map) $M\times\dots\times M\to\gg^*$.

\section{{Vortex motion on a M\"obius band}}  
 \label{sec:vortex motion Mobius}
In this section we apply all the apparatus developed above to the setting of vortex motion on the M\"obius band. We adopt the boundary-free model of the M\"obius band as a strip of fixed  width $\pi r$ and infinite height in the plane with opposite sides identified with opposite orientation: this infinite model of the M\"obius band has no boundary, so we do not need to consider vortex-boundary interactions. If we denote the coordinate on the strip as $z = x+iy$, its real part has to be considered modulo $\pi r$. 

The double cover of this model is the cylinder $\mathbb{C}/2\pi r \mathbb{Z}$ of circumference $2\pi r$ and coordinates $(x,y)$. We orient this by $dx\wedge dy$.  The imaginary line that we draw on the cylinder to signify the division (i.e., $x=0,\pi r, 2\pi r$) we will further refer to as 
\textit{the imaginary boundary}. Figure \ref{fig:Mb} depicts our model; the blue arrows in it and all following figures indicate the oppositely oriented sides of the imaginary boundary, and the dashed horizontal line is the line $y=0$. As a quotient of the cylinder, the M\"obius band is obtained by identifying $(x,y)$ with $\tau(x,y)$, where $\tau$ is the orientation reversing isometry given by $\tau(x,y)=(x+\pi r,\,-y)$.

\begin{figure}
      %{R}{0.3\textwidth}
\centering
\begin{tikzpicture}[scale=0.5]
    \draw[directed] (-5,-1) -- (-5,5);
      \draw[reverse directed] (1,-1) -- (1,5);
     %\node [up] at (1,0); 
     \draw[directed] (7,-1) -- (7,5);
      \draw[dashed] (-5,2) -- (7,2);
      \draw (-2,-0.3) circle (0.6);
\node [right] at (-1.5,0) {\scriptsize{$\Gamma$}};
\draw[decorate,decoration={markings, mark=at position 0.3cm with
    {\arrow[line width=0.3mm]{<}};}]{ (-2,-0.3) circle (0.6)};
    \draw (4,4.3) circle (0.6);
\node [right] at (4.5,4.7) {\scriptsize{$-\Gamma$}};
\draw[decorate,decoration={markings, mark=at position 0.3cm with
    {\arrow[line width=0.3mm]{>}};}]{ (4,4.3) circle (0.6)};
      \end{tikzpicture}
\caption{Two copies of the model of the M\"obius band, covered by a cylinder, with a vortex and its image under $\tau$. We refer to the vertical lines as the `imaginary boundary': they are at $x=0,\, \pi r,\, 2\pi r$.  The dashed line is $y=0$.}
\label{fig:Mb}
\end{figure}

The cylinder has the symmetry group $O(2)\times E(1)$, where $E(1)$ is the Euclidean group for the line. The subgroup consisting of those elements commuting with $\tau$ descends to the group of symmetries of the M\"obius band.  This subgroup is $O(2)\times\mathbb{Z}_2$, where $\mathbb{Z}_2$ acts by $(x,y)\mapsto (x,-y)$.

\subsection{Equations of motion}

For brevity, we will refer to the pair $(z,\Gamma)$, i.e.,  the complex planar coordinate of the vortex centre and the vortex strength, as \textit{vortex parameters}, or just \textit{parameters}.

The dynamics of a vortex with parameters $(z,\Gamma)$ will be covered by the dynamics of two vortices on the cylinder:  one with parameters $(z,\Gamma)$  and the other $(\bar{z} + \pi r,-\Gamma$). 

The Hamiltonian for the cylinder reads (see \cite{montaldi2003vortex}):
      \begin{equation*}
          \mathcal{H} = -\frac{1}{2\pi}\sum_{k<l} \Gamma_k \Gamma_l \log\left|\sin \frac{z_k-z_l}{2r}\right|,
      \end{equation*}
      or, as expressed through $x$-$y$ coordinates,
      \begin{equation*}
          \mathcal{H} = -\frac{1}{4\pi}\sum_{k<l}\Gamma_k\Gamma_l\log\left[\sin^2\left(\frac{x_k-x_l}{2r}\right) + \sinh^2\left(\frac{y_k-y_l}{2r}\right)\right].
      \end{equation*}
      
     We periodize the Hamiltonian  as in Section \ref{sec: the Hamiltonian} to obtain the explicit formula for the M\"obius band Hamiltonian:
           \begin{equation}
           \label{Ham}
      \begin{split}
          \mathcal{H} =& -\frac{1}{4\pi}\sum\limits_{k<l}\Gamma_k\Gamma_l\log\left(\sin^2\left(\frac{x_k-x_l}{2r}\right)+ \sinh^2\left(\frac{y_k-y_l}{2r}\right)\right)  \\
 &+\frac{1}{4\pi}\sum\limits_{k<l}\Gamma_k\Gamma_l\log\left(\cos^2\left(\frac{x_k-x_l}{2r}\right)+ \sinh^2\left(\frac{y_k+y_l}{2r}\right)\right)\\
          &+\frac{1}{8\pi}\sum\limits_{k}\Gamma_k^2\log\left(1+ \sinh^2\left(\frac{y_k}{r}\right)\right);
          \end{split}
      \end{equation}
     The last term is the Robin function $R_M$. We can easily see that it is unaffected by the change of the local orientation or the exchange of the covering vortices.

     Writing the system of equations on one chart, we obtain:
      \begin{equation}
      \label{Mot}
      \begin{dcases}
       \dot{x}_k =& -\frac{1}{8\pi r }\sum\limits_{l\ne k}\Gamma_l\frac{\sinh\left(\frac{y_k-y_l}{r}\right)}{\sin^2\left(\frac{x_k-x_l}{2r}\right)+ \sinh^2\left(\frac{y_k-y_l}{2r}\right)}  + \frac{1}{8\pi r }\sum\limits_{l\ne k} \Gamma_l\frac{\sinh\left(\frac{y_k+y_l}{r}\right)}{\cos^2\left(\frac{x_k-x_l}{2r}\right)+ \sinh^2\left(\frac{y_k+y_l}{2r}\right)}  \\
              &\quad {} + \frac{1}{4\pi r}\Gamma_k\tanh\left(\frac{y_k}{r}\right),\\
              \dot{y}_k =& \frac{1}{8\pi r }\sum\limits_{l\ne k}\Gamma_l\frac{\sin\left(\frac{x_k-x_l}{r}\right)}{\sin^2\left(\frac{x_k-x_l}{2r}\right)+ \sinh^2\left(\frac{y_k-y_l}{2r}\right)}  + \frac{1}{8\pi r }\sum\limits_{l\ne k}\Gamma_l\frac{\sin\left(\frac{x_k-x_l}{r}\right)}{\cos^2\left(\frac{x_k-x_l}{2r}\right)+ \sinh^2\left(\frac{y_k+
              y_l}{2r}\right)} .
              \end{dcases}
      \end{equation}
          
      %\begin{rem}
      One can see that changing the value of $r$ can be compensated by rescaling $x,y$ and time, so from now  on we assume that $r=1$.
      %\end{rem}
      
In the framework of this model, a point vortex moves as it would on an oriented manifold, as long as it does not cross the imaginary boundary. If that happens, the vortex has to `jump' to the other side  of the strip \textcolor{black}{(by decreasing the value of $x$ by $\pi$)}, change the sign of its $y$ coordinate and the sign of its strength.

This setup might seem unnatural, but the intuitive approach to vortex motion had to be sacrificed  in order to have local orientation and be able to treat the $\Gamma$s as scalars. However, this does not lead to a contradiction: crossing the imaginary boundary is the same as changing orientation of a small area surrounding the vortex, and therefore $\Gamma$ has to change its sign.

\subsection{Invariants and relative equilibria}
\begin{Def} 
For a system of point vortices on the M\"obius band, let $( z_i,\Gamma_i)$ be its covering system on the cylinder. We call the transformation that takes $(z_i,\Gamma_i)\mapsto  (\bar{z}_i - \pi,-\Gamma_i)$ the M\"obius flip.  
\end{Def}

This is the formalization of the `vortex jump' described above; however, in this case the vortex need not be on the imaginary boundary. 

 The M\"obius flip on the cylinder is the involution $\tau$ from Sec.\ \ref{sec: twisted definition}; therefore, applying it to all the point vortices must leave the Hamiltonian and the equations unchanged (this is easy to check with the explicit form of the Hamiltonian (\ref{Ham})). Additionally, it is easy to check that the local Hamiltonian equations are not affected by the change of orientation.

Observe that invariance under the M\"obius flip implies that the location of the imaginary boundary does not affect our motion; this illustrates that the equations are well-defined.

 As described above, the system on the  M\"obius band has a one-parameter group of Hamiltonian symmetries inherited from the cyliner: $SO(2)$ acting by horizontal translations, giving rise to  a conserved quantity $\Phi:=\sum_i\Gamma_iy_i$. As in the  discussion in Section \ref{sec:momentum map}, $\Phi$ is a regular function: indeed, changing the orientation changes the signs of $y_j$ together with $\Gamma_j$, and the function stays unchanged. 

         It is immediate from the form of the symmetry group that relative equilibria are horizontally moving rigid  configurations of point vortices.

\subsection{A general class of equilibria}
    \label{sec:equilibria}
It was pointed out in \cite{montaldi2003vortex} that on the cylinder, given a sequence of $2N$ point vortices placed around a horizontal circle, with vorticities of alternating signs, then for each given cyclic ordering (the order of the arrangement on the circle)  of the vortices 
there is an equilibrium point (it is not known whether or not this is unique, although numerical experiment suggests it is for $N=3$).  

For such a configuration to arise from one with $N$ vortices on the M\"obius band, the set of points must be $\tau$-invariant, and moreover $(x,\Gamma)=(\tau(x),-\Gamma)$ for all centres $x$. For this to be compatible with having alternating signs of vorticity, $N$ must be odd. 

\begin{prop}\label{prop:equatorial equilibrium}
In a preferred chart, consider the set of all configurations $\{(x_j,0)\mid j=1,\dots,N\}$ with $0\leq x_1<x_2<\dots<x_N\leq \pi$, and with vorticities of alternating signs: $\Gamma_j\Gamma_{j+1}<0$ for $j=1,\dots,N-1$ and where $N$ is odd. Then there is an equilibrium configuration in this set (see Figure \ref{fig:N=3 equilibrium}).
\end{prop}

Note that since $N$ is odd, the first and last vortices also have opposite signs when viewed in a coordinate chart that crosses the original imaginary boundary. 

\begin{proof}
Consider the domain of these configurations.  The Hamiltonian is a function of the $x_j$ which tends to $+\infty$ as the configuration tends to the boundary (i.e.\ two vortices tend to a collision). It follows that in the interior of the domain, the Hamiltonian must attain a minimum. This critical point will be an equilibrium of the system.
\end{proof}

A particular instance of this is where, in our preferred chart, the points have alternating vorticity $\Gamma,\,-\Gamma,\dots$ and placed at the vertices of a regular $N$-gon (for odd $N$).

 \begin{figure}
     \centering
\begin{tikzpicture}[scale=0.6]
    \draw[directed] (-5,-3) -- (-5,3);
    \draw[reverse directed] (5,-3) -- (5,3);
    \draw[dashed] (-5,0) -- (5,0);
\draw (-3,0) circle (0.6);
\node [right] at (-3,0.9) {\scriptsize{$\Gamma_1$}};
\draw[decorate,decoration={markings, mark=at position 0.3cm with
    {\arrow[line width=0.3mm]{>}};}]{ (-3,0) circle (0.6)};
\draw (0,0) circle (0.6);
\node [right] at (0,0.9) {\scriptsize{$\Gamma_2$}};
\draw[decorate,decoration={markings, mark=at position 0.3cm with
    {\arrow[line width=0.3mm]{<}};}]{ (0,0) circle (0.6)};
\draw (3.75,0) circle (0.6);
\node [right] at (3.75,0.9) {\scriptsize{$\Gamma_3$}};
\draw[decorate,decoration={markings, mark=at position 0.3cm with
    {\arrow[line width=0.3mm]{>}};}]{ (3.75,0) circle (0.6)};
 \end{tikzpicture}
      \caption{A typical equilibrium configuration for $N=3$. \textcolor{black}{Note that in the chosen orientation, $\Gamma_1$ and $\Gamma_3$ will be positive numbers, whereas $\Gamma_2$ will be negative. }}
     \label{fig:N=3 equilibrium}
 \end{figure}

\subsection{The N-ring relative equilibria}
\label{sec: n ring equilibria}
These configurations are more easily described on the cylinder.  Consider, on a cylinder,  a regular ring of $N\geq1$ identical vortices   with vorticity $\Gamma\neq0$ that lie  on a horizontal circle with a common vertical coordinate $y>0$. Together with this consider the `antipodal ring', the image of the first ring under the involution $\tau$, and with vorticities $-\Gamma$. It has common vertical coordinate $-y$. Altogether, there are $2N$
 vortices on the cylinder. 

Note that when $N$ is even, then the two rings are vertically aligned (both on the cylinder and on the M\"obius band), while if $N$ is odd, they are vertically staggered.  In terms of the Sch\"onflies notation for $O(3)$, for even $N$ the configuration has symmetry $D_{Nh}$ while for odd $N$ it has symmetry $D_{Nd}$ (see for example \cite[Table 2]{LMR2001} for this notation).

Since these configurations are invariant under $\tau$ they project to single rings with $N$ vortices on the M\"obius band. We call these $N$-rings; they are illustrated in Figure \ref{fig:N-rings}. If $N=1$ this is the single point vortex described in Section \ref{sec: one vortex motion} below.

\begin{figure}
      \centering
        \subfigure[A vertically aligned $N$-ring RE for even values of $N$ (here $N=6$).]{
        \begin{tikzpicture}[scale=0.6]
    \draw[directed] (-5,-3) -- (-5,3);
    \draw[reverse directed] (5,-3) -- (5,3);
    \draw[dashed] (-5,0) -- (15,0);
     \draw[directed] (15,-3) -- (15,3);
\draw (-3.333,2) circle (0.6);
\node [right] at (-3.333,2.8) {\scriptsize{$\Gamma$}};
\draw[decorate,decoration={markings, mark=at position 0.2cm with
    {\arrow[line width= 0.2mm]{>}};}]{ (-3.333,2) circle (0.6)};
\draw (0,2) circle (0.6);
\node [right] at (0,2.8) {\scriptsize{$\Gamma$}};
\draw[decorate,decoration={markings, mark=at position 0.2cm with
    {\arrow[line width= 0.2mm]{>}};}]{ (0,2) circle (0.6)};
\draw (3.333,2) circle (0.6);
\node [right] at (3.333,2.8) {\scriptsize{$\Gamma$}};
\draw[decorate,decoration={markings, mark=at position 0.2cm with
    {\arrow[line width= 0.2mm]{>}};}]{ (3.333,2) circle (0.6)};
\draw (-3.333,-2) circle (0.6);
\node [right] at (-3.333,-2.8) {\scriptsize{$-\Gamma$}};
\draw[decorate,decoration={markings, mark=at position 0.2cm with
    {\arrow[line width= 0.2mm]{<}};}]{ (-3.333,-2) circle (0.6)};
\draw (-0,-2) circle (0.6);
\node [right] at (0,-2.8) {\scriptsize{$-\Gamma$}};
\draw[decorate,decoration={markings, mark=at position 0.2cm with
    {\arrow[line width= 0.2mm]{<}};}]{ (0,-2) circle (0.6)};
\draw (3.333,-2) circle (0.6);
\node [right] at (3.333,-2.8) {\scriptsize{$-\Gamma$}};
\draw[decorate,decoration={markings, mark=at position 0.2cm with
    {\arrow[line width= 0.2mm]{<}};}]{ (3.333,-2) circle (0.6)};

\draw (10,-2) [dashed] circle (0.6);
\node [right] at (10,-2.8) {\scriptsize{$-\Gamma$}};
\draw[decorate,decoration={markings, mark=at position 0.2cm with
    {\arrow[line width= 0.2mm]{<}};}]{ (10,-2) circle (0.6)};
\draw (6.666,-2)[dashed] circle (0.6);
\node [right] at (6.666,-2.8) {\scriptsize{$-\Gamma$}};
\draw[decorate,decoration={markings, mark=at position 0.2cm with
    {\arrow[line width= 0.2mm]{<}};}]{ (6.666,-2) circle (0.6)};
\draw (6.666,2)[dashed] circle (0.6);
\node [right] at (6.666,2.8) {\scriptsize{$\Gamma$}};
\draw[decorate,decoration={markings, mark=at position 0.2cm with
    {\arrow[line width= 0.2mm]{>}};}]{ (6.666,2) circle (0.6)};
\draw (10,2) [dashed]circle (0.6);
\node [right] at (10,2.8) {\scriptsize{$\Gamma$}};
\draw[decorate,decoration={markings, mark=at position 0.2cm with
    {\arrow[line width= 0.2mm]{>}};}]{ (10,2) circle (0.6)};
\draw (13.333,2)[dashed] circle (0.6);
\node [right] at (13.333,2.8) {\scriptsize{$\Gamma$}};
\draw[decorate,decoration={markings, mark=at position 0.2cm with
    {\arrow[line width= 0.2mm]{>}};}]{ (13.333,2) circle (0.6)};
    \draw (13.333,-2)[dashed] circle (0.6);
\node [right] at (13.333,-2.8) {\scriptsize{$-\Gamma$}};
\draw[decorate,decoration={markings, mark=at position 0.2cm with
    {\arrow[line width= 0.2mm]{<}};}]{ (13.333,-2) circle (0.6)};
 \end{tikzpicture}}
 
 \bigskip

\subfigure[A vertically staggered $N$-ring RE for odd values of $N$ (here $N=5$).]{
\begin{tikzpicture}[scale=0.6]
    \draw[directed] (-5,-3) -- (-5,3);
    \draw[reverse directed] (5,-3) -- (5,3);
    \draw[dashed] (-5,0) -- (15,0);
     \draw[ directed] (15,-3) -- (15,3);
\draw (-4,2) circle (0.6);
\node [right] at (-4,2.8) {\scriptsize{$\Gamma$}};
\draw[decorate,decoration={markings, mark=at position 0.2cm with
    {\arrow[line width= 0.2mm]{>}};}]{ (-4,2) circle (0.6)};
    \draw (6,-2) [dashed]circle (0.6);
\node [right] at (6,-2.8) {\scriptsize{$-\Gamma$}};
\draw[decorate,decoration={markings, mark=at position 0.2cm with
    {\arrow[line width= 0.2mm]{<}};}]{ (6,-2) circle (0.6)};
\draw (0,2) circle (0.6);
\node [right] at (0,2.8) {\scriptsize{$\Gamma$}};
\draw[decorate,decoration={markings, mark=at position 0.2cm with
    {\arrow[line width= 0.2mm]{>}};}]{ (0,2) circle (0.6)};
    \draw (10,-2) [dashed]circle (0.6);
\node [right] at (10,-2.8) {\scriptsize{$-\Gamma$}};
\draw[decorate,decoration={markings, mark=at position 0.2cm with
    {\arrow[line width= 0.2mm]{<}};}]{ (10,-2) circle (0.6)};
\draw (4,2) circle (0.6);
\node [right] at (4,2.8) {\scriptsize{$\Gamma$}};
\draw[decorate,decoration={markings, mark=at position 0.2cm with
    {\arrow[line width= 0.2mm]{>}};}]{ (4,2) circle (0.6)};
\draw (14,-2) [dashed] circle (0.6);
\node [right] at (14,-2.8) {\scriptsize{$-\Gamma$}};
\draw[decorate,decoration={markings, mark=at position 0.2cm with
    {\arrow[line width= 0.2mm]{<}};}]{ (14,-2) circle (0.6)};
%%%
\draw (-2,-2) circle (0.6);
\node [right] at (-2,-2.8) {\scriptsize{$-\Gamma$}};
\draw[decorate,decoration={markings, mark=at position 0.2cm with
    {\arrow[line width= 0.2mm]{<}};}]{ (-2,-2) circle (0.6)};
\draw (8,2) [dashed] circle (0.6);
\node [right] at (8,2.8) {\scriptsize{$\Gamma$}};
\draw[decorate,decoration={markings, mark=at position 0.2cm with
    {\arrow[line width= 0.2mm]{>}};}]{ (8,2) circle (0.6)};
\draw (2,-2) circle (0.6);
\node [right] at (2,-2.8) {\scriptsize{$-\Gamma$}};
\draw[decorate,decoration={markings, mark=at position 0.2cm with
    {\arrow[line width= 0.2mm]{<}};}]{ (2,-2) circle (0.6)};
\draw (12,2) [dashed] circle (0.6);
\node [right] at (12, 2.8) {\scriptsize{$\Gamma$}};
\draw[decorate,decoration={markings, mark=at position 0.2cm with
    {\arrow[line width= 0.2mm]{>}};}]{ (12,2) circle (0.6)};
  \end{tikzpicture}}
 \caption{The two types of $N$-ring relative equilibrium on a cylinder and the M\"obius band. The vortices on the M\"obius band are solid, the ones on the "opposite side" of the cylinder are dashed. }        \label{fig:N-rings}
\end{figure}

\begin{theorem} 
\label{st: N ring equilibria}
On the M\"obius band, an $N$-ring is a relative equilibrium. 

Suppose that the $N$-ring in question is an aligned one: that means $N$ is an even number. Then  if the strength of the vortices in the upper row is $\Gamma$ and their vertical coordinate $y$, the angular velocity is given by 
 \begin{equation} \label{eq:xia}
\xi_a =\begin{cases}
  \displaystyle\frac{\Gamma N}{4\pi} \coth\left(\frac{Ny}{2}\right) & \text{if $N/2$ is even} \\[12pt]
 \displaystyle  \frac{\Gamma N}{4\pi}\coth(Ny) & \text{if $N/2$ is odd}. \end{cases}
\end{equation}
On the other hand, the angular velocity of staggered equilibria ($N$ odd) is given  by
\begin{equation}  \label{eq:xis}
\xi_s = \frac{\Gamma}{8\pi}\left(\tanh(y) - \coth(y)\right) + \frac{\Gamma N}{4\pi}\coth\left(2Ny\right). 
\end{equation}
\end{theorem}

A simple symmetry argument using the symmetries $D_{Nh}$ and $D_{Nd}$ mentioned above, demonstrates that such configurations are, indeed, relative equilibria. The momentum value at the ring is $N\Gamma y$, so those at different heights arise for different values of the momentum. For detailed calculations of the velocities, see Appendix \ref{sec: Appendix}.  

On the universal cover (the plane) these $N$-rings lift to von K\'arm\'an vortex streets, see for example \cite{Hally1980}.

\subsection{Motion of a single vortex}
\label{sec: one vortex motion}

    A well-known fact is that a single vortex on a cylinder, sphere or plane remains stationary.  As we see below, this is not the case in general for the M\"obius band, and the equation of motion of one vortex is derived from \eqref{Mot} and is due solely to the Robin function in \eqref{Ham}:
$$ \begin{cases}
 \dot{x} = \frac{1}{4\pi}\Gamma\tanh{y},\\
 \dot{y} = 0;
 \end{cases}$$
 showing that a vortex is stationary if and only if it lies on the line $y=0$. 
 
\begin{figure} \centering
    \includegraphics[scale =1
    ]{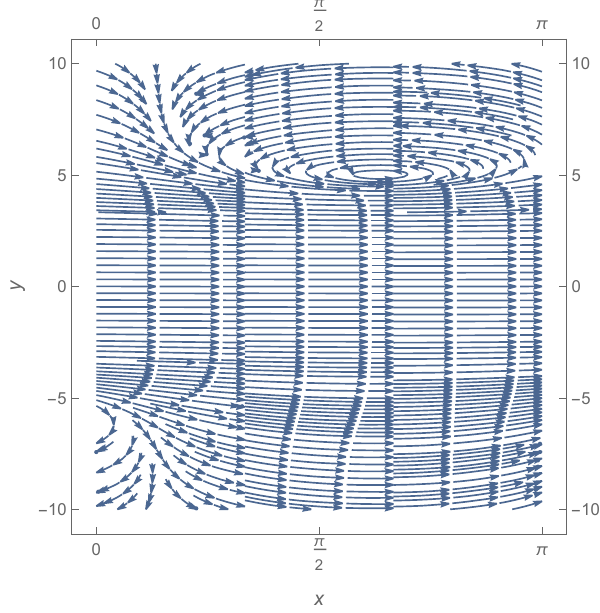}
   \caption{Instantaneous fluid flow induced by a solitary vortex on the M\"obius band.}
    \label{1vort}
\end{figure}

     From the system of equations (\ref{Mot}) we obtain the streamlines of one vortex on the band: see Figure \ref{1vort}. As time progresses, the vector field translates horizontally with the vortex.

    For this case, let us demonstrate the construction of a global vector field $(\dot{x},\dot{y})$ and investigate how the trajectories of motion come together. Due to non-orientability, the least number of covering charts for the M\"obius strip is 2, which we will denote $U_1$ ($U_1'$ when the orientation on $U_1$ is changed) and $U_2$.  
    
%      \begin{wrapfigure}{L}{0pt}
\begin{figure}
 \centering
    \includegraphics[scale =0.7]{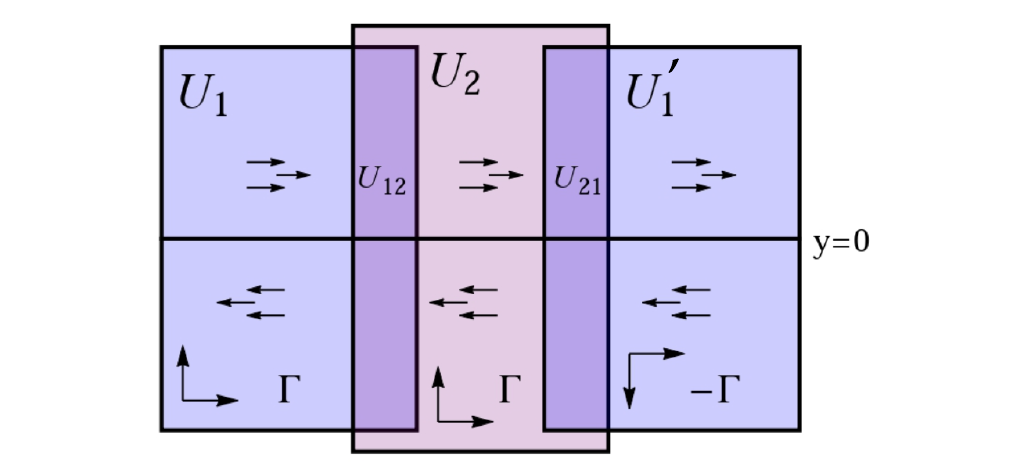}
   \caption{Two charts cover of a M\"obius band.  Horizontal arrows denote the direction of the horizontal vector field that is the solution of the Hamiltonian equations on each chart. The schematic basis at the bottom shows the orientation, and   $\Gamma$ denotes the (scalar) vorticity of the point vortex in the given orientation.  }
    \label{fig:cov}
   \end{figure}
% \end{wrapfigure} 

Let $U_{12}$ be the left intersection of the charts, and $U_{21}$ the right one as illustrated in Figure \ref{fig:cov}. The Jacobian of the coordinate change will be positive on $U_{12}$ and negative on $U_{21}$.
    
    We suppose that the vortex is initially placed on $U_1$; moving horizontally, it reaches $U_{12}$: the sign of its strength in local coordinates will not change transferring to $U_2$. Once it reaches $U_{21}$, however, its strength will change its sign, and so will its $y$-coordinate. The resulting vector field is depicted in Figure \ref{fig:cov}. 
    
  Observe that despite each of the vector fields on the charts in Figure \ref{fig:cov} being the solution of the Hamiltonian system of equations (as written on each of them) the resulting solution is \textbf{not} a vector field on the M\"obius band. However, it is (and that is how we will interpret it) a distribution of the instantaneous velocities for given initial conditions of the system. 
  
Despite this, the trajectories of motion are globally defined: due to changes in signs they can be glued together as the vortex moves from one chart to another.

\subsection{Motion of two vortices}
    \label{sec:two vortices}

The equations of motion for the system with two point vortices are:
\begin{small}
\begin{equation}
\label{eq:2vorteq}
\begin{split}
    \begin{dcases}
    \dot{x}_1 = -\frac{1}{8\pi}\Big(\Gamma_2\frac{\sinh(y_1-y_2)}{\sin^2\left(\frac{x_1-x_2}{2}\right) + \sinh^2\left(\frac{y_1-y_2}{2}\right)} - \Gamma_2\frac{\sinh(y_1+y_2)}{\cos^2\left(\frac{x_1-x_2}{2}\right) + \sinh^2\left(\frac{y_1+y_2}{2}\right)} - 2\Gamma_1\tanh(y_1)\Big)\\
     \dot{x}_2 = -\frac{1}{8\pi}\Big(\Gamma_1\frac{\sinh(y_2-y_1)}{\sin^2\left(\frac{x_2-x_1}{2}\right) + \sinh^2\left(\frac{y_2-y_1}{2}\right)} - \Gamma_1\frac{\sinh(y_2+y_1)}{\cos^2\left(\frac{x_2-x_1}{2}\right) + \sinh^2\left(\frac{y_2+y_1}{2}\right)} - 2\Gamma_2\tanh(y_2)\Big)\\
     \dot{y}_1 = \frac{\Gamma_2\sin(x_1-x_2)}{8\pi}\left(\frac{1}{\sin^2\left(\frac{x_1-x_2}{2}\right) + \sinh^2\left(\frac{y_1-y_2}{2}\right)} + \frac{1}{\cos^2\left(\frac{x_1-x_2}{2}\right) + \sinh^2\left(\frac{y_1+y_2}{2}\right)}\right)\\
       \dot{y}_2 = \frac{\Gamma_1\sin(x_2-x_1)}{8\pi}\left(\frac{1}{\sin^2\left(\frac{x_2-x_1}{2}\right) + \sinh^2\left(\frac{y_2-y_1}{2}\right)} + \frac{1}{\cos^2\left(\frac{x_2-x_1}{2}\right) + \sinh^2\left(\frac{y_2+y_1}{2}\right)}\right)
    \end{dcases}
\end{split}
\end{equation}
\end{small}
 To get an intuition about the motion, below we consider some of the simplest examples.
 
  \begin{figure} %{R}{4cm}
\centering
\begin{tikzpicture}[scale = 0.5]
    \draw[directed] (-5,-1) -- (-5,5);
      \draw[reverse directed] (1,-1) -- (1,5);
 %    \node [up] at (1,0); 
      \draw[dashed] (-5,2) -- (1,2);
   \draw[decorate,decoration={ markings,mark=at position 0cm with
{\arrow[line width= 0.2mm]{>}};}]{ (-2,3) circle (0.75)};
\draw  (-2,3) circle (0.75);
\draw[decorate,decoration={ markings,mark=at position 0cm with
{\arrow[line width= 0.2mm]{<}};}]{ (-2,1) circle (0.75)};
\draw (-2,1) circle (0.75);
\draw [->, violet] (-2,1) -- (-1.7,1);
\draw [->, violet] (-2,3) -- (-1.7,3);
\draw[fill] (-2,3 ) circle (0.01);
\draw[fill] (-2,1 ) circle (0.01);
\node [left] at (-2,3) {\small{A}};
\node [left] at (-2,1) {\small{B}};
   \end{tikzpicture}
\caption{Two vortices with opposite vortex strength and centres arranged with vertical symmetry (a 2-ring in the terminology of Section \ref{sec: n ring equilibria}) --- see Example \ref{st: exam 1}.}
   \label{fig:eleq}
  \end{figure}
  
\begin{Exam}
    \label{st: exam 1}
      Consider first, a 2-ring (as in  Section \ref{sec: n ring equilibria}), depicted in Figure \ref{fig:eleq}.

      If the initial coordinates of the centres are $(x_0, \pm y_0)$, Equation (\ref{eq:2vorteq}), as well as Theorem \ref{st: N ring equilibria}, yield  $\dot{y}_i = 0, \ \dot{x}_i= \frac{1}{2\pi }\coth\left(2y_0\right)$.
  
  In the preferred chart, both vortices will move horizontally towards the imaginary boundary, reach it, `jump' to the other side, exchanging places, and then continue to move in the same fashion.  Since the velocity of both vortex centres is the same, this configuration is an example of a relative equilibrium.
  \end{Exam}
  
  \begin{Exam}
\label{st: Example 2}
Now let the two vortices again have opposite signs: $\Gamma_1 = -\Gamma_2$, but be positioned symmetrically with respect to the line $x = \frac{\pi}{2}$.  This configuration is fixed by the corresponding reflection in $O(2)$. 

 Denote the initial positions of the centres by $(x_0, y_0)$ and $(\pi-x_0, y_0)$ respectively, and suppose that $y_0<0$ as illustrated in the leftmost part of  Figure \ref{fig:two vort opp vort period}. Substituting these values into the equations (\ref{eq:2vorteq}) easily demonstrates that the symmetric arrangement of point vortex centres is preserved throughout the motion. The restriction of the Hamiltonian to this fixed point space is 
\[
\mathcal{H} = \frac{1}{4\pi}\log\left(\frac{\cos^2x\,\cosh^2y}{\sin^2x + \sinh^2y}\right).
\]

\begin{figure}
\centering
      \begin{tikzpicture}[scale = 0.3]
    \draw[directed] (-7,-3) -- (-7,7);
      \draw[reverse directed] (1,-3) -- (1,7);
      \draw[dashed] (-7,2) -- (1,2);
   \draw[decorate,decoration={ markings,mark=at position 0.6cm with
{\arrow[line width= 0.3mm]{>}};}]{ (-5,-0.25) circle (0.7)};
\draw  (-5,-0.25) circle (0.7);
\draw[decorate,decoration={ markings,mark=at position -0.9cm with
{\arrow[line width= 0.3mm]{<}};}]{ (-1,-0.25) circle (0.7)};
\draw (-1,-0.25) circle (0.7);
\draw[fill] (1,-1 ) circle (0.01);
\draw[fill] (-7,-1 ) circle (0.01);
\node [below left] at (-4.8,-0.8 ) {\scriptsize{A}};
\node [below left] at (0.7,-0.8 ) {\scriptsize{B}};
\draw (1,-1 ) arc (-90:-270:3cm);
\draw (-7,-1) arc (-90:90:3cm);
   \end{tikzpicture} \qquad
   \begin{tikzpicture}[scale = 0.3]
    \draw (-7,-3) -- (-7,4);
    \draw[directed] (-7,4) -- (-7,7);
      \draw (1,-3) -- (1,3);
      \draw[reverse directed] (1,3) -- (1,7);
      \draw[dashed] (-7,2) -- (1,2);
   \draw[decorate,decoration={ markings,mark=at position 0.2cm with
{\arrow[line width= 0.3mm]{>}};}]{ (-7,5) circle (0.7)};
\draw  (-7,5) circle (0.7);
\draw[decorate,decoration={ markings,mark=at position 0.1cm with
{\arrow[line width= 0.3mm]{<}};}]{ (1,5) circle (0.7)};
\draw (1,5) circle (0.7);
\draw[fill] (1,-1 ) circle (0.01);
\draw[fill] (-7,-1 ) circle (0.01);
\node [below left] at (-7.2,6.7 ) {\scriptsize{A}};
\node [above right] at (1,5.2) {\scriptsize{B}};
\draw (1,-1 ) arc (-90:-270:3cm);
\draw (-7,-1) arc (-90:90:3cm);
   \end{tikzpicture}\qquad
  \begin{tikzpicture}[scale =  0.3]
    \draw[directed] (-7,-3) -- (-7,7);
      \draw[reverse directed] (1,-3) -- (1,7);
      \draw[dashed] (-7,2) -- (1,2);
   \draw[decorate,decoration={ markings,mark=at position -0.2cm with
{\arrow[line width= 0.3mm]{>}};}]{ (-7,-1) circle (0.7)};
\draw  (-7,-1) circle (0.7);
\draw[decorate,decoration={ markings,mark=at position 0.5cm with
{\arrow[line width= 0.3mm]{<}};}]{ (1,-1) circle (0.7)};
\draw (1,-1) circle (0.7);
\draw[fill] (1,-1 ) circle (0.01);
\draw[fill] (-7,-1 ) circle (0.01);
\node [below left] at (-4.8,-0.8 ) {\scriptsize{B}};
\node [below left] at (0.9,-0.9 ) {\scriptsize{A}};
\draw (1,-1 ) arc (-90:-270:3cm);
\draw (-7,-1) arc (-90:90:3cm);
   \end{tikzpicture} \qquad
   \begin{tikzpicture}[scale = 0.3]
    \draw[directed] (-7,-3) -- (-7,7);
      \draw[reverse directed] (1,-3) -- (1,7);
      \draw[dashed] (-7,2) -- (1,2);
   \draw[decorate,decoration={ markings,mark=at position 0.1cm with
{\arrow[line width= 0.3mm]{>}};}]{ (-5.6,4.7) circle (0.7)};
\draw  (-5.6,4.7) circle (0.7);
\draw[decorate,decoration={ markings,mark=at position 0.5cm with
{\arrow[line width= 0.3mm]{<}};}]{ (-0.4,4.7) circle (0.7)};
\draw (-0.4,4.7) circle (0.7);

\node [below left] at (-7.3,6.7 ) {\scriptsize{B}};
\node [above right] at (1,5.2 ) {\scriptsize{A}};
\draw (1,-1 ) arc (-90:-270:3cm);
\draw (-7,-1) arc (-90:90:3cm);
   \end{tikzpicture}
    \caption{Schematic motion of two point vortices with horizontally symmetric placement of the centres; the semi-circles are  approximate trajectories of their motion --- see Example \ref{st: Example 2}. \textcolor{black}{The pictures are arranged in chronological order: from left to right, both vortices move upwards until they mreach the imaginary boundary, ``jump" to the other side, exchange places and move upwards again, switch once more and repeat. }}
\label{fig:two vort opp vort period}
   \end{figure}
 
   Additionally, through drawing the level sets of the Hamiltonian, we can conclude that the motion goes as follows: two vortices start moving   symmetrically in opposite directions towards the symmetry line; upon crossing the line  $y =0$ they start moving away from each other  until they reach the imaginary boundary after the same finite amount of time, `jump over' to the other side, go along the trajectory that the other used to occupy, reach the boundary and exchange places once more. This resembles a relative equilibrium on the  plane, with two vortices going in a circle around their centre of vorticity, but here the trajectory is neither a circle, nor is such a configuration a  relative equilibrium.  Figure \ref{fig:two vort opp vort period} illustrates four consecutive stages of motion.
\end{Exam}

\begin{Exam}
\label{sec: Example3}
Here, we look at the asymptotic case: $y_1 = \infty$, with $\Gamma_1>\Gamma_2>0$. \textcolor{black}{Imposing this condition is equivalent to placing the first vortex infinitely high on the strip.  }  Recalling the existence of the invariant $\Gamma_1y_1 + \Gamma_2y_2 = C$ and assuming that $C$ is finite, we derive that $y_2 = -\infty$ and that $y_1 + y_2 = y_1\left(1 - \frac{\Gamma_1}{\Gamma_2}\right) + \frac{C}{\Gamma_2}\to-\infty$. \textcolor{black}{Therefore, if we assume that the invariant is finite, the second vortex must be located infinitely low}.  By taking limits of our system of equations, we get
\begin{equation}
\label{eq: examp 3}
    \begin{cases}
    \dot{x}_1  = -\frac{1}{4\pi}(2\Gamma_2 -\Gamma_1),\\
    \dot{x}_2 = -\frac{1}{4\pi}\Gamma_2 ,\\
    \dot{y}_1 = \dot{y}_2 = 0;
    \end{cases}
\end{equation}

\textcolor{black}{Firstly, we observe that both vortices remain infinitely far removed from the horizontal centre line}. \textcolor{black}{Secondly, note that } due to the difference in velocities the motion will not be a relative equilibrium. Also, note that whether the two vortices are moving in the opposite or same directions depends on the relations between $\Gamma_1$ and $\Gamma_2$.
\end{Exam}

\subsubsection{Fixed equilibria}
\label{sec:fixed equil}

We continue by determining whether the general 2-vortex system has any fixed equilibrium points. We show that these only exist when the vortex strengths are of the same sign but distinct. 

Fixed equilibria, as critical points of the Hamiltonian, correspond to zeros of (\ref{eq:2vorteq}). Observe that fixed equilibria have to have $x_1-x_2=0$. This and a  few manipulations turn (\ref{eq:2vorteq}) into
\begin{equation}
    \label{eq: fixed equilibria conditions}
    \begin{cases}
    2\Gamma_2\cosh y^{\ast}_1\cosh y^{\ast}_2 = \Gamma_1\sinh^2y^{\ast}_1 - \Gamma_1\sinh y^{\ast}_1\sinh y^{\ast}_2,\\
    2\Gamma_1\cosh y^{\ast}_1\cosh y^{\ast}_2 = \Gamma_2\sinh^2y^{\ast}_2 - \Gamma_2\sinh y^{\ast}_1\sinh y^{\ast}_2
    \end{cases}
\end{equation}
for some equilibrium values of $y$-coordinates $y_1^{\ast}$ and $y_2^{\ast}$.

Equating the right hand sides and solving as a quadratic equation for $\frac{\sinh y_2^{\ast}}{\sinh y_1^{\ast}}$ gives
\begin{equation*}
\frac{\sinh y_2^{\ast}}{\sinh y_1^{\ast}} = 1 \ \mathrm{or}    \ \frac{\sinh y_2^{\ast}}{\sinh y_1^{\ast}} = -\frac{\Gamma_1^2}{\Gamma_2^2}.
\end{equation*}

The first case is impossible, since our vortices would then occupy the same point. The second one enables us to rewrite the first equation in (\ref{eq:2vorteq}) as
\begin{equation}
    2\sqrt{\Gamma_2^4+ \Gamma_1^4\sinh^2y^{\ast}_1} = \frac{\Gamma_1}{\Gamma_2} \frac{\left(\Gamma_1^2+\Gamma_2^2\right)\sinh^2y^{\ast}_1}{\sqrt{1 + \sinh^2 y^{\ast}_1}},
\end{equation}
which implies that fixed equilibria can only occur when $\Gamma_1$ and $\Gamma_2$ are of the same sign. Solving (\ref{eq:2vorteq}) yields
\begin{equation}
\label{eq: equil cond one y1 and y2}
\begin{split}
   &y^{\ast}_1 = \pm\arcsinh\left( \sqrt{\frac{2\Gamma_2^2 \left(\Gamma_1^4+\Gamma_2^4 +\sqrt{ \Gamma_1^8+\Gamma_1^6 \Gamma_2^2+\Gamma_1^2 \Gamma_2^6+\Gamma_2^8}\right)}{\Gamma_1^2\left(\Gamma_1^2- \Gamma_2^2\right)^2}}\right)\\
   &y^{\ast}_2 =  \mp\arcsinh\left( \sqrt{\frac{2\Gamma_1^2 \left(\Gamma_1^4+\Gamma_2^4 +\sqrt{ \Gamma_1^8+\Gamma_1^6 \Gamma_2^2+\Gamma_1^2 \Gamma_2^6+\Gamma_2^8}\right)}{\Gamma_2^2\left(\Gamma_1^2- \Gamma_2^2\right)^2}}\right)
   \end{split}
\end{equation}
The expression under the outer square roots is well-defined when $\Gamma_1\ne \Gamma_2$. Thus, we have shown

%$$\sqrt{ \Gamma_1^8+\Gamma_1^6 \Gamma_2^2+\Gamma_1^2 \Gamma_2^6+\Gamma_2^8} = (\Gamma_1^2+\Gamma_2^2)\sqrt{\Gamma_1^4-\Gamma_1^2\Gamma_2^2+\Gamma_2^4\,}\quad\fbox{Useful?}$$

\begin{prop}
\label{lem: fixed equilib 2 vort}

When $\Gamma_1\Gamma_2>0$ and $\Gamma_1\ne\Gamma_2$, the system (\ref{eq:2vorteq}) is in a state of a fixed equilibrium when  $x_1 = x_2$, and $y^{\ast}_1,\, y^{\ast}_2$ are as in (\ref{eq: equil cond one y1 and y2}). When $\Gamma_1\Gamma_2<0$ or $\Gamma_1=\Gamma_2$ there is no fixed equilibrium. 
 \end{prop}

Notice from \eqref{eq: equil cond one y1 and y2} that as $\Gamma_1-\Gamma_2\to0$, the equilibrium points tend to infinity. 

For fixed distinct values of $\Gamma_1$ and $\Gamma_2$ of the same sign, we will call the two (opposite in sign) values of  $\Phi(y^{\ast}_1,y^{\ast}_2)= \Gamma_1y^{\ast}_1 + \Gamma_2y^{\ast}_2$ at fixed equilibria the \textit{equilibrium} values of the momentum map.

\subsubsection{General motion of two vortices}

In this section, we describe the more general motion of two point vortices on the M\"obius band.  To analyze our problem, we lift the two vortices of strengths $\Gamma_1$ and $\Gamma_2$ on the M\"obius band to four on the cylinder, with strengths $\pm\Gamma_1, \pm\Gamma_2$, positioned with due symmetry.
\begin{obs}
Some key points ease our computations significantly:
\begin{itemize}

\item For a fixed orientation on the cylinder, the lift of the system on the M\"obius band is unique;

  %  \item Periodicity of motion of the system on the M\"obius band is retained when lifting the system onto the cylinder, and vice versa; 
   \item The symmetric arrangement of the vortices on the cylinder is preserved by the motion, therefore on the cylinder the trajectories  of one pair of inequivalent vortices  completely determines the motion of the whole system and the motion on the M\"obius band as well. This way, we can observe only that inequivalent pair;
   
 \item Since it does not matter where we draw the imaginary boundary, we can safely assume the signs of $\Gamma_i$:  we  assume $\Gamma_1> \Gamma_2>0$.
    \item The value of the globally defined momentum map on the cylinder is twice the value of the momentum map on the M\"obius band; therefore, the $y$-coordinates of the two vortices that we choose to observe are related by
    $\Gamma_1y_1 + \Gamma_2y_2 = C$
   
\end{itemize}
\end{obs}

The next step is to reduce the system. As stated above, we assume that $\Gamma_1>0,\Gamma_2>0$ and we have an invariant $\Gamma_1y_1+\Gamma_2y_2= C$. Expressing $y_2$ through $y_1$, we substitute it into the system:
\begin{footnotesize}
\begin{equation}
\label{eq: system 2 pv with gammas}
    \begin{dcases}
    \dot{x}_1  &= -\frac{1}{8\pi}\Bigg[\Gamma_2\frac{\sinh\left(\frac{(\Gamma_2+\Gamma_1)y_1-C}{\Gamma_2}\right)}{\sin^2\left(\frac{x_1-x_2}{2}\right) + \sinh^2\left(\frac{(\Gamma_2+\Gamma_1)y_1-C}{2\Gamma_2}\right)}  -  \Gamma_2 \frac{\sinh\left(\frac{(\Gamma_2-\Gamma_1)y_1+C}{\Gamma_2}\right)}{\cos^2\left(\frac{x_1-x_2}{2}\right)+ \sinh^2\left(\frac{(\Gamma_2-\Gamma_1)y_1+C}{2\Gamma_2}\right)}\\&- 2\Gamma_1\tanh(y_1)\Bigg]\\
\dot{y}_1 &= \frac{\Gamma_2\sin(x_1-x_2)}{8\pi} \Biggl[\frac{1}{\sin^2\left(\frac{x_1-x_2}{2}\right) + \sinh^2\left(\frac{(\Gamma_2+\Gamma_1)y_1- C}{2\Gamma_2}\right)} + \frac{1}{\cos^2\left(\frac{x_1-x_2}{2}\right)+ \sinh^2\left(\frac{(\Gamma_2-\Gamma_1)y_1+ C}{2\Gamma_2}\right)}\Biggr]\\
        \dot{x}_2 &= -\frac{1}{8\pi}\Bigg[-\Gamma_1\frac{\sinh\left(\frac{(\Gamma_2+\Gamma_1)y_1-C}{\Gamma_2}\right)}{\sin^2\left(\frac{x_1-x_2}{2}\right) + \sinh^2\left(\frac{(\Gamma_2+\Gamma_1)y_1-C}{2\Gamma_2}\right)} - \Gamma_1 \frac{\sinh\left(\frac{(\Gamma_2-\Gamma_1)y_1+ C}{\Gamma_2}\right)}{\cos^2\left(\frac{x_1-x_2}{2}\right)+ \sinh^2\left(\frac{(\Gamma_2-\Gamma_1)y_1+ C}{2\Gamma_2}\right)}\\&- 2\Gamma_2\tanh\left(\frac{C-\Gamma_1y_1}{\Gamma_2}\right)\Bigg]\\
        \dot{y}_2 &= -\frac{\Gamma_1}{\Gamma_2}\dot{y}_1
        \end{dcases}
        \end{equation}

\end{footnotesize}

    We can reduce the system by one more degree of freedom: the right hand sides of all equations depend only on $x_1-x_2$ and $y_1$. We subtract the third equation from the first to get a system in two variables:
    
\begin{footnotesize}
    \begin{equation}
        \begin{dcases}
        \dot{x}_1-\dot{x}_2 &= -\frac{1}{8\pi}\left[\left(\Gamma_2 + \Gamma_1\right)\frac{\sinh\left(\frac{(\Gamma_2+\Gamma_1)y_1-C}{\Gamma_2}\right)}{\sin^2\left(\frac{x_1-x_2}{2}\right) + \sinh^2\left(\frac{(\Gamma_2+\Gamma_1)y_1-C}{2\Gamma_2}\right)}\right. \\
        &-\left.\left(\Gamma_2-\Gamma_1\right)\frac{\sinh\left(\frac{(\Gamma_2-\Gamma_1)y_1+ C}{\Gamma_2}\right)}{\cos^2\left(\frac{x_1-x_2}{2}\right) +
        \sinh^2\left(\frac{(\Gamma_2-\Gamma_1)y_1+ C}{2\Gamma_2}\right)} - 2\Gamma_1\tanh(y_1) + 2\Gamma_2\tanh\left(\frac{C - \Gamma_1y_1}{\Gamma_2}\right)\right]\\
       \dot{y}_1 &= \frac{\Gamma_2\sin(x_1-x_2)}{8\pi} \Bigg[\frac{1}{\sin^2\left(\frac{x_1-x_2}{2}\right) +
        \sinh^2\left(\frac{(\Gamma_2+\Gamma_1)y_1-C}{2\Gamma_2}\right)} + \frac{1}{\cos^2\left(\frac{x_1-x_2}{2}\right)+ \sinh^2\left(\frac{(\Gamma_2-\Gamma_1)y_1+ C}{2\Gamma_2}\right)}\Bigg]
        \end{dcases}    
    \end{equation}
\end{footnotesize}

The same reduction technique can be applied to the Hamiltonian, giving
\begin{equation}
\label{Hamred}
    \begin{split}
        \mathcal{H} &=-\frac{\Gamma_1\Gamma_2}{4\pi}\log\left(\frac{\sin^2\left(\frac{x_1-x_2}{2}\right) + \sinh^2\left(\frac{y_1}{2}\left(1 + \frac{\Gamma_1}{\Gamma_2}\right) - \frac{C}{2\Gamma_2}\right)}{\cos^2\left(\frac{x_1-x_2}{2}\right) + \sinh^2\left(\frac{y_1}{2}\left(1 - \frac{\Gamma_1}{\Gamma_2}\right) + \frac{C}{2\Gamma_2}\right)}\right) \\ &\quad+ \frac{\Gamma_1^2}{4\pi}\log\left(\cosh(y_1)\right) + \frac{\Gamma_2^2}{4\pi}\log\left(\cosh\left(\frac{C-\Gamma_1y_1}{\Gamma_2}\right)\right)
    \end{split}
\end{equation}
%\end{center}

Henceforth, our main goal is to reconstruct the motion of the initial system from the information we are able to obtain about the reduced one.

We start with searching for singular and critical points of the Hamiltonian. Without loss of generality, we may suppose that $-\pi\le x_1-x_2\le\pi$. However, in the pictures we draw four periods, to highlight the symmetries that the trajectories possess. 

Noticing that the partial derivative $\frac{\partial\mathcal{H}}{\partial (x_1-x_2)}$ is proportional to $\sin(x_1-x_2)$ multiplied by some strictly negative function, we deduce that all the critical points lie on the lines $x_1-x_2 = 0, \pm \pi$. Additionally, the Hamiltonian is $2\pi$-periodic in $x_1-x_2$ and symmetric with respect to the vertical axes listed above.

Assessing asymptotic behaviour of the function with a fixed $x_1-x_2$ and  $y_1\to\pm \infty$ gives  $\mathcal{H}\sim (\Gamma_1^2-\Gamma_1\Gamma_2)|y|$.  Therefore, $\mathcal{H}\to+\infty$ if $y_1\to\pm\infty$.

\begin{figure}
\centering
\subfigure[Level set curves]{
    \includegraphics[scale=0.5]{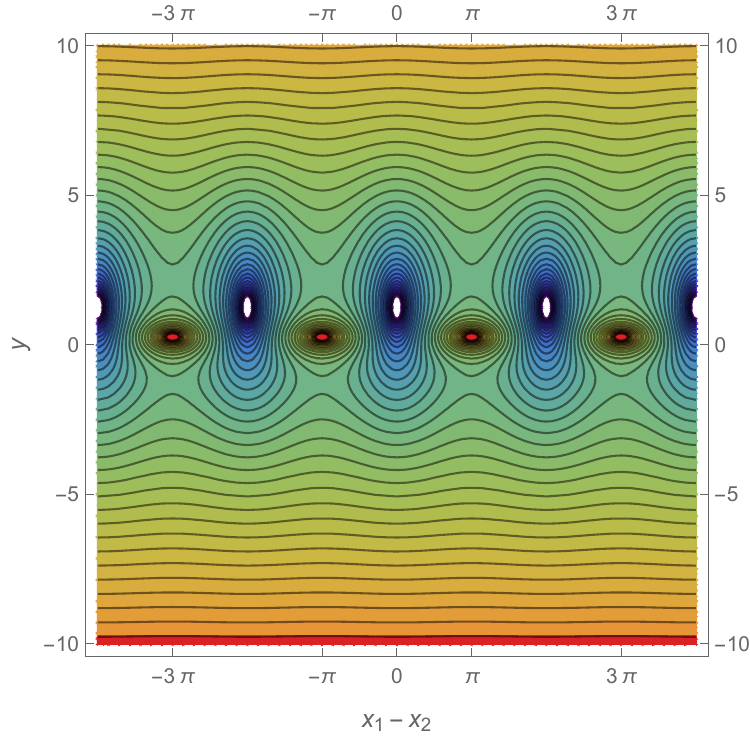}
    }
\subfigure[Vector plot]{
     \includegraphics[scale=0.5]{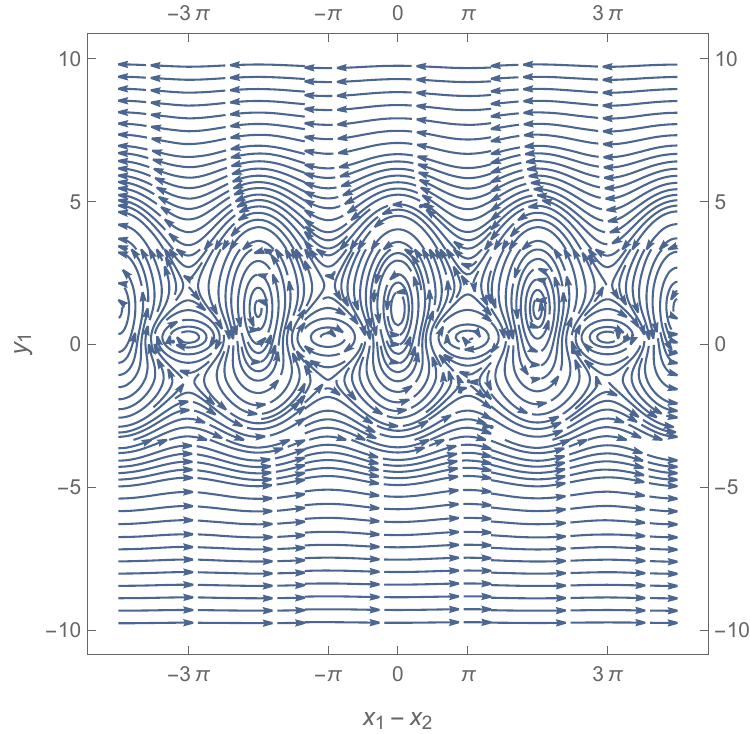}
     }
    \caption{Level sets of the reduced Hamiltonian with the fewest number of critical points}
    \label{fig:Mobius-least}%my_label}
\end{figure}

\begin{itemize}
    \item $x_1-x_2 = 0$.
    \begin{equation*}
        \begin{split}
            \mathcal{H} =& -\frac{\Gamma_1\Gamma_2}{2\pi}\log\left(\frac{\sinh^2\left(\frac{y_1}{2}\left(1 + \frac{\Gamma_1}{\Gamma_2}\right) - \frac{C}{2\Gamma_2}\right)}{\cosh^2\left(\frac{y_1}{2}\left(1 - \frac{\Gamma_1}{\Gamma_2}\right) + \frac{C}{2\Gamma_2}\right)}\right)\\
            &+\frac{\Gamma_1^2}{\pi}\log(\cosh(y_1)) + \frac{\Gamma_2^2}{\pi}\log\left(\cosh\left(\frac{C-\Gamma_1y_1}{\Gamma_2}\right)\right)
        \end{split}
    \end{equation*}
       Obviously, $y_1 = \frac{C}{\Gamma_1 + \Gamma_2}$ is a singular point, as the point  of collision of the vortices $\Gamma_1$ and $\Gamma_2$; here, $\mathcal{H}\to +\infty$. 
    
          As we have stated above, when $y_1\to \pm\infty$, $\mathcal{H}\to +\infty$, which means that on the line $x_1-x_2=0$ there must exist at least two more critical points, both saddles: maximum in $x_1-x_2$, minimum in $y$. 
            
       Both of these points are relative equilibria  when $C$ is not an equilibrium value; for each of two equilibrium values of $C$ one of them turns into a fixed equilibrium.
       
       Through estimates on the first and second derivatives of $\dot{x}_1-\dot{x}_2$ as a function of $y_1$ it can be shown that no more  critical points exist on this line. 
       \item $x_1-x_2 = \pm\pi$\\
       
       Here, the simplified Hamiltonian becomes 
       \begin{equation*}
        \begin{split}
           \mathcal{H} =& -\frac{\Gamma_1\Gamma_2}{\pi}\log\left(\frac{\cosh^2\left(\frac{y_1}{2}\left(1 + \frac{\Gamma_1}{\Gamma_2}\right) - \frac{C}{2\Gamma_2}\right)}{\sinh^2\left(\frac{y_1}{2}\left(1 - \frac{\Gamma_1}{\Gamma_2}\right) + \frac{C}{2\Gamma_2}\right)}\right) \\
           &\quad+\frac{\Gamma_1^2}{\pi}\log\left(\cosh(y_1)\right) + \frac{\Gamma_2^2}{\pi}\log\left(\cosh\left(\frac{C-\Gamma_1y_1}{\Gamma_2}\right)\right).
        \end{split}
    \end{equation*}
      $y_1 = \frac{C}{\Gamma_1-\Gamma_2}$ will always be a singular point, at which the Hamiltonian will be $-\infty$. Here, the collision happens between the vortices $\Gamma_1$ and $-\Gamma_2$.
   
\end{itemize}

   Configurations as depicted in Figure \ref{fig:Mobius-least} %my_label} 
   occur for certain values of the $\Gamma_i$ and $C$; however, another possibility exists, shown in Figure \ref{fig:my_label2}. In that case, the reduced Hamiltonian  has an additional saddle point and a minimum below it. Observe that both are instances of relative equilibria that have the $\Gamma_1$ and $-\Gamma_2$ vortices on a vertical line above each other; however, as in Proposition  \ref{lem: fixed equilib 2 vort} they cannot be fixed equilibria, as $\Gamma_1$ and $\Gamma_2$ have different signs on the chart on the M\"obius band.  
   
   It seems (though the rigorous proof presents too big a computational challenge) that these are the two only possibilities. In principle, there could be any number of the  saddle-minimum pairs on the lines $x_1-x_2 = \pm\pi$, however this appears not to be the case.

 For now we assume that  $C\ne 0$ (we address the case of $C=0$ separately below) and that the level sets of the Hamiltonian are similar to the ones in Figure \ref{fig:Mobius-least}, %my_label}, 
 i.e.\ the Hamiltonian has no additional critical points on the lines $x_1-x_2 = 0, \pm\pi$. 

\begin{figure}
\centering
\subfigure[Level set curves]{
    \includegraphics[scale =0.55]{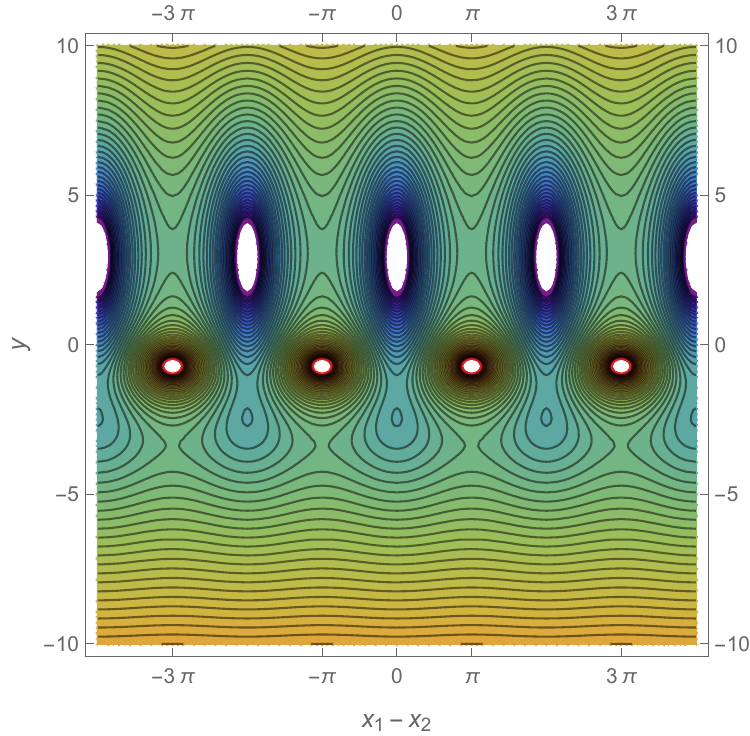}
    }
\subfigure[Vector field plot]{
     \includegraphics[scale =0.55]{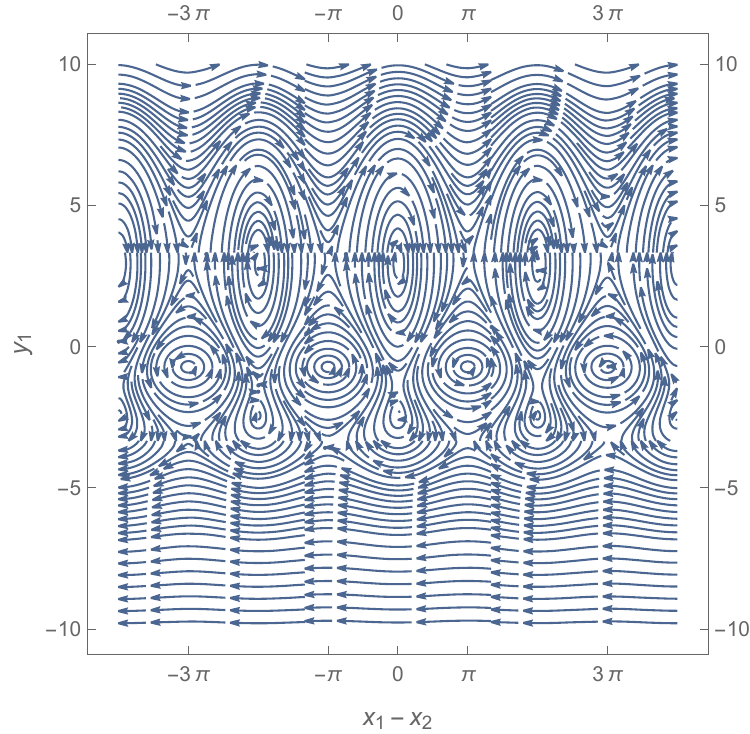}
     }
    \caption{Level sets of the reduced Hamiltonian with additional critical points}
    \label{fig:my_label2}
\end{figure}

The trajectories as given by the level sets of the reduced Hamiltonian lack the information about  the rotation around a cylinder. In what follows, we attempt to restore the motion of the system from the information provided to us in Figure \ref{fig:Mobius-least}.
%my_label}.

Three types of curves are present in the picture: the closed trajectories going around the points with $(x_1-x_2, y_1)$ coordinates $(2k\pi, \frac{C}{\Gamma_1 + \Gamma_2})$ they fill the part of the plane that we will refer to as Region I), closed trajectories around the points of type $((2k+1)\pi, \frac{C}{\Gamma_1 - \Gamma_2})$ (Region II) and  the rest (Region III). The motion of the system will then be described by the 

\begin{theorem}
\label{st: two vortex motion on Mobius band}
For two point vortices with $\Gamma_1,\Gamma_2,C$ such that the Hamiltonian has the level sets as in Figure \ref{fig:Mobius-least}, %my_label}, 
we have the following:
\begin{itemize}
       \item If the initial coordinates are  in  Regions I or II, the  two vortices in consideration will rotate around each other, while simultaneously moving forward on the M\"obius band as a pair. Due to continuous dependence of the integral, this motion will be periodic on a set of trajectories of planar measure 0.
       \item On trajectories of Type III, the vertical and the horizontal distances between the two vortices will change with a certain period, however, they will not rotate around each other.The two vortices may move in the same or opposite directions (both cases occur). \end{itemize}
\end{theorem}

\begin{proof}
Note that Region I, Region II and Region III notation is purely a matter of convenience, in order to distinguish between three different types of behaviour.

In Regions I and II, the two vortices are rotating around each other; that is not the case for the Region III. In order to reconstruct the motion fully, however, we need some additional information, which we will be obtaining from certain limiting cases. 

 From the system (\ref{eq: system 2 pv with gammas}) we observe that on every curve $\mathcal{C} = \{x_1-x_2(t),y(t)\}$ $\dot{x}_1$ (as well as $\dot{x}_1 - \dot{x}_2$ and $\dot{y}_1$) is a function of $x_1-x_2$ and $y_1$. In order to restore the motion,  we need to compute $ \int\limits_0^{T} \dot{x}_1\mathrm{d}t$, where $T$ is the period of motion on the curve $C$.  We rewrite it tautologically as $   \int\limits_{\mathcal{C}} \frac{\dot{x}_1(x_1-x_2,y_1)}{\sqrt{(\dot{x}_1-\dot{x}_2)^2+(\dot{y}_1)^2}}\mathrm{d}s$, with $\mathrm{d}s$ the element of length on the curve $\mathcal{C}$. This quantity tells us how much the first vortex moves on the cylinder with every period of motion.
 
 The motion of the system will be periodic in two cases: when the integral is 0, or, on a cylinder with circumference $2\pi$, has the form $\frac{p}{q}\pi,\  p,\ q\ \in\mathbb{Z}$.
 
  Firstly, we try to establish whether the integral in question is always equal to zero. For Region I, we suppose that our trajectory is very close to the point $\left(0, \frac{C}{\Gamma_1 + \Gamma_2}\right)$ (for Region II, it is the point $\left(\pm\pi,\frac{C}{\Gamma_1-\Gamma_2}\right)$). On these trajectories, the pair of vortices that are very close together will mimic (to the rest of the system) the behaviour of one vortex with vorticity $\Gamma_1 + \Gamma_2$ ($\Gamma_1 - \Gamma_2)$, see baby vortices section in \cite{montaldi2003vortex}.
  
  In the case of four vortices this means that each pair of two closely placed vortices behaves like one bigger vortex. Therefore, the motion is very closely approximated by that of two vortices of opposite strengths on  each side of the cylinder. This construction is stationary if and only if their vertical coordinates are zero, which is not the case, since $C\ne 0$.  Therefore, the value of the integral is not identically equal to zero and the translational component is present in the motion of the system. 
    
    In order to proceed, we require a few lemmas:
    \begin{lem}
  $\int\limits_C \frac{\dot{x}_1}{\sqrt{(\dot{x}_1-\dot{x}_2)^2+(\dot{y}_1)^2}}\mathrm{d}s$ for all curve types I, II and III continuously depends on the curve $C$.
  \end{lem}
   \begin{proof}
           This is quite straightforward, as we can take any bounded subregion $D$ that contains no critical points of the function but fully contains curves that intersect it and consider two close curves $\mathcal{C}$ and $\mathcal{C}'$ within it. Then \begin{equation*}
            \begin{split}   
          &\left|\int\limits_{\mathcal{C}} \frac{\dot{x}_1}{\sqrt{(\dot{x}_1-\dot{x}_2)^2+(\dot{y})^2}}\mathrm{d}s - \int\limits_{\mathcal{C}'} \frac{\dot{x}_1}{\sqrt{(\dot{x}_1-\dot{x}_2)^2+(\dot{y})^2}}\mathrm{d}s\right|<\\
          &\qquad<\max\left|\frac{\dot{x}_1}{\sqrt{(\dot{x}_1  - \dot{x}_2)^2  + \dot{y}_1^2}}\right|\left(\int\limits_{\mathcal{C}}\mathrm{d}s - \int\limits_{\mathcal{C}'}\mathrm{d}s \right)\ \to \ 0,
          \end{split}
  \end{equation*}
  since all the functions $\frac{\partial \mathcal{H}}{\partial(x_1-x_2)}$, $\frac{\partial \mathcal{H}}{\partial y_1}$, $\mathcal{H}$, as well as tangent vectors to the curves are bounded due to the choice of $D$
   \end{proof}
   
   \begin{lem}
The integral    $\int\limits_C \frac{\dot{x}_1}{\sqrt{(\dot{x}_1-\dot{x}_2)^2+(\dot{y}_1)^2}}\mathrm{d}s$ depends non-trivially on the curve $C$.
   \end{lem}
   \begin{proof}
           
          First, we establish this for  Region I (for II, the statement can be proven similarly).
           
           Suppose the configuration is very close to the singular point inside Region I, i.e.\  the two vortices are very close to each other. Therefore, when calculating the speed of rotation around each other, we may disregard the influence of the opposite pair of vortices, and the time required for the vortices to go one full circle will be 
           \[
           T \sim\epsilon^2 (\sim \epsilon, \ \mathrm{in  \ Region \ II}),
           \]
           where $\epsilon$ is the distance between vortices. The horizontal velocity of the vortex pair will be approximately            \[
           \dot{x}_1 \sim \frac{\Gamma_2 + \Gamma_1}{2\pi}\tanh(y_1).
           \]
           The rotational speed is dependent on $\epsilon$, unlike $\dot{x}_1$. Therefore, depending on the initial point of the trajectory, a very different number of full turns will ``fit" into a pair going full circle around the cylinder. Thus, the integral depends non-trivially on the trajectory.

       For Region III two different cases exist: when $C$ is not an equilibrium value and when it is.

       Suppose $C$ is not an equilibrium value; therefore, the two saddle critical points on the line $x_1-x_2=0$ are relative equilibria. Since the motion depends continuously on the initial parameters, $\dot{x}_1\ne 0$ on the trajectories near separatrices (and, consequently, relative equilibrium points). However, as we approach relative equilibria, $T\to\infty$ and, therefore, $\int\limits_0^T\dot{x_1}\mathrm{d}t\to \infty$.

       When $C$ is an equilibrium value, non-triviality stems from different values of the integrals at two asymptotic cases. Suppose $C>0$ and the $y$-coordinate of the relative equilibrium is greater than 0 as well (the opposite case can be tackled in similar way). As we approach the separatrix containing the equilibrium point, $\int\limits_0^T\dot{x}_{1(2)}\mathrm{d}t\to0$ , since the motion smoothly depends on the initial conditions. On the other hand, take $y_1\to +\infty,  \ y_2\to-\infty $ (as in Example 3). Since the velocities are given by (\ref{eq: examp 3}) and integral of $\dot{x}_1-\dot{x}_2$ over the period $T$ must be $2\pi$, 
       \[
       T = \frac{8\pi^2}{\Gamma_1 - \Gamma_2}.
       \]
       Thus, when $y_1\to +\infty,  \ y_2\to-\infty $,  $\int\limits_0^T\dot{x}_{1}\mathrm{d}t\to -\frac{2\pi\left(2\Gamma_2 - \Gamma_1\right)}{\Gamma_1 - \Gamma_2}$ and $\int\limits_0^T\dot{x}_{2}\mathrm{d}t\to -\frac{2\pi\Gamma_2}{\Gamma_1 - \Gamma_2}$, at least one of which is not equal to 0. 
        \end{proof}

        Thus, for Regions I and II the integral  $\int\limits_0^T \dot{x}_1 \ \mathrm{d}t $ will almost never be a rational  multiple of  $\pi$, and the motion will consist of rotation around each other and translating around the cylinder.

        In Region III, when $2\Gamma_2\ge \Gamma_1$, the pair of infinitely removed vortices will rotate in the same direction (meaning that $\int\limits_0^T \dot{x}_1 \ \mathrm{d} t$ and $\int\limits_0^T \dot{x}_2 \ \mathrm{d} t$ over the period of motion will have the same signs). When $2\Gamma_2\ge \Gamma_1$, the directions of rotation are opposite (with the opposite signs of the two integrals). From the non-trivial and smooth dependence on the trajectory we conclude that it is almost never a rational number multiplied by $\pi$ either.

        Therefore, two types of motion can exist: when the two integrals have the same sign, the particles are moving in the same direction, and a difference between the horizontal components in their velocities allows the first integral to differ by $\pm 2\pi$.  The picture is very much the same when they are moving in the opposite directions.

        However, there seems to be nothing that would prevent one of the integrals from turning zero and then switching the sign: in this case the motion will turn periodic, with one vortex rotating and the other going around the cylinder. This might allow for the transfers between the two types of motion in the region.
\end{proof}
    
    If additional critical points (Figure \ref{fig:my_label2}) appear on lines $x_1-x_2 = \pm\pi$,  the motion around them will be of the same type as for Regions I and II. Additional non-closed trajectories will result in motion of the same type as in trajectories in Region III; the required proof of non-trivial and smooth dependence on the trajectory can be repeated  verbatim.  
%   \begin{wrapfigure}{R}{0pt}
\begin{figure} \centering
    \includegraphics[scale =0.5]{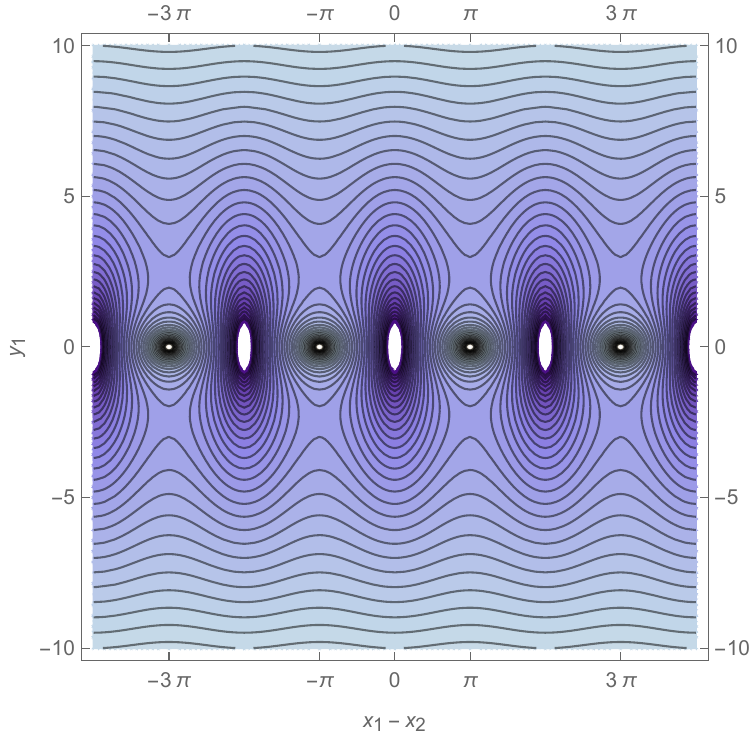}
   \caption{Level sets of the reduced Hamiltonian with zero value of momentum map}
    \label{C0}
    \end{figure}
    %\end{wrapfigure} 
   
 Two cases remain unaddressed: those of $C=0$ and $\Gamma_1=\Gamma_2$. We start with the first one. 
\vskip .2cm
\noindent\textbf{Zero momentum}
\vskip .2cm%{\texorpdfstring{$C=0$}{C=0}}
 Putting $C=0$ makes (\ref{Hamred}) a symmetric function of $y_1$, as can be easily observed from  Figure \ref{C0}. Thus, for a  curve $\{(x_1-x_2)(t),\, y_1(t)\}$ with period $T$ in Regions I and II we have  $(x_1-x_2)(t+ \frac{T}{2}) =  - (x_1-x_2)(t), \: y_1(t+ \frac{T}{2}) = -y_1(t)$, giving 
 \begin{equation}
 \label{eq:xero int}
       \begin{split}
           \int\limits_0^T \dot{x}_1((x_1-x_2)(t), y_1(t)) \mathrm{d}t &= \int\limits_0^{\frac{T}{2}} \dot{x}_1((x_1-x_2)(t), y_1(t)) \mathrm{d}t+ \int\limits_{\frac{T}{2}}^T \dot{x}_1((x_1-x_2)(t), y_1(t)) \mathrm{d}t  \\
           &= \int\limits_0^{\frac{T}{2}} \dot{x}_1((x_1-x_2)(t), y_1(t)) \mathrm{d}t+ \int\limits^{\frac{T}{2}}_0 \dot{x}_1(-(x_1-x_2)(t), -y_1(t)) \mathrm{d}t  \\
           &= \int\limits_0^{\frac{T}{2}} \dot{x}_1((x_1-x_2)(t), y_1(t)) \mathrm{d}t- \int\limits^{\frac{T}{2}}_0 \dot{x}_1((x_1-x_2)(t), y_1(t)) \mathrm{d}t \\ &=0,
       \end{split}
   \end{equation}
     owing to the explicit form of $\dot{x}_1$. Hence, the vortex pair does not rotate around the cylinder; the motion consists solely of the two vortices rotating around each other.  However, for Region III the motion does not differ from the general case.

\vskip .2cm
\noindent
$\mathbf{\Gamma_1 = \Gamma_2}$
\vskip .2cm
 When $\Gamma_1 = \Gamma_2$, the major difference occurring is that  $\mathcal{H}$ has a finite limit  when $y\to\pm\infty$. As can be seen in Figure \ref{fig:G1=G2}, the picture is symmetric, but now the symmetry is with respect to the line $y = \frac{C}{2}$, robbing us of the zero integral as in (\ref{eq:xero int}). Additionally, the critical point at $\frac{C}{\Gamma_1-\Gamma_2}$ disappears, leaving us with trajectories of Types I and III only. 
 
 For Region I, we may employ precisely the same reasoning as we have before, for the non-zero $C$ case. Region III, however, requires more caution. 
 
 If we set $\Gamma_1 = \Gamma_2$ in Example 3, the velocities of the two point vortices will coincide. This contradicts the fact that $\int\limits_0^T \dot{x}_1 -\dot{x}_2 \ \mathrm{d}t = \pm2\pi $.
 
 However, this paradox can be explained: as we have mentioned above, with $\Gamma_1>\Gamma_2$ and $y_1\to\pm\infty, \ \mathcal{H}\to \infty$ for all values of $x_1-x_2$. This is not the case when $\Gamma_1 = \Gamma_2$: here the limit of $\mathcal{H}$ with $y_1\to\pm\infty$ will be equal to $\frac{1}{\cos^2\left(\frac{x_1-x_2}{2}\right) + \cosh(C)}$. Therefore, the `level set' of the Hamiltonian at $\pm\infty$ is not well-defined. 
 
 But we have another limiting case to draw the information from: when $x_1-x_2= \pm \pi$, we have precisely one relative equilibrium point: $y_1=-y_2=\frac{C}{2}$, where two vortices with opposite vorticity are above each other. As we have demonstrated in Example 2, they will move parallel to each other in the same direction. Therefore, on the trajectories very close to the separatrix we will have the two point vortices moving in the same direction.

 \begin{rem}
   When $\Gamma_1 = \Gamma_2$ and $C=0$ (see Example \ref{st: Example 2}), we do not have Region III; as can be checked, the Hamiltonian turns infinite in the lines $\frac{x_1 - x_2}{2} = \pm\pi$.This happens due to the fact that the separatrix from Region II stretches and goes to infinity as $\Gamma_2\to\Gamma_1$ or vice versa. Thus, in those cases, all the motion is periodic and without additional rotation, as  was demonstrated explicitly above.
 \end{rem}
   
%  \begin{wrapfigure}{L}{0pt}
\begin{figure} 
\centering
    \includegraphics[scale =0.5]{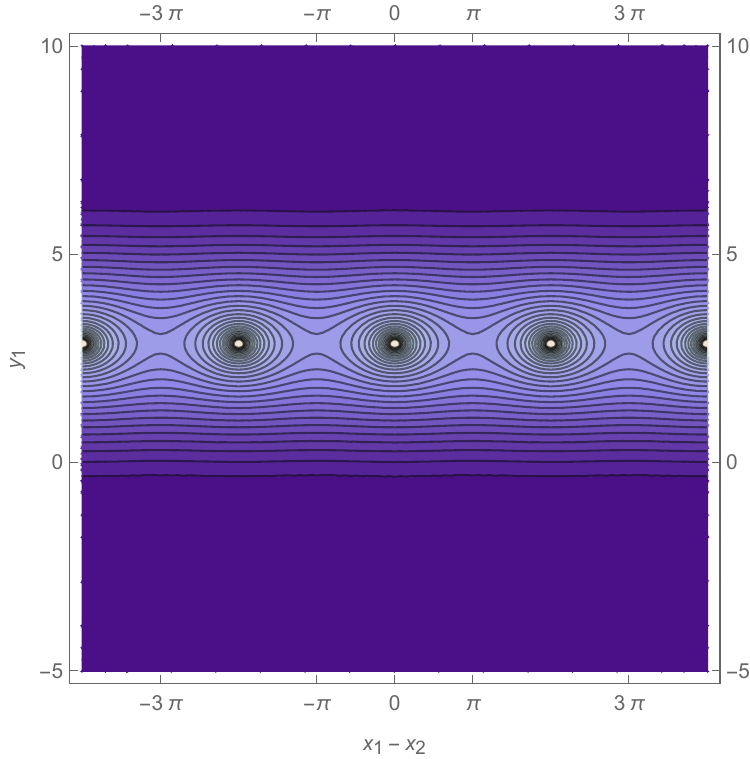}
   \caption{\small{Level sets of the reduced Hamiltonian for the case  $\Gamma_1 = \Gamma_2$}}
    \label{fig:G1=G2}
 %   \end{wrapfigure} 
\end{figure}
\section{Motion on the Klein bottle}
\label{sec: Klein bottle general}
In this section, we will discuss the motion of point vortices on the surface of the Klein bottle. 

The setup is almost identical to the one for the Mobius band, except for one important detail: the model of the Mobius strip that we employed above was a non-compact surface;  contrastingly, the Klein bottle is a compact, closed manifold.  

This difference manifests in an alteration that needs to be done to the vorticity form: when a manifold $\widetilde{M}$ (which is, as before, the double cover of some other manifold  $M$) is compact, the form of the vorticity has to be such that  $\iint\limits_{\Mt} \omega = 0$ (see \cite{dritschel2015motion}). 
This is achieved through subtracting the inverse of the area $A$ of $\widetilde{M}$: therefore, in the compact case the  vorticity will be 
\begin{equation}
\label{eq: vorticity for compact}
\omega(x) = \Gamma\left(\delta_{x_0}(x) - \frac{1}{A}\right)
\end{equation}

If, using the double cover as above, we place a point vortex of strength $\Gamma$ at a point $x\in M$ with preimages $x,\widetilde{x}\in\widetilde{M}$,  the total vorticity of the fluid on $\widetilde{M}$  will be $\Gamma\delta_{x}  -\Gamma/A - \Gamma\delta_{\widetilde{x}} + \Gamma/A  = \Gamma\delta_{x} - \Gamma\delta_{\widetilde{x}}$ , the integral of which over $\widetilde{M}$ can be easily seen to be 0.

\subsection{The Hamiltonian and equations of motion}
\label{sec: the Ham and motion equations Klein}
 We adopt different methods of periodisation to determine the form of the Hamiltonian on the Klein bottle. Our go-to model is a $\pi$-by-$\pi$ square that has a $2\pi$-by-$\pi$ torus as its double cover (such as is shown in Figure \ref{fig:model of the Klein bottle}). In this context, we refer to the oppositely oriented  sides as the \defn{vertical imaginary boundary}  and to the two sides with the same orientation as the  \defn{horizontal imaginary boundary}.

\subsubsection{Jacobi Theta functions}

We carry out the calculations in a similar manner to those for the Mobius band and the cylinder; however, unlike in the previous cases, the manner in which we perform these calculations depends on the cover that we choose.

 We introduce the functions that we will be using; for details on the topic past the definition and some initial properties, see, for example, \cite{whittaker2020course, bellman2013brief}.

\begin{Def}
\label{def: theta functions}
Jacobi theta functions are the quasi-doubly periodic functions of two complex arguments $z$ and $q$, given by the formulae:
\begin{equation}
    \begin{split}
    &\theta_1(z,q):=\sum\limits_{n=-\infty}^{+\infty}(-1)^{n-\frac{1}{2}}q^{\left(n+\frac{1}{2}\right)^2}e^{(2n+1)iz} = 2\sum\limits_{n=0}^{+\infty}(-1)^{n}q^{\left(n+\frac{1}{2}\right)^2}\sin((2n+1)z)\\
        &\theta_2(z,q):=\sum\limits_{n=-\infty}^{+\infty}q^{\left(n+\frac{1}{2}\right)^2}e^{(2n+1)iz} = 2\sum\limits_{n=0}^{+\infty}q^{\left(n+\frac{1}{2}\right)^2}\cos((2n+1)z)\\
      &\textcolor{black}{\theta_3(z,q) = \sum\limits_{n=-\infty}^{\infty}q^{n^2}e^{2 n  i z} = 1 +2\sum\limits_{n=1}^{\infty}q^{n^2}\cos(2 nz)}\\
        &\theta_4(z,q):=\sum\limits_{n=-\infty}^{+\infty}(-1)^{n}q^{n^2}e^{2niz} = 1+2\sum\limits_{n=1}^{+\infty}(-1)^{n}q^{n^2}\cos(2nz)
    \end{split}
\end{equation}
The functions and their properties can also be expressed in  terms of the number $\mu$, such that  $q = e^{i\pi\mu}$. When $q$ (and, consequently, $\mu$)  is fixed and its value is clear from the context, we will refer to theta functions as $\theta_i(z)$. 
\end{Def}
 Let $\theta'(z,q)$ be the derivative of $\theta(z,q)$ with respect to $z$. It is easy to see that the following relations hold: 
\begin{equation}
\centering
 \label{eq: Jac th element props}
\begin{split}
     &\theta_{1}(-z,q) = -\theta_{1}(z,q),  \\
     &\theta'_{1}(-z,q) = \theta'_{1}(z,q),\\
  & \theta_{1}(\bar{z},\bar{q}) = \overline{\theta_{1}(z,q)},
  \end{split}\ \ \  \ \ \ \ \ \ \ \ \ \ \ \ \ \ \ 
  \begin{split}
&\theta_{2}(-z,q) = \theta_{2}(z,q),\\
&\theta'_{2}(-z,q) = -\theta'_{2}(z,q),\\
&\theta_{2}(\bar{z},\bar{q}) = \overline{\theta_{2}(z,q)}
\end{split}
 \end{equation}
     as well as some periodic ones (see \cite{du1973elliptic} for proofs):
     \begin{equation}
     \label{eq: more properties of thetas}
\begin{split}
 &\theta_2(z,q) = \theta_1(z + \frac{\pi}{2},q), \\  &
         \theta_1'(z + \pi,q) = -\theta_1'(z,q),\\ 
& \frac{\theta'_1(z + \pi\mu,q)}{\theta_1(z + \pi\mu,q)} = -2i +\frac{\theta'_1(z,q)}{\theta_1(z,q)},\\ &\theta_1(z + \pi\mu,q)= -\frac{e^{-2 i z}}{q}\theta_1(z,q).
\end{split}
\ \ \ \ \ \ \ \ \ \ \ \ \
\begin{split}
&\theta_1(z+\pi,q) = -\theta_1(z,q),\\
&\frac{\theta'_1(z + \pi,q)}{\theta_1(z + \pi,q)} = \frac{\theta'_1(z,q)}{\theta_1(z,q)},\\
&\theta_1''(z + \pi,q) = -\theta_1''(z,q),\\ 
\end{split}
\end{equation}
Armed with this, we proceed with the explicit computation for the Hamiltonian.

\begin{figure}
    \centering
    \begin{tikzpicture}[scale=0.3]
    \draw[reverse directed] (-5,-4) -- (-5,8);
    \draw[directed] (-5,8) -- (7,8);
    \draw[directed] (7,8) -- (19,8);
     \draw[directed] (-5,-4) -- (7,-4);
      \draw[directed] (7,-4) -- (19,-4);
      %\node [up] at (1,0);
     \draw[ directed] (7,-4) -- (7,8);
      \draw[dashed] (-5,2) -- (19,2);
       \draw[ reverse directed] (19,-4) -- (19,8);
       \draw  (-2,0) circle (1.5);
\draw[decorate,decoration={ markings,mark=at position -3 cm with
{\arrow[line width= 0.3mm]{>}};}]{ (-2,0) circle (1.5)};
\draw[fill] (-2,0 ) circle (0.05);
\draw  (10,4) circle (1.5);
\draw[decorate,decoration={ markings,mark=at position -3 cm with
{\arrow[line width= 0.3mm]{<}};}]{ (10,4) circle (1.5)};
\draw[fill] (10,4 ) circle (0.05);
\node [left] at (-5,2) {0};
\node [left] at (-5,-4) {$-\pi/2$};
\node [left] at (-5,8) {$\pi/2$};
\node [left] at (-4,-5) {0};
\node [left] at (8.2,-5) {$2\pi$};
\node [left] at (20,-5) {$4\pi$};
\end{tikzpicture}
    \caption{A $2\pi$-by-$\pi$ torus as a double  cover of a $\pi$-by-$\pi$ Klein bottle. The dashed line is $y = 0$}
     \label{fig:model of the Klein bottle}
\end{figure}

\subsubsection{The Hamiltonian for the torus}

It is straightforward that the Hamiltonian for the Klein bottle can be obtained from the one for the torus through imposing certain symmetries on the system; for  $\mu \in \mathbb{C}$, \cite{o1989hamiltonian} gives the explicit form of the Hamiltonian for the $\pi$-by-$\mu\pi$ torus: 
 \begin{equation}
\label{eq: Ham for torus}
\mathcal{H}_T =  -\frac{1}{2\pi}\sum_{k<l}\Gamma_k\Gamma_l\left(\log\left|\theta_1\left(z'_k-z'_l, e^{i\pi\mu}\right)\right| - \frac{\left(\mathrm{Im}(z'_k-z'_l)\right)^2}{\pi\mathrm{Im}\mu}\right).
\end{equation}

The Hamiltonian for the Klein bottle can be achieved from this formula; however, we have observed interesting symmetry breaking that occurs from interactions between different summing up methods and involutions. In the scope of this work, we aim to provide an explicit calculation as well as a justification for our choice. 

Consider a lattice comprised of $2\pi$-by-$\pi$ rectangles (as in Figure \ref{fig: periodization}) and an appropriately periodised system of point vortices of strengths $\Gamma_i$, centred at points $z_i$. Then the Hamiltonian is given by an infinite (divergent) double sum:
\begin{equation}
\label{eq: Hamiltonian tirus double sum}
\mathcal{H} = -\frac{1}{2\pi}\sum_{k<l}\sum_{m,n}\Gamma_k\Gamma_l\log\left|z_k-z_l +\pi i n + 2\pi m \right|
\end{equation}
There are two natural ways of computing  $\sum\limits_{m,n}\log\left|z_k-z_l +\pi i n + 2\pi m \right|$:  summing up horizontally then vertically and the other way around. Physical reasons dictate that the two answers  coincide (up to addition of a constant). However, as we will see below, the form of the two is  drastically different.

\begin{figure}
\centering
 \begin{tikzpicture}[scale = 0.9]
    \draw[] (-5,-1.5) -- (-5,5);
         \draw[] (1,-1.5) -- (1,5);
     \draw[] (7,-1.5) -- (7,5);
      \draw[] (-5,4) -- (7,4);
           \draw[] (-5,1.5) -- (7,1.5);
             \draw[] (-5,-1) -- (7,-1);
\draw  (0.2,-0.3) circle (0.5);
\draw[decorate,decoration={ markings,mark=at position -3 cm with
{\arrow[line width= 0.3mm]{>}};}]{ (0.2,-0.3) circle (0.5)};
\draw  (0.2,2.2 ) circle (0.5);
\draw[decorate,decoration={ markings,mark=at position -3 cm with
{\arrow[line width= 0.3mm]{>}};}]{ (0.2,2.2 ) circle (0.5)};
\draw  (6.2,-0.3 ) circle (0.5);
\draw[decorate,decoration={ markings,mark=at position -3 cm with
{\arrow[line width= 0.3mm]{>}};}]{ (6.2,-0.3 ) circle (0.5)};
\draw  (6.2,2.2) circle (0.5);
\draw[decorate,decoration={ markings,mark=at position -3 cm with
{\arrow[line width= 0.3mm]{>}};}]{ (6.2,2.2) circle (0.5)};
\draw[fill] (0.2,-0.3 ) circle (0.01);
\draw[fill] (0.2,2.2 ) circle (0.01);
\draw[fill] (6.2,-0.3 ) circle (0.01);
\draw[fill] (6.2,2.2 ) circle (0.01);
\node [left] at (-5,-1.5) {0};
\node [left] at (1.2,-1.5) {$2\pi$};
\node [left] at (7.2,-1.5) {$4\pi$};
\node [left] at (-5,1.5) {$\pi$};
\node [left] at (-5,4) {$2\pi$};
\end{tikzpicture}
   \caption{Double periodisation for a single vortex on the torus}
    \label{fig: periodization}
\end{figure}

We give detailed computations for summing up horizontally then vertically. To force our infinite sum to converge, we employ the same trick as in \cite{montaldi2003vortex}: subtracting an infinitely large but constant number from our Hamiltonian. This transition will be denoted by $\rightarrow$ in the calculations below. 
\begin{small}
\label{eq: long calculation}
\begin{equation*}
    \begin{split}
      \sum_{m,n}\log\left|z_k-z_l +\pi i n + 2\pi m \right|  =& \log\left|\prod_{ m>0,n}\left((z_k-z_l +\pi i n)^2 - 4\pi^2 m^2\right) \ast\prod_n\left(z_k-z_l +\pi i n\right)\right|\xrightarrow[]{A} \\
      \xrightarrow[]{A}&\log\left|\prod_n\sin\left(\frac{z_k-z_l + \pi n i}{2}\right)\right|= \\=&
      \log\left|\prod_n\left(\sin\left(\frac{z_k-z_l }{2}\right)\cosh\left(\frac{\pi n }{2}\right) +i\cos\left(\frac{z_k-z_l }{2}\right)\sinh\left(\frac{\pi n }{2}\right)\right)\right| = \\=&
    \log\left|\prod_{n>0}\left(\sin^2\left(\frac{z_k-z_l }{2}\right)\cosh^2\left(\frac{\pi n }{2}\right) +\cos^2\left(\frac{z_k-z_l }{2}\right)\sinh^2\left(\frac{\pi n i}{2}\right)\right)\right|  \\ +& \log\left|\sin\left(\frac{z_k-z_l}{2}\right)\right| =\\ =&\log\left|\prod_{n>0}\left(\cos^2\left(\frac{z_k-z_l}{2}\right) - \cosh^2\left(\frac{\pi n}{2}\right)\right)\right| +\log\left|\sin\left(\frac{z_k-z_l}{2}\right)\right|\xrightarrow[]{B}\\\xrightarrow[]{B}&\log\left|\theta_1\left(\frac{z_k-z_l}{2},e^{-\frac{\pi}{2}}\right)\right|.
    \end{split}
\end{equation*}
\end{small}
Here we used that $\sin(ix) = i\sinh(x), \ \cos(ix)  = \cosh(x)$ for $x\in \mathbb{R}$, as well as the representations of theta functions through infinite products from  \cite{bonn} and the representation of $\sin(x) = x\prod\limits_n\left(1-\frac{x^2}{n^2\pi^2}\right)$. In the transition $\xrightarrow[]{A}$ we subtract  $\sum\limits_{m>0}\log|4\pi^2m^2|$ from our sum, and in $\xrightarrow[]{B}$ we subtract $\sum\limits_{n>0}\log\left|\cosh^2\left(\frac{\pi n }{2} \right)\right|$.

The function  $\log\left|\theta_1\left(\frac{z_k-z_l}{2},e^{-\frac{\pi}{2}}\right)\right|$ is $2\pi$-periodic in real parts of $z_j$; however, it is quasi $\pi$-periodic in imaginary and hence, not well-defined on the torus. The initial periodicity of (\ref{eq: Hamiltonian tirus double sum}) was lost when we `folded' the sum, assuming that $m$ and $n$ are greater than 0. 

In order to make the last expression in (\ref{eq: long calculation}) periodic, we need to subtract $\frac{\left(\mathrm{Im}\left(z_k-z_l\right)\right)^2}{2\pi}$ (the periodicity of the resulting function can be checked from the last relation on $\theta_1$ in (\ref{eq: more properties of thetas})). Therefore, the Hamiltonian on the torus obtained from this periodisation is given by 
\begin{equation}
    \label{eq: Ham torus horizontal periodizing}
   \mathcal{H}^T_1 =  -\frac{1}{2\pi}\sum_{k<l}\Gamma_k\Gamma_l\left(\log\left|\theta_1\left(\frac{z_k-z_l}{2},e^{-\frac{\pi}{2}}\right)\right| - \frac{\left(\mathrm{Im}\left(z_k-z_l\right)\right)^2}{2\pi}\right). 
\end{equation}
\begin{rem}
 Observe that in (\ref{eq: Ham torus horizontal periodizing}) $\log\left|\theta_1\left(\frac{z_k-z_l}{2},e^{-\frac{\pi}{2}}\right)\right|$ is the part of the Green function obtained through the method of images. The expression  $\frac{\left(\mathrm{Im}\left(z_k-z_l\right)\right)^2}{4\pi^2}$ after the application of $\Delta$ yields $1/2\pi^2$, which is the inverse of the surface area of the torus. This coincides with the relation (\ref{eq: vorticity for compact}) for point vortices on closed orientable surfaces.  
\end{rem}
Now, we sum (\ref{eq: Hamiltonian tirus double sum}) vertically then horizontally, to obtain:
\begin{small}
\begin{equation*}
    \begin{split}
      \sum_{m,n}\log\left|z_k-z_l +\pi i n + 2\pi m \right|  =& \log\left|\prod_{n>0, m}\left((z_k-z_l +2 \pi m)^2 + \pi^2 n^2\right) \ast\prod_m\left(z_k-z_l +2\pi m\right)\right| \xrightarrow[]{A'} \\
      \xrightarrow[]{A'}&\log\left|\prod_n\sinh\left(z_k-z_l + 2\pi m\right)\right|= \\=&
      \log\left|\prod_n\left(\sinh\left(z_k-z_l\right)\cosh\left(2\pi m\right) +\cosh\left(z_k-z_l\right)\sinh\left(2\pi m\right)\right)\right| \xrightarrow[]{B'}\\\xrightarrow[]{B'}&\log\left|\theta_1\left(i\left(z_k-z_l\right),e^{-2\pi}\right)\right|.
    \end{split}
\end{equation*}
\end{small}

Here, transition $\xrightarrow[]{A'}$ is subtraction of $\sum\limits_{n>0}\log\left|\pi^2n^2\right|$ and $\xrightarrow[]{B'}$  of $\sum\limits_{m>0}\log\left|\sinh^2\left(2\pi m\right)\right|$. \textcolor{black}{Compared to the previous calculation, the subtracted infinite constants are different: this happens because we consider a $2\pi$-by-$\pi$ rectangle for our periodisation as opposed to a square.}

This function is $\pi$-periodic in imaginary parts of its arguments and $2\pi$-quasi periodic in real. Analogously, we remedy that through subtracting $\frac{\left(\mathrm{Re}(z_k-z_l)\right)^2}{2\pi}$, to obtain
\begin{equation}
    \label{eq: Ham torus vertical}
    \mathcal{H}^T_2=  -\frac{1}{2\pi}\sum_{k<l}\Gamma_k\Gamma_l\left(\log\left|\theta_1\left(i\left(z_k-z_l\right),e^{-2\pi}\right)\right| - \frac{\left(\mathrm{Re}(z_k-z_l)\right)^2}{2\pi}\right).
\end{equation}
\begin{rem}
 Observe how the change $z\mapsto iz$ transforms (\ref{eq: Ham torus vertical}) into (\ref{eq: Ham for torus}), as written for a $\pi$-by-$2\pi$ torus: indeed, multiplication by $i$ is the ninety degree rotation of the plane.
\end{rem}
We have remarked above that the functions $\mathcal{H}^T_1$ and $\mathcal{H}^T_2$ must differ by a constant; however, they look nothing like each other. Nonetheless, the following holds:

\begin{lem}
$\mathcal{H}_1^T + \frac{\log(2)}{4\pi}\sum\limits_{k<l}\Gamma_k\Gamma_l  = \mathcal{H}_2^T$.
\end{lem}
\begin{proof}
We employ the following  equality from \cite{bonn}:
\begin{equation}
\label{eq: connection with i}
%\frac{1}{\sqrt{\lambda}}\,e^{z^2/\pi\lambda}\,\theta_1\left(\frac{z}{\lambda}, e^{-{\pi}/{\lambda}}\right)= -i\theta_1\left(iz, e^{-\pi\lambda}\right)\\
\frac{1}{\sqrt{\lambda}}\,e^{z^2/(\pi\lambda)}\,\theta_1\left(\lambda^{-1}z, e^{-{\pi}/{\lambda}}\right)= -i\theta_1\left(iz, e^{-\pi\lambda}\right)
\end{equation}
in order to write for each $z_k$ and $z_l$ pair:
\begin{small}
\begin{equation*}
    \begin{split}
     & \log\left|\theta_1\left(\frac{z_k-z_l}{2},e^{-\frac{\pi}{2}}\right)\right|   - \frac{(\mathrm{Im}(z_k-z_l))^2}{2\pi} \\&=\log\left|\theta_1\left(i(z_k-z_l),e^{-2\pi}\right)\right|   - \frac{(\mathrm{Im}(z_k-z_l))^2}{2\pi}- \log\left|\mathrm{exp}\left(\frac{(z_k-z_l)^2}{2\pi}\right)\right| + \log(\sqrt{2}) \\&=\log\left|\theta_1\left(i(z_k-z_l),e^{-2\pi}\right)\right|   - \frac{(\mathrm{Im}(z_k-z_l))^2}{2\pi} - \frac{(\mathrm{Re}(z_k-z_l))^2}{2\pi} + \frac{(\mathrm{Im}(z_k-z_l))^2}{2\pi} + \log(\sqrt{2})\\&=\log\left|\theta_1\left(i(z_k-z_l),e^{-2\pi}\right)\right|  - \frac{(\mathrm{Re}(z_k-z_l))^2}{2\pi} + \log(\sqrt{2}),
    \end{split}
\end{equation*}
\end{small}
whence the statement of the lemma follows immediately.
\end{proof}
\subsubsection{ The Hamiltonian for the Klein bottle}
\label{sec: Ham Klein}
The next step is to periodise the Hamiltonian for the torus in order to get the one for the Klein bottle.

%Since $\tau$ in (\ref{eq: more properties of thetas}) denotes part of the argument in the  Jacobi functions $\theta_j$, moving forward  we will  use $\mu$ for the antisymplectic involution  on the double cover. 

Our rectangle covers two copies of the Klein bottle, as drawn in the Figure \ref{fig:model of the Klein bottle}. Assuming that the bottom left corner of the rectangle is placed at the point with coordinates $(-\pi/2, -\pi/2)$, the involution will be explicitly given by $\tau:
z\mapsto\bar{z} + \pi$.

We periodise both forms of the Hamiltonian to see if two functions (\ref{eq: Ham torus horizontal periodizing}) and (\ref{eq: Ham torus vertical}) will give the same results.  Consider first $\mathcal{H}^T_2$.

Observing that $\left|\theta_1(z,q)\right| = \left|\theta_1(\bar{z},q)\right|$ when $q\in\mathbb{R}$ allows us to simplify it to the expression (we also subtract a constant and divide the final expression by 2 for the reasons discussed in the first section of this work): 
\begin{equation}
    \begin{split}
    \label{eq: Ham the wromg one}
        \mathcal{H}_0 =& -\frac{1}{2\pi}\sum_{k<l}\Gamma_k\Gamma_l\log\left|\theta_1\left(i(z_k-z_l), e^{-2\pi}\right)\right| + \frac{1}{4\pi}\sum_{k\ne l}\Gamma_k\Gamma_l\log\left|\theta_1\left(i(z_k-\bar{z}_l- \pi), e^{-2\pi}\right)\right|  \\&+\frac{1}{4\pi}\sum_k\Gamma_k^2\log\left|\theta_1\left(2y_k,e^{-2\pi}\right)\right|
    \end{split}
\end{equation}

The function (\ref{eq: Ham the wromg one}) has the same periodicity properties as its counterpart on the torus; however, is it invariant under $\tau$? It turns out, not quite:
\begin{lem}
$\mathcal{H}_0(\tau(z_1), z_2,\ldots,z_n) = \mathcal{H}_0(z_1,\ldots,z_n) - \frac12\sum\limits_{k\ne 1}\Gamma_1\Gamma_k$.
\end{lem}
\begin{proof}
From \cite{whittaker2020course}, we know that the following relation holds for $z\in\mathbb{C}$:
\[
\frac{\theta_1(z + 2i\pi, e^{-2\pi})}{\theta_1(z, e^{-2\pi})} = -e^{2\pi -2iz}.
\]
Considering that the involution changes the sign of $\Gamma_1$, a computation yields 
\begin{equation*}
    \begin{split}
        \mathcal{H}_0(\tau(z_1), z_2,\ldots,z_n) &= \mathcal{H}_0(z_1,\ldots,z_n) + \frac{1}{4\pi}\sum_{k\ne 1}\Gamma_1\Gamma_k\log\left|\mathrm{exp}\left(2\pi + 2(z_1 - \bar{z}_k - \pi)\right)\right| \\&- \frac{1}{4\pi}\sum_{k\ne 1}\Gamma_1\Gamma_k\log\left|\mathrm{exp}\left(2\pi + 2(z_1-z_k)\right)\right|\\&=\mathcal{H}_0(z_1,\ldots,z_n) - \frac{1}{4\pi}\sum_{k\ne 1}\Gamma_1\Gamma_k\log\left|\mathrm{exp}\left(2(\bar{z}_k-z_k) + 2\pi\right)\right| \\&= \mathcal{H}_0(z_1,\ldots,z_n)- \frac{1}{2}\sum_{k\ne 1}\Gamma_1\Gamma_k
    \end{split}
\end{equation*}
\end{proof}

Therefore, this function is not well-defined on the square model of the Klein bottle: this form of the Hamiltonian on the torus is incompatible with this concrete periodisation.

However, periodising differently yields a well-defined Hamiltonian: suppose our double cover is instead  as in Figure \ref{fig:model of the Klein bottle alt}, i.e. the Klein bottle is $2\pi$-by-$\pi/2$. Then $\tau':z\mapsto -\bar{z} + i\frac{\pi}{2}$. It can be checked that this periodisation, as applied to $\mathcal{H}^T_2$, gives a Hamiltonian with proper periodicities and invariant under $\tau'$. 
\begin{figure}
    \centering
    \begin{tikzpicture}[scale=0.3]
    \draw[ directed] (-5,-4) -- (-5,8);
    \draw[directed] (-5,8) -- (19,8);
      \draw[directed] (-5,-4) -- (19,-4);
      %\node [up] at (1,0);
     \draw[ directed] (7,-4) -- (7,8);
      \draw[reverse directed] (-5,2) -- (19,2);
       \draw[directed] (19,-4) -- (19,8);
       \draw  (4,4) circle (1.5);
\draw[decorate,decoration={ markings,mark=at position -3 cm with
{\arrow[line width= 0.3mm]{>}};}]{ (4,4) circle (1.5)};
\draw[fill] (4,4 ) circle (0.05);
\draw  (10,-2) circle (1.5);
\draw[decorate,decoration={ markings,mark=at position -3 cm with
{\arrow[line width= 0.3mm]{<}};}]{ (10,-2) circle (1.5)};
\draw[fill] (10,-2 ) circle (0.05);
\node [left] at (-5,-5) {0};
\node [left] at (8,-5) {$\pi$};
\node [left] at (20,-5) {$2\pi$};
\node [left] at (-5,2) {$\pi/2$};
\node [left] at (-5,8  ) {$\pi$};
\end{tikzpicture}
    \caption{A $2\pi$-by-$\pi$ torus as a double  cover of a $2\pi$-by-$\pi/2$ Klein bottle}
     \label{fig:model of the Klein bottle alt}
\end{figure}

The square model is more convenient for  computational purposes, so we  obtain the final Hamiltonian  for the Klein bottle from periodising (and again, dividing by 2) $\mathcal{H}^T_1$: 

\begin{small}
\begin{equation}
\begin{split}
    \label{eq: ONE TRUE KLEIN HAM}
    \mathcal{H} =& -\frac{1}{2\pi}\sum\limits_{\alpha<\beta}\Gamma_{\alpha}\Gamma_{\beta}\log\left|\theta_1\left(\frac{z_{\alpha}-z_{\beta}}{2},e^{-\frac{\pi}{2}}\right)\right| + \frac{1}{2\pi}\sum\limits_{\alpha<\beta}\Gamma_{\alpha}\Gamma_{\beta}\log\left|\theta_2\left(\frac{z_{\alpha}-\bar{z}_{\beta}}{2},e^{-\frac{\pi}{2}}\right)\right|  \\&+\frac{1}{2\pi}\sum\limits_{\alpha<\beta}\Gamma_{\alpha}\Gamma_{\beta}\left(\frac{(y_{\alpha} - y_{\beta})^2}{2\pi} - \frac{(y_{\alpha} + y_{\beta})^2}{2\pi}\right) + \frac{1}{4\pi}\sum\limits_{\alpha}\Gamma_{\alpha}^2\left(\log\left|\theta_1\left(iy_{\alpha} - \frac{\pi}{2},e^{-\frac{\pi}{2}}\right)\right| - \frac{2y_{\alpha}^2}{\pi}\right)  \\=&
    -\frac{1}{2\pi}\sum\limits_{\alpha<\beta}\Gamma_{\alpha}\Gamma_{\beta}\log\left|\theta_1\left(\frac{z_{\alpha}-z_{\beta}}{2},e^{-\frac{\pi}{2}}\right)\right| + \frac{1}{2\pi}\sum\limits_{\alpha<\beta}\Gamma_{\alpha}\Gamma_{\beta}\log\left|\theta_2\left(\frac{z_{\alpha}-\bar{z}_{\beta}}{2},e^{-\frac{\pi}{2}}\right)\right| \\&-\frac{1}{\pi^2}\sum\limits_{\alpha<\beta}\Gamma_{\alpha}\Gamma_{\beta}y_{\alpha}y_{\beta} + \frac{1}{4\pi}\sum\limits_{\alpha}\Gamma_{\alpha}^2\left(\log\left|\theta_1\left(iy_{\alpha} - \frac{\pi}{2},e^{-\frac{\pi}{2}}\right)\right| - \frac{2y_{\alpha}^2}{\pi}\right).     \end{split}
\end{equation}
\end{small}

\begin{lem}
The Hamiltonian (\ref{eq: ONE TRUE KLEIN HAM}) is $\pi$-periodic vertically, $2\pi$-periodic horizontally and invariant under $\tau$. 
\end{lem}

\begin{proof}
The first two statements for our Hamiltonian follow from the corresponding properties of its predecessor on the torus (\ref{eq: Ham torus horizontal periodizing}). The last statement can be easily checked as well: under the change $x_i\mapsto x_i + \pi,\ y_i\mapsto-y_i,\ \Gamma_i\mapsto-\Gamma_i$ the first two summands exchange places, and the last two remain unchanged. 
\end{proof}
\begin{rem}
\label{st: Robin function Klein peiodicity}
The Robin function in this case will be  $\log\left|\theta_1\left(iy - \frac{\pi}{2}, e^{-\frac{\pi}{2}}\right)\right| - \frac{2y^2}{\pi}$. At first glance, it does not look like a periodic function; however, the following holds:
\[
\log\left|\theta_1\left(iy - \frac{\pi}{2}, e^{-\frac{\pi}{2}}\right)\right| - \frac{2y^2}{\pi} = \log\left|\theta_4\left(2y,e^{-2\pi}\right)\right| + C.
\]
for a constant $C$. 

Indeed, from (\ref{eq: connection with i}), we have
\begin{small}
\begin{equation*}
    \begin{split}
\log\left|\theta_1\left(iy-\frac{\pi}{2},e^{-\frac{\pi}{2}}\right)\right| - \frac{2y^2}{\pi} =&\log\left|\theta_1\left(2y +i\pi,e^{-2\pi} \right)\right| + \log\left|\mathrm{exp}\left(\frac{(2y + i\pi)^2}{2\pi}\right)\right| -\log(\sqrt{2}) - \frac{2y^2}{\pi}\\=&\log\left|\theta_1\left(2y + i\pi,e^{-2\pi}\right)\right| + C'.
\end{split} 
\end{equation*}
\end{small}
where $C'$ is some constant.

On the other hand, we know from \cite{wolfram} that 
\begin{equation*}
    \label{eq: quasiperiodicity}
\theta_1\left(z,q\right) = -ie^{iz + \pi i\tau/4} \ \theta_4\left(z + \frac{1}{2}\pi\tau,q\right),
\end{equation*}
which gives us 
\begin{equation*}
    \begin{split}
        \log\left|\theta_1\left(2y + i\pi,e^{-2\pi}\right)\right| &= \log\left|\theta_4\left(2y + i\pi -i\pi,e^{-2\pi} \right)\right|+ \log\left|\mathrm{exp}\left(2iy- \pi -\frac{\pi}{2}\right)\right|   \\&=\log\left|\theta_4\left(2y,e^{-2\pi}\right)\right| + C''.
    \end{split}
\end{equation*}
for some constant $C''$.

Therefore, the Robin function has required periodicities.

\end{rem}

From the discussion above, one can see that   \textcolor{black}{of the two}  naturally constructed Hamiltonians  (\ref{eq: ONE TRUE KLEIN HAM}) is the only one well-defined on the square model of the  Klein bottle and invariant under the more computationally convenient involution $\tau$. Hence, this is the energy function we will use going forward. 
\begin{rem}
 In the light of Remark \ref{st: Robin function Klein peiodicity}, the Green's function on the Klein bottle is 
\begin{small}
\[
G_K(z,w) = \frac{1}{2\pi}\log\left|\theta_1\left(\frac{z-w}{2},e^{-\frac{\pi}{2}}\right)\right|  - \frac{1}{2\pi}\log\left|\theta_2\left(\frac{z-\bar{w}}{2},e^{-\frac{\pi}{2}}\right)\right|  +\frac{1}{\pi^2}\mathrm{Im}(z)\mathrm{Im}(w)
.
\]

\end{small}
\end{rem}
\subsection{Motion of one vortex}
\label{sec: motion one vort klein}
\begin{figure}
    \centering
   \subfigure[The Robin function]{\includegraphics[scale =0.45]{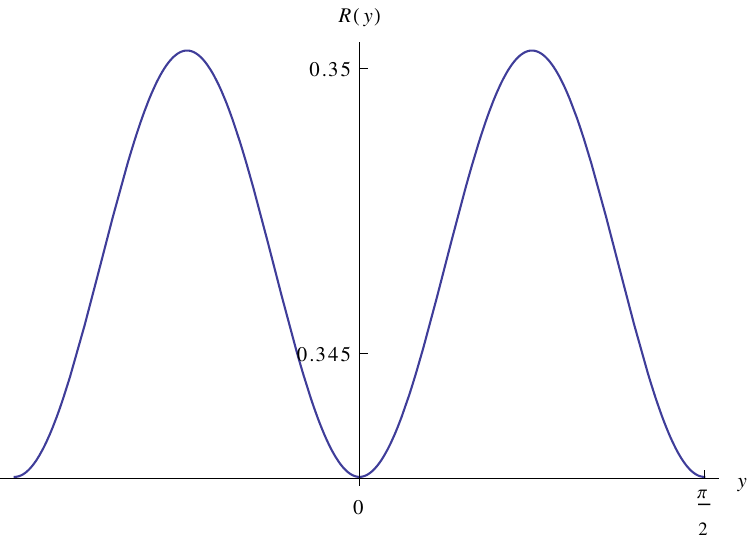}}
   \subfigure[The derivative of the Robin function]{\includegraphics[scale =0.45]{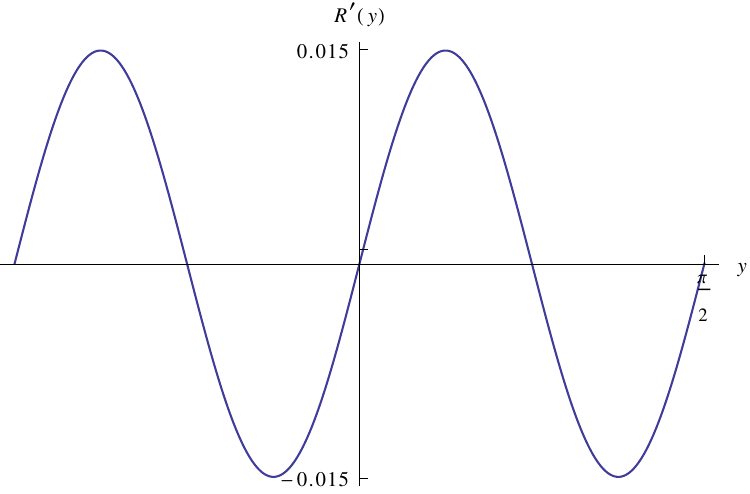}}
   \caption{The Robin function and its derivative}
   \label{fig: Robin function klein}
\end{figure}

A simple calculation gives that the equation of motion of one point vortex of strength $\Gamma$ has the form
\[
\begin{cases}
\dot{x} =  \frac{\Gamma}{4\pi}\left(i\frac{\theta'_1\left(iy - \frac{\pi}{2}\right)}{\theta_1\left(iy - \frac{\pi}{2}\right)} - \frac{4y}{\pi}\right),\\
\dot{y} = 0.
\end{cases}
\]
\begin{rem}
\label{st: remark purely imaginal}
 The expression $\frac{\theta'_1\left(iy - \frac{\pi}{2}\right)}{\theta_1\left(iy - \frac{\pi}{2}\right)}$ is a purely imaginary number when $y\in\mathbb{R}$.
To see this, we use the following relation from \cite{wolfram}:
 \begin{equation}
 \label{eq: thetaprime/theta}
 \frac{\theta'_1(z)}{\theta_1(z)} = \cot(z) + 4\sum_{n=1}^{\infty}\frac{q^{2n}\sin(2z)}{q^{4n} - 2q^{2n}\cos(2z) + 1}
 \end{equation}
 This entails that when $z = iy - \frac{\pi}{2}$, (\ref{eq: thetaprime/theta}) turns into
 \begin{equation*}
     \begin{split}
         &\cot\left(iy - \frac{\pi}{2}\right)+ 4\sum_{n=1}^{\infty}\frac{q^{2n}\sin(2iy - \pi)}{q^{4n} - 2q^{2n}\cos(2iy - \pi) + 1} = \\&=-i\tanh(y) - 4i\sum_{n=1}^{\infty}\frac{q^{2n}\sinh(2y)}{q^{4n} + 2q^{2n}\cosh(2y ) + 1}
     \end{split}
 \end{equation*}
\end{rem}
The plot of the Robin function  and its derivative are depicted in Figure \ref{fig: Robin function klein} (a) and (b). 

We have already demonstrated   $\frac{\pi}{2}$-periodicity of $R(y)$; one can also observe that it is an even function. 

Note that the function is not symmetric with respect to the reflection across the $y$-axis and translation by $\pi/4 $ (and neither is, consequently, its derivative) - this fact can be numerically checked. 

The velocity of the point vortex, in turn, is a $\frac{\pi}{2}$-periodic odd function; a solitary  vortex will be stationary if and only if placed on the lines $y = 0,\pm\frac{\pi}{4}, \pm\frac{\pi}{2}$ . Observe that the first and the last lines are the loci of fixed points for the orientation changing isometry, and therefore  are necessarily  critical points of the Robin function.

\begin{figure}
\centering
 \subfigure[One copy of the bottle]{\includegraphics[scale = 0.5]{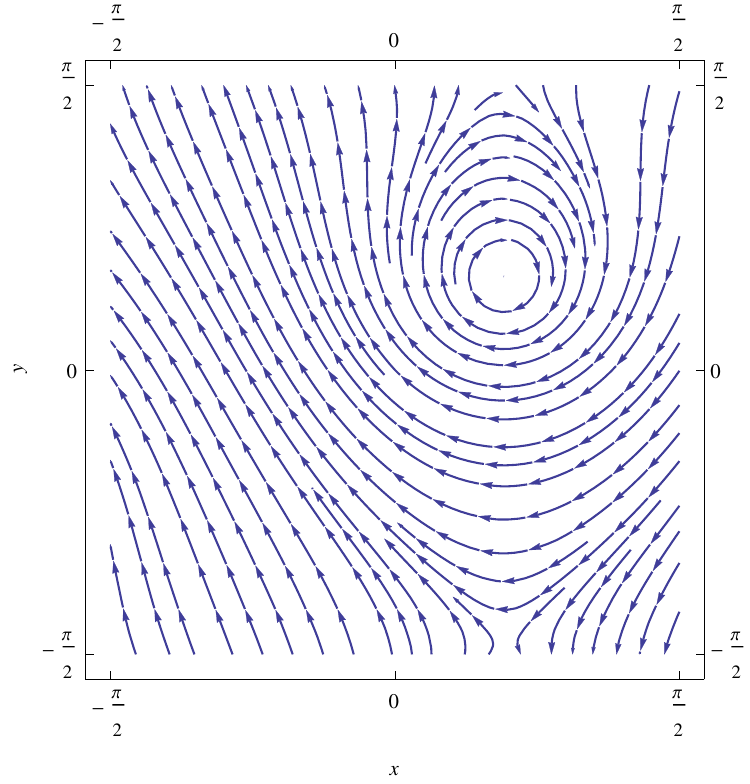}}
\subfigure[Four copies of the bottle]{\includegraphics[scale = 0.5]{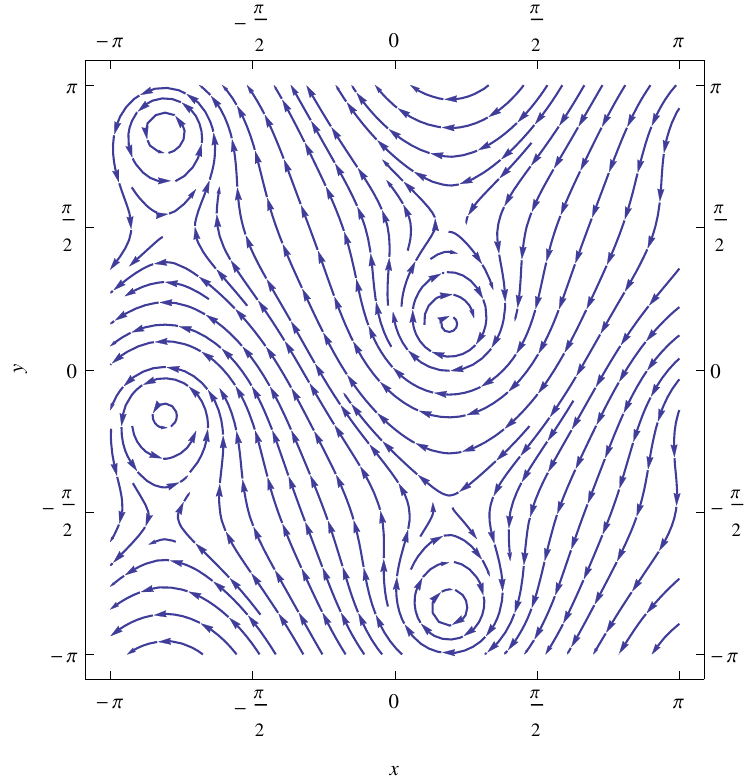}}
 \caption{Vector field created by a single vortex on the Klein bottle}
 \label{fig: vector field single vortex Klein bottle}
 \end{figure}

% \begin{rem}
   %Unlike the M\"obius band and the Klein bottle, a solitary point vortex will be stationary on the projective plane.  In the two former cases, the geodesic distance between the two preimages of the vortex on the double cover depends on the placement of its centre; that is not the case for $\mathbb{R}P^2$.
 %\end{rem}
With the help of the equation of motion (see details below) we can reconstruct the vector field created by one vortex: see Figure \ref{fig: vector field single vortex Klein bottle} (a) for one copy of the band and (b) for multiple.

  \subsection{Symmetries and invariants}
  \label{sec: symmetries Klein}
From the form of the Hamiltonian it can be seen that, as in the case of the M\"obius band, the group of symmetries of motion will be $S^1$, acting by horizontal translations. However, now the function $C := \sum\limits_{k}\Gamma_{k}y_{k}$ is  a local (as opposed to global) invariant of motion.

Due to the nature of the $S^1$-action, relative equilibria on the Klein bottle will behave  identically  to the ones on the M\"obius band: vortices will move horizontally, maintaining a rigid configuration. 

Analogously to Section \ref{sec:vortex motion Mobius} and \cite{montaldi2000relative}, we can observe that the following configurations will be fixed and relative equilibria respectively:  
\begin{itemize}
    \item an odd number of  point vortices  on the lines $y = 0$ or $y = \frac{\pi}{2}$ with alternating signs of strengths;  existence of fixed equilibria for arrangements like this can be demonstrated  in precisely the same manner as the one in \cite{montaldi2003vortex}: existence of critical points of he Hamiltonian is deduced from its behaviour at configurations where the vortices collide . 
    \item configurations inherited from $N$-rings on the torus; two aligned rings  when $N$ is even and two staggered ones when $N$ is odd (see \cite{montaldi2000relative} and \cite{laurent2001point} for the application of the Principle of Symmetric Criticality in this case). 
\end{itemize}

  \subsection{Two vortices}
  \label{sec: two point vortices Klein}
 For two point vortices with strengths $\Gamma_1$ and $\Gamma_2$ and centres at $z_1, z_2$ the Hamiltonian has the form 
 \begin{equation}
 \label{eq: Ham Klein two vort}
 \begin{split}
     \mathcal{H} =& -\frac{1}{2\pi}\Gamma_1\Gamma_2\log\left|\theta_1\left(\frac{z_1-z_2}{2},e^{-\frac{\pi}{2}}\right)\right| + \frac{1}{2\pi}\Gamma_1\Gamma_2\log\left|\theta_2\left(\frac{z_1-\bar{z}_2}{2},e^{-\frac{\pi}{2}}\right)\right| \\&- \frac{1}{\pi^2}\Gamma_1\Gamma_2y_1y_2 + \frac{1}{4\pi}\Gamma_1^2\left(\log\left|\theta_1\left(iy_1 - \frac{\pi}{2},e^{-\frac{\pi}{2}}\right)\right| - \frac{2y_1^2}{\pi}\right) + \\&+ \frac{1}{4\pi}\Gamma_2^2\left(\log\left|\theta_1\left(iy_2 - \frac{\pi}{2},e^{-\frac{\pi}{2}}\right)\right| - \frac{2y_2^2}{\pi}\right)
     \end{split}
 \end{equation}
  
\begin{figure}
    \centering
    \includegraphics[scale = 0.7]{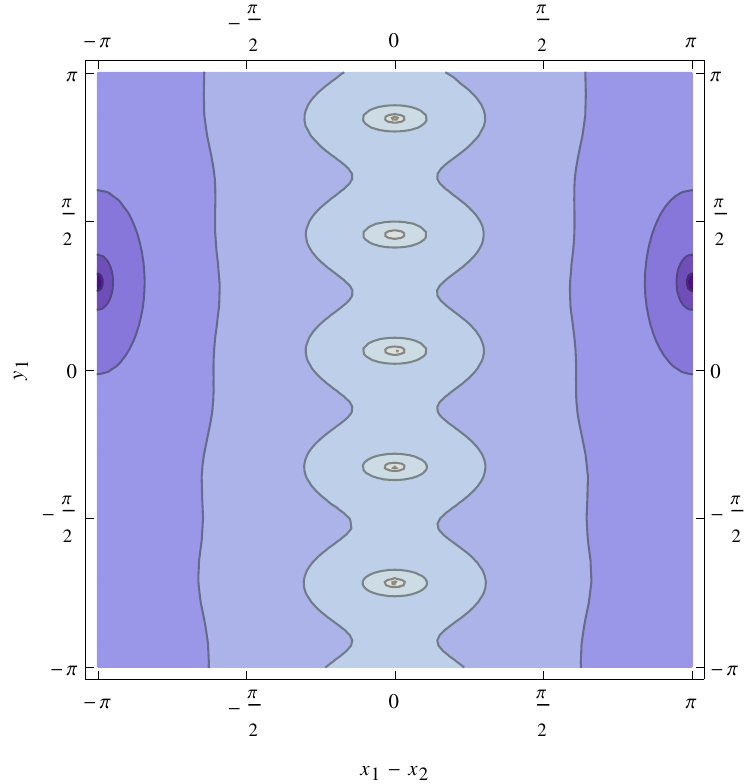}
    \caption{Level sets of the reduced  Hamiltonian for the Klein bottle}
    \label{fig:my_label}
\end{figure}

We recall the following formula from \cite{montaldi2003vortex} in order to rewrite our equations of motion using complex numbers:
\[
\dot{z}_k = -2i\frac{\partial\mathcal{H}}{\partial \bar{z}_k}.
\]
Differentiating and rewriting the result through $\bar{z}_k$ with the help of the relations (\ref{eq: more properties of thetas}), we obtain the equations of vortex motion:
\begin{footnotesize}
\begin{equation}
    \label{eq: Klein equations of motion}
    \begin{cases}
    \dot{z}_1 &=-2i\Bigl[-\frac{1}{4\pi}\Gamma_2\frac{\theta'_1\left(\frac{\bar{z}_1 - \bar{z}_2}{2},e^{-\frac{\pi}{2}}  \right)}{\theta_1\left(\frac{\bar{z}_1 - \bar{z}_2}{2},e^{-\frac{\pi}{2}} + \right)}  +\frac{1}{4\pi}\Gamma_2\frac{\theta'_2\left(\frac{\bar{z}_1 - z_2}{2},e^{-\frac{\pi}{2}} \right)}{\theta_2\left(\frac{\bar{z}_1 - z_2}{2},e^{-\frac{\pi}{2}} + \right)} + \frac{1}{4\pi^2}\Gamma_2\left(\bar{z}_2-z_2\right) + \frac{1}{4\pi}\Gamma_1\Bigl(\frac{\theta_1'\left(\frac{\bar{z}_1-z_1}{2} + \frac{\pi}{2}, e^{-\frac{\pi}{2}}\right)}{\theta_1\left(\frac{\bar{z}_1-z_1}{2} + \frac{\pi}{2}, e^{-\frac{\pi}{2}}\right)} +\\&+ \frac{\bar{z}_1-z_1}{\pi}\Bigr)\Bigr]\\
    \dot{z}_2 &=-2i\Bigl[\frac{1}{4\pi}\Gamma_1\frac{\theta'_1\left(\frac{\bar{z}_1 - \bar{z}_2}{2},e^{-\frac{\pi}{2}} + \right)}{\theta_1\left(\frac{\bar{z}_1 - \bar{z}_2}{2},e^{-\frac{\pi}{2}} + \right)}  +\frac{1}{4\pi}\Gamma_1\frac{\theta'_2\left(\frac{\bar{z}_2 - z_1}{2},e^{-\frac{\pi}{2}} + \right)}{\theta_2\left(\frac{\bar{z}_2 - z_1}{2},e^{-\frac{\pi}{2}} + \right)} + \frac{1}{4\pi^2}\Gamma_1\left(\bar{z}_1-z_1
    \right) + \frac{1}{4\pi}\Gamma_2\Bigl(\frac{\theta_1'\left(\frac{\bar{z}_2-z_2}{2} + \frac{\pi}{2}, e^{-\frac{\pi}{2}}\right)}{\theta_1\left(\frac{\bar{z}_2-z_2}{2} + \frac{\pi}{2}, e^{-\frac{\pi}{2}}\right)} +\\&+ \frac{\bar{z}_2-z_2}{\pi}\Bigr)\Bigr]
    \end{cases}
\end{equation}
\end{footnotesize}
Rewriting (\ref{eq: Klein equations of motion}) with $x_i$ and $y_i$ rather than $z_i$, one can observe that the Hamiltonian depends on $x_1-x_2$ rather than on $x_1$ and $x_2$ separately. Additionally, we have a constant of motion $C = \Gamma_1y_1 + \Gamma_2y_2$; above we have stressed that this invariant is a local one.

However, after making some adjustments to the method from Section \ref{sec:vortex motion Mobius}, we may treat it as a global one: we suppose that the initial placement of the two vortices is such that their $y-$coordinates lie in the interval $\left[-\frac{\pi}{2},\frac{\pi}{2}\right]$ and proceed to observe this concrete pair of vortices without restricting the values of their $y$-coordinates.

In doing so we transfer to a covering system on a cylinder. Clearly, the motion of these two vortices defines that of the entire system, therefore no information is lost.   For this covering  system, the value of  $C$ does not change. Therefore,  following one concrete pair of vortices we may restore our motion while  treating $C$ as a global constant.

 From the periodicity of the Hamiltonian in each $x_i$ we deduce that it is $2\pi$-periodic in $x_1-x_2$; we assume the signs of $\Gamma_1,\Gamma_2$ (supposing, as before, that $\Gamma_1>\Gamma_2>0$) and therefore have to consider the entire period in $x_1-x_2$. Due to this periodicity $x_1-x_2\in (-\pi,\pi)$. Next, we substitute $y_2 = \frac{C}{\Gamma_2} - \frac{\Gamma_1}{\Gamma_2}y_1$ into the Hamiltonian (\ref{eq: Ham Klein two vort}) and drop the assumption that $y_1$ is bounded.  By drawing the level sets of the reduced Hamiltonian, we obtain Figure \ref{fig:my_label}.

 Note how for the reduced Hamiltonian the  periodicity in $x$-coordinate is retained while  periodicity in $y_1$ is lost -- once again, this happens due to $C$ being the local invariant.

\begin{lem}
\label{st: lemma crit points of the Hamiltonian Klein}
Critical points of the  reduced Hamiltonian  belong to the lines $x_1-x_2 = 0,\pm \pi$. 
\end{lem}

This property is identical to those on the cylinder and the M\"obius band;  however, the rigorous proof of this particular statement is a very technical computation, and we provide it below, in Appendix \ref{sec: Appendix B }.

\begin{lem}
The only singular points of the reduced Hamiltonian on the lines $x = 0, \ \pm\pi$ are the points of the form $y_1= \frac{\pi k\Gamma_2 + c}{\Gamma_1 + \Gamma_2}$ and $y_1 = \frac{k\pi\Gamma_2-c}{\Gamma_2 -\Gamma_1}$ respectively, where $k\in\mathbb{Z}$. 
\end{lem}
\begin{proof}
All singularities that the Hamiltonian has are at the configurations where the point vortices collide. When $x_1-x_2 = 0$ and $y_1 = \frac{k\pi\Gamma_2+C}{\Gamma_1 + \Gamma_2}$, $y_2$ is equal to $\frac{C-k\pi\Gamma_1}{\Gamma_1 + \Gamma_2}$. Then the following relation holds:
\[
y_1-k\pi = \frac{C-k\pi\Gamma_1}{\Gamma_1 + \Gamma_2}= y_2,
\]
and the covering copies of the two vortices on the plane collide.

When $x_1-x_2 =\pm\pi$, $y_1 = \frac{k\pi\Gamma_2-C}{\Gamma_2-\Gamma_1}$, which entails $y_2 = \frac{C-k\pi\Gamma_1}{\Gamma_2 - \Gamma_1}$.
Therefore,
\[
y_1 - k\pi = \frac{k\pi\Gamma_1-C}{\Gamma_2-\Gamma_1} = -y_2. 
\]
Similarly to the case of the M\"obius band, the first point vortex in this case collides with a `-' copy of the second one. 

\end{proof}
\begin{theorem}
Consider the level sets of the Hamiltonian, as drawn in  Figure \ref{fig:my_label}. On the closed curves  the motion will be the same as in Regions I and II on the M\"obius band. On non-closed trajectories the global motion will have both vertical and horizontal translational components, and the relative motion of the vortices need not be periodic.
\end{theorem}
\begin{proof}
In here, we rely heavily on the results that we obtained for the case of the M\"obius band in Section \ref{sec:vortex motion Mobius}, Theorem \ref{st: two vortex motion on Mobius band}: all the reasoning about smooth and non-trivial dependence of integrals $\int\limits_0^T\dot{x}_k\mathrm{d}t$ on the trajectory can be repeated verbatim. 

We single out three separate cases: trajectories  in Region I are the closed trajectories around singular points on the line $x_1-x_2 = 0$; trajectories in Region II are the closed ones around singularities on $x_1-x_2 = \pm\pi$. 
Trajectories in Region III are the ones that are not closed. 

Analogously to the case of the M\"obius band, we deduce that in Regions I and II the two point vortices rotate around each other and around the bottle; since they emulate the behaviour of a solitary point vortex when in close proximity to each other, a vortex pair will be stationary if and only if $C = 0, \pm\frac{\pi}{4}(\Gamma_1 + \Gamma_2),\pm\frac{\pi}{2}(\Gamma_1  + \Gamma_2)$ (these being the momentum values for which a solitary point vortex is stationary). Otherwise the motion will have a horizontal translation component. 

On the ``vertical" (non-closed) trajectories the two point vortices will move approximately as  point vortices in Region III of the M\"obius band do; but here, as we have mentioned, periodicity in $y_1$ is lost: if we 'glue' the torus according to the periodisations, the trajectories in this region will not necessarily be closed curves. Additionally, the motion is unbounded in $y_1$, and, therefore, in $y_2$ as well.

Employing the reasoning similar to the one in the case of the  M\"obius band,  we can conclude that the horizontal translation component (the integral $\int_0^T\dot{x}_1\mathrm{d}t$) of motion is nonzero for almost all values of $C$;  observe  that the motion very close to  relative equilibria on the line $x_1-x_2=0$ must have nonzero horizontal components.

\end{proof}

\section{Concluding remarks}

In this paper, we considered the Hamiltonian approach to point vortex motion on non-orientable manifolds.  We introduced a general method of approaching such systems as Hamiltonian flows and discussed in detail the flows on two most well-known examples of two-dimensional non-orientable manifolds. 

A number of questions remain unanswered for the motion on the Mobius band, the prime example being proving that the maximal number of critical points of the Hamiltonian on the M\"obius band is as depicted in Figure \ref{fig:my_label2}. Since this function has three independent parameters and complicated structure, all the standard approaches to a proof present considerable numerical difficulties.

The other problem is a description of fixed and relative equilibria for three point vortices. One can demonstrate  that fixed equilibria on a straight vertical line must conform to the condition $\Gamma_1^2y_1 + \Gamma_2^2y_2 + \Gamma_3^2y_3=0$, but the remaining conditions are hard to pin down. 

For vortex motion on the Klein bottle, the computations are significantly impeded by the form of the Hamiltonian and the equations of motion. It would be interesting to devise a method, using numerical computations and the properties of Jacobi theta functions, to classify all relative and fixed equilibria of two vortices. Additionally, a more detailed examination of the motion in Region III would be beneficial for a complete  understanding of the behaviour of point vortices.

%%%%%%%%%%%%%%%%%%%%%%%%%%%%%%%%%%%%%%%%%%%%%%%%%%%%
\appendix
\section{Velocity of the N-ring relative equilibria}
\label{sec: Appendix}

Here we give a proof of the formulae for the angular velocities of the $N$-ring relative equilibria stated in Theorem \ref{st: N ring equilibria}. 

We begin with the case of two aligned rings.  As we have mentioned, $N =2K$: therefore, $K$ is the number of point vortices in the upper row on the chart on the M\"obius band.

Without loss of generality we may assume that the leftmost vortex in the top row is positioned at $(0,y)$: the  $i$th vortex in the top row is at $(\frac{\pi(i-1)}{K},y)$, with similar coordinates for the lower row.  Since the configuration is a relative equilibrium, we only need to determine the velocity of one of the vortices: we do it for the leftmost one in the upper row. After substituting the values above into (\ref{Mot}), we get
\begin{equation}
\label{eq: Re ring velocity}
\dot{x}_1= \frac{\Gamma}{4\pi}\left(\tanh(y) + \coth(y)\right)+ \frac{\Gamma}{8\pi}\sum_{j=1}^{K-1}\frac{\sinh(2y)}{\sin^2\left(\frac{\pi j}{K}\right)  + \sinh^2(y)} + \frac{\sinh(2y)}{\cos^2\left(\frac{\pi j}{K}\right)  + \sinh^2(y)}
\end{equation}
\textcolor{black}{\begin{rem} A stragihtforward calculation can show that $ \frac{\sinh(2y)}{\sin^2\left(\frac{\pi K}{K}\right) + \sinh^2(y)} + \frac{\sinh(2y)}{\cos^2\left(\frac{\pi K}{K}\right) + \sinh^2(y)} = 2\left(\tanh(y) + \coth(y)\right)$, and therefore
$$
\dot{x}_1 = \frac{\Gamma}{8\pi}\sum\limits_{j=1}^{K}\frac{\sinh(2y)}{\sin^2\left(\frac{\pi j}{K}\right) + \sinh^2(y)} + \frac{\sinh(2y)}{\cos^2\left(\frac{\pi j}{K}\right) + \sinh^2(y)}.
$$
\end{rem}
}
 
\begin{lem}
The two sums above are given by
\label{st: the horrible sums}
$$\sum_{j=1}^{K}\frac{\sinh(2y)}{\sin^2\left(\frac{\pi j}{K}\right)  + \sinh^2(y)} = 2K\coth(Ky)$$
and for the second,
$$\sum_{j=1}^{K}\frac{\sinh(2y)}{\cos^2\left(\frac{\pi j}{K}\right)  + \sinh^2(y)} = \begin{cases}
  2K\coth(Ky) & \text{if $K$ is even} \\[12pt]
 \displaystyle 2K\tanh(Ky) & \text{if $K$ is odd}. 
\end{cases}
$$

\end{lem}

\def\Res{\mathop\mathrm{Res}}

\begin{proof}
We adapt the method posted on the stackexchange forum \cite{2661896}: observe that the function \[
\frac{2K}{z(z^{2K} -1)} 
\]
has residues 1 at $e^{\pi ij/K}$ (and $-2K$ at 0).  For our first sum, consider the function

\begin{equation}
    \label{eq: the sum but function}
    f(z): = \frac{2K}{z(z^{2K} -1)}\, \frac{\sinh(2y)}{\sinh^2(y) + \left(\frac{z - \frac{1}{z}}{2i}\right)^2} = \frac{-8K}{(z^{2K} -1)} \, \frac{\sinh(2y)z}{z^4 - 2z^2(1 + 2\sinh^2(y)) + 1}
\end{equation}
This rational function has poles at the $2K$ roots of unity, and at the four distinct points $\pm e^y, \pm e^{-y}$, which are the roots of 
$$z^4 - 2z^2(1 + 2\sinh^2(y)) + 1 = (z^2-e^{2y})(z^2-e^{-2y})$$ 
(all the poles of $f$ are simple). 

To evaluate the sum, we note that the residue of $f$ at $z=e^{i\pi j/K}$ is now 
$$\frac{\sinh(2y)}{\sin^2(\pi j/K)+\sinh^2(y)}.$$ 
Since the sum of all the residues of $f$ vanishes, to find the required sum we only need to calculate the remaining residues, and multiply their sum by $-1/2$ (since each term in the sum is counted twice). One finds,
$$\Res(f,\pm e^y) =\frac{-2K}{e^{2Ky} - 1},\quad\text{and}\quad 
\Res(f,\pm e^{-y}) =\frac{2K}{e^{-2Ky} - 1}.$$
Summing these four residues and multiplying by $-\frac{1}{2}$ gives us the first sum of the lemma.

Analogously, we construct a rational function $g(z)$ for the second sum::
\[
g(z): = \frac{2K}{z(z^{2N} -1)}\,\frac{\sinh(2y)}{\sinh^2(y) + \left(\frac{z + \frac{1}{z}}{2}\right)^2} = \frac{8K}{(z^{2N} -1)}\,\frac{z\sinh(2y)}{z^4 + 2z^2(1 + 2\sinh^2(y)) + 1}
\]
In this case, we need to compute residues of $g(z)$ at poles located at $\pm i e^y$ and $\pm ie^{-y}$, the roots of $z^4 + 2z^2(1 + 2\sinh^2(y)) + 1$. We have, 
$$\Res(g,\pm ie^y) =\frac{-2K}{(-1)^Ke^{2Ky} - 1},\quad\text{and}\quad 
\Res(g,\pm ie^{-y}) =\frac{2K}{(-1)^Ke^{-2Ky} - 1}.$$
Summing these residues and multiplying by $-\frac{1}{2}$ gives us the second sum of the lemma.
\end{proof}

Suppose $K$ is even; using Lemma \ref{st: the horrible sums},
(\ref{eq: Re ring velocity}) becomes 
\begin{equation}
\label{eq: ring Re final formula velocity}
\begin{array}{rcl}
\dot{x}_1 &=& \frac{\Gamma}{2\pi}\coth(2y) + \frac{\Gamma }{8\pi}\left(4K\coth(Ky)  - \frac{\sinh(2y)}{\sinh^2(y)}  - \frac{\sinh(2y)}{\cosh^2(y)}\right)  \\[4pt]
&=& \frac{\Gamma K}{2\pi} \coth(Ky),
\end{array}
\end{equation}
since in Lemma \ref{st: the horrible sums} we have additionally counted $j=K$. 

Using the same method for odd $K$, we get 
\[
\dot{x}_1 = \frac{\Gamma K}{2\pi}\coth(2Ky).
\]

We act analogously when the two rows are staggered. However, in this case $N$ is an odd number: we suppose that $N = 2K-1$, and that the top row of vortices on the chart on the M\"obius band has $K$ vortices in it.  Using the same notation, we observe that the horizontal distance between the leftmost vortex (which we assume to be above the line $y=0$) and the other vortices in the upper row is $\frac{2\pi j}{2K-1},\ j = 1,\ldots,K-1$. In turn, horizontal distances between the same vortex and the vortices in the bottom row are given by $\frac{(2j-1)\pi}{2K-1},  \ j = 1,\ldots,K-1$. Thus, we have for the velocity:
\[
\dot{x}_1 = \frac{\Gamma}{4\pi}\tanh(y) + \frac{\Gamma\sinh(2y)}{8\pi}\sum_{j=1}^{K-1}\frac{1}{\sinh^2(y) + \sin^2\left(\frac{(2j-1)\pi}{2K-1}\right)} + \frac{1}{\sinh^2(y) + \cos^2\left(\frac{2j\pi}{2K-1}\right)}
\]
We consider the expression $\sum_{j=1}^{K-1}\frac{1}{\sinh^2(y) + \sin^2\left(\frac{(2j-1)\pi}{2K-1}\right)} + \frac{1}{\sinh^2(y) + \cos^2\left(\frac{2j\pi}{2K-1}\right)}$ separately. One can observe that due to the fact that $\sin$ and $\cos$ are squared in the sum,
\begin{equation*}
    \begin{split}
\sum_{j=1}^{K-1}&\left(\frac{1}{\sinh^2(y) + \sin^2\left(\frac{(2j-1)\pi}{2K-1}\right)} + \frac{1}{\sinh^2(y) + \cos^2\left(\frac{2j\pi}{2K-1}\right)}\right) \\ 
 &= \frac{1}{2}\sum_{l=1}^{2K-2} \left( \frac{1}{\sinh^2(y) + \sin^2\left(\frac{l\pi}{2K-1}\right)} + \frac{1}{\sinh^2(y) + \cos^2\left(\frac{l\pi}{2K-1}\right)}\right),
\end{split}
\end{equation*}
which reduces this sum to the one from Lemma \ref{st: the horrible sums}; in particular, the top limit of the sum is always an odd number. Therefore,
\begin{small}
\begin{equation}
\begin{split}
\dot{x}_1 =&\frac{\Gamma \tanh (y)}{4 \pi } + \frac{\Gamma }{16\pi}  \Bigl((4K-2)\coth\left((2K-1)y\right) + (4K-2)\tanh\left((2K-1)y\right) \\
&- \frac{\sinh(2y)}{\sinh^2(y)} - \frac{\sinh(2y)}{\cosh^2(y)}\Bigr) \\
=& \frac{\Gamma}{8\pi}\left(\tanh(y) - \coth(y)\right) + \frac{\Gamma(2K-1)}{4\pi}\coth\left((4K-2)y\right).
\end{split}
\end{equation}
\end{small}
and  recalling that $2K-1 = N$, we get the statement of our theorem.
%\end{proof}

\begin{rem}
It can be explicitly checked that (\ref{eq:xia}) turns into $\frac{\Gamma}{2\pi}\coth(2y)$ when $N=2$: this is Example \ref{st: exam 1}. When $y\to 0$, $\xi_a\sim\frac{\Gamma}{4\pi y}$, consistent with the fact that the angular velocity tends to $+\infty$ when the two aligned rows are infinitely close.

When we substitute $N=1$ in (\ref{eq:xis}), we get the motion of one point vortex: and indeed, (\ref{eq:xis}) readily becomes $\frac{\Gamma}{4\pi}\tanh(y)$. When $y=0$, staggered $N$-rings become fixed equilibria. With $y\to0$ the angular velocity $\xi_s$ is given by
\[
\xi_s = \frac{\Gamma y \left(1 + 2N^2\right)}{12 \pi } + \bar{O}(y^3),
\]
and becomes 0 when $y=0$.
\end{rem}   
\section{A necessary condition for relative equilibria on the Klein bottle}
\label{sec: Appendix B }
In this section we provide the detailed proof of Lemma~\ref{st: lemma crit points of the Hamiltonian Klein}.
Critical points of the reduced Hamiltonian correspond to relative equilibria, which we have established to be horizontally moving configurations. The necessary condition for that is $\dot{y}_1 = \dot{y}_2=0$. 

We use the first part of the equation (\ref{eq: Klein equations of motion}) for illustrative purposes. 

\begin{equation}
    \begin{split}
     \dot{y}_1 &=  -2\mathrm{Re}\biggl( \frac{1}{4\pi}\Gamma_2\frac{\theta'_1\left(\frac{\bar{z}_1 - \bar{z}_2}{2},e^{-\frac{\pi}{2}}  \right)}{\theta_1\left(\frac{\bar{z}_1 - \bar{z}_2}{2},e^{-\frac{\pi}{2}}  \right)}  +\frac{1}{4\pi}\Gamma_2\frac{\theta'_2\left(\frac{\bar{z}_1 - z_2}{2},e^{-\frac{\pi}{2}} \right)}{\theta_2\left(\frac{\bar{z}_1 - z_2}{2},e^{-\frac{\pi}{2}}  \right)} + \frac{1}{4\pi^2}\Gamma_2\left(\bar{z}_2-z_2\right) \\&+ \frac{1}{4\pi}\Gamma_1\Bigl(\frac{\theta_1'\left(\frac{\bar{z}_1-z_1}{2} + \frac{\pi}{2}, e^{-\frac{\pi}{2}}\right)}{\theta_1\left(\frac{\bar{z}_1-z_1}{2} + \frac{\pi}{2}, e^{-\frac{\pi}{2}}\right)} + \frac{\bar{z}_1-z_1}{\pi}\Bigr)\biggr)
    \end{split}
\end{equation}
As we demonstrated in Remark
\ref{st: remark purely imaginal}, the expression $\frac{\theta_1'\left(\frac{\bar{z}_1-z_1}{2} + \frac{\pi}{2}, e^{-\frac{\pi}{2}}\right)}{\theta_1\left(\frac{\bar{z}_1-z_1}{2} + \frac{\pi}{2}, e^{-\frac{\pi}{2}}\right)} $ is purely imaginary (note that due to the second periodicity relation in the second column of (\ref{eq: more properties of thetas}) there is no difference between adding and subtracting $\frac{\pi}{2}$ in the argument ). It is also clear that $\bar{z}_2-z_2$ and $\bar{z}_1-z_1$ are imaginary numbers.

Therefore, the only elements of $\dot{y}_1$ left to investigate are the two first fractions, and our  condition for $\dot{y_1}=0$ reads as
\begin{equation}
\label{eq: crit points of the Hamiltonian Klein}
\mathrm{Re}\left(\Gamma_2\frac{\theta'_1\left(\frac{\bar{z}_1 - \bar{z}_2}{2},e^{-\frac{\pi}{2}}  \right)}{\theta_1\left(\frac{\bar{z}_1 - \bar{z}_2}{2},e^{-\frac{\pi}{2}}  \right)}\right)= \mathrm{Re}\left(\Gamma_2\frac{\theta'_2\left(\frac{\bar{z}_1 - z_2}{2},e^{-\frac{\pi}{2}} \right)}{\theta_2\left(\frac{\bar{z}_1 - z_2}{2},e^{-\frac{\pi}{2}} \right)}\right).
\end{equation}
We set out to prove that a solution only exists when $\mathrm{Re}\left(\bar{z}_1 -\bar{z}_2\right)= \mathrm{Re}\left(\bar{z}_1 -z_2\right)  = k\pi$ for $k\in\mathbb{Z}$. 
\begin{rem}
\begin{enumerate}
    \item Due to the invariant $\Gamma_1y_1 + \Gamma_2y_2=C$, the condition for $\dot{y}_2=0$ is identical to (\ref{eq: crit points of the Hamiltonian Klein}).
    \item Geometrically, this means that we need to demonstrate that the flow generated by a solitary vortex is horizontal only on vertical lines that contain the centre of our vortex and the centres of its copies. This is equivalent to substituting $\Gamma_1=0$ into the first equation of (\ref{eq: crit points of the Hamiltonian Klein}) and then taking its real part. 
\end{enumerate}
\end{rem}
Adopting the geometric interpretation from the Remark above, we investigate the locus of the points of strictly horizontal flow created by the first vortex.

This allows us to  assume without loss of generality that $y_1=0$. Then 
\begin{equation*}
    \begin{split}
        &\frac{\bar{z}_1-\bar{z}_2}{2} = \frac{x_1-x_2}{2} + \frac{i y_2}{2};\\
        &\frac{\bar{z}_1-z_2}{2}= \frac{x_1-x_2}{2} - \frac{i y_2}{2}. 
    \end{split}
\end{equation*}

Let $\frac{x_1-x_2}{2}  = x$ and $\frac{y_2}{2} = y$. Using this notation, we divide  (\ref{eq: crit points of the Hamiltonian Klein}) by $\Gamma_2$ and rewrite it  as
\begin{equation}
    \label{eq: transformation criticl Hamiltonian rewrite}
    \begin{split}
    \mathrm{Re}\left(\frac{\theta'_1\left(x + i y, e^{-\frac{\pi}{2}}\right)}{\theta_1\left(x + i y, e^{-\frac{\pi}{2}}\right)}\right) = \mathrm{Re}\left(\frac{\theta'_2\left(x - i y, e^{-\frac{\pi}{2}}\right)}{\theta_2\left(x - i y, e^{-\frac{\pi}{2}}\right)}\right).
    \end{split}
\end{equation}

We employed one of the following formulae from \cite{wolfram} above, but we repeat them here for convenience:
\begin{equation}
\label{eq: theta' divided by theta}
\begin{split}
    \frac{\theta'_1(z,q)}{\theta_1(z,q)} &=\cot(z) + 4\sum\limits_{n=1}^{\infty}\frac{q^{2n}\sin(2z)}{q^{4n} - 2q^{2n}\cos(2z) + 1},\\
    \frac{\theta'_2(z,q)}{\theta_2(z,q)} &=-\tan(z) - 4\sum\limits_{n=1}^{\infty}\frac{q^{2n}\sin(2z)}{q^{4n} +2q^{2n}\cos(2z) + 1}.
\end{split}
\end{equation}
Using the equalities above to expand (\ref{eq: transformation criticl Hamiltonian rewrite}) turns it into 
\begin{tiny}
\begin{equation}
\label{eq: splitting Ham KLein}
\begin{split}
    \mathrm{Re}\Biggl( \cot(x+iy) + \tan(x-iy) + 4\sum\limits_{n=1}^{\infty} \frac{e^{-\pi n}\sin(2(x+iy)}{e^{-2\pi n } - 2e^{-\pi n }\cos(2(x+iy) + 1} +\frac{e^{-\pi n}\sin(2(x-iy))}{e^{-2\pi n } - 2e^{-\pi n }\cos(2(x-iy)) + 1} \biggr) =0.
\end{split}
\end{equation}
\end{tiny}
 For trigonometric functions of a complex number $x+ iy$ the following hold:
\begin{equation}
    \label{eq: tan and cot complex}
    \begin{split}
        \mathrm{Re}\left(\tan(x + iy)\right) &=\frac{\tan(x)(1-\tanh^2(y))}{1 + \tan^2(x)\tanh^2(y)}\\
        \mathrm{Re}\left(\cot(x + iy)\right) &= \frac{\cot(x)(\coth^2(y)-1)}{\cot^2(x) + \coth^2(y)},\\
         \cos(x + iy) &= \cos(x)\cosh(y) - i\sin(x)\sinh(y),\\
         \sin(x + iy) &= \sin(x)\cosh(y) + i\cos(x)\sinh(y).
    \end{split}
\end{equation}
We consider the expression $\cot(x+iy) + \tan(x-iy)$ and the elements of  the infinite sum separately, starting with the former:
\begin{equation}
    \label{eq: Ham crit points Klein first part}
    \begin{split}
   \mathrm{Re}\bigl(\cot(x+iy) + \tan(x-iy)\bigr) &=\frac{\tan (x) \left(1-\tanh ^2(y)\right)}{\tan ^2(x) \tanh ^2(y)+1}+\frac{\cot (x) \left(\coth ^2(y)-1\right)}{\cot ^2(x)+\coth ^2(y)}\\&=\frac{4 \sin (2 x) \cosh (2 y)}{\cosh (4 y)-\cos (4 x)}
    \end{split}
\end{equation}

For the $n$th element of the sum in (\ref{eq: splitting Ham KLein}) we have 
\begin{equation}
    \label{eq: klen sum second part}
    \begin{split}
       & \mathrm{Re}\left(\frac{e^{-\pi n}\sin(2(x+iy))}{e^{-2\pi n } - 2e^{-\pi n }\cos(2(x+iy)) + 1} +\frac{e^{-\pi n}\sin(2(x-iy))}{e^{-2\pi n } - 2e^{-\pi n }\cos(2(x-iy)) + 1} \right)\\& = \mathrm{Re}\left(\frac{\sin (2 x-2 i y)}{\cosh (\pi  n)+\cos (2 x-2 i y)}+\frac{\sin (2 x+2 i y)}{\cosh (\pi  n)-\cos (2 x+2 i y)}\right)
    \end{split}
\end{equation}
Substituting the last two rows of  (\ref{eq: tan and cot complex}) in (\ref{eq: klen sum second part}) yields
\begin{footnotesize}
\begin{equation}
\label{eq: whatever}
    \begin{split}
    \mathrm{Re}\left(    \frac{4 \cosh (2 y) (\cosh (\pi  n) \sin (x) \cos (x)+i \sinh (y) \cosh (y))}{(\cosh (\pi  n)+\cos (2 x) \cosh (2 y)+i \sin (2 x) \sinh (2 y)) (\cosh (\pi  n)-\cos (2 x) \cosh (2 y)+i \sin (2 x) \sinh (2 y))}\right)
    \end{split}
\end{equation}
\end{footnotesize}
To single out the real part of this expression, we multiply (\ref{eq: whatever}) by the conjugates of the two numbers in the denominator, i.e. by 
\begin{small}
\[
(\cosh (\pi  n)+\cos (2 x) \cosh (2 y)-i \sin (2 x) \sinh (2 y)) (\cosh (\pi  n)-\cos (2 x) \cosh (2 y)-i \sin (2 x) \sinh (2 y)).
\]
\end{small}
This turns the numerator into 
\begin{equation}
    \label{eq: denominator}
    \begin{split}
        &\cosh (\pi  n) \sin (2 x) \cosh (2 y) \left(\cosh (2 \pi  n)-2 \cos ^2(2 x)+\cosh (4 y)\right) \\ &+ \frac{1}{4} i (2 \cosh (2 \pi  n) \cos (4 x) \sinh (4 y)-\sinh (8 y))
    \end{split}
\end{equation}
Its real part is 
\[
\mathrm{Re}(\mathrm{\ref{eq: denominator}}) = \cosh (\pi  n) \sin (2 x) \cosh (2 y) \left(\cosh (2 \pi  n)-2 \cos ^2(2 x)+\cosh (4 y)\right).
\]
For brevity, we denote
\begin{footnotesize}
    \[
\left|(\cosh (\pi  n)+\cos (2 x) \cosh (2 y)+i \sin (2 x) \sinh (2 y)) (\cosh (\pi  n)-\cos (2 x) \cosh (2 y))+i \sin (2 x) \sinh (2 y)\right|^2 = K_n^2,
\]
\end{footnotesize}
which is clearly a non-negative number. 

Combining everything, we obtain the following.
\begin{small}
\begin{equation}
    \begin{split}
       \mathrm{(\ref{eq: splitting Ham KLein})}&=  \frac{4 \sin (2 x) \cosh (2 y)}{\cosh (4 y)-\cos (4 x)} + 2\sum\limits_{n=1}^{\infty}\frac{\sin (2 x)\cosh (\pi  n)  \cosh (2 y) \left(\cosh (2 \pi  n)-2 \cos ^2(2 x)+\cosh (4 y)\right)}{K^2_n}\\&=2\sin(2x)\left(\frac{2 \cosh (2 y)}{\cosh (4 y)-\cos (4 x)} + \sum\limits_{n=1}^{\infty}\frac{\cosh (\pi  n)  \cosh (2 y) \left(\cosh (2 \pi  n)-2 \cos ^2(2 x)+\cosh (4 y)\right)}{K_n^2}\right).
    \end{split}
\end{equation}
\end{small}
The expression above is $\sin(2x)$ multiplied by a positive number - therefore, it is only equal to 0 when $\sin(2x) = \sin\left(x_1-x_2\right)$ is, forcing $x_1-x_2$ to be a multiple of $\pi$, which proves our statement.

\setlength{\parindent}{0pt}
\small
\hrulefill

\bigskip 

{N. Balabanova} \\
School of Mathematics, \\
University of Birmingham\\ Birmingham, B15 2TT, UK\\
\texttt{n.balabanova@bham.ac.uk}\\

{J.~Montaldi}\\
Department of Mathematics\\
University of Manchester\\
Manchester M13 9PL, UK\\
\texttt{j.montaldi@manchester.ac.uk}

\end{document}